\documentclass[a4paper,12pt]{amsart}
\usepackage{mathrsfs}
\usepackage{amsmath,amssymb,latexsym,amsfonts,amscd}

\title{Explicit birational geometry of threefolds of general type}
\author{Jungkai A. Chen and Meng Chen}
\address{\rm Taida Institute for Mathematical Sciences,
National Center for Theoretical Sciences, Taipei Office, and
Department of Mathematics, National Taiwan University, Taipei,
106, Taiwan} \email{jkchen@math.ntu.edu.tw}

\address{\rm School of Mathematical Sciences, Fudan University,
Shanghai, 200433, People's Republic of China}
\email{mchen@fudan.edu.cn}

\thanks{The first author was partially supported by TIMS, NCTS/TPE
and  National Science Council of Taiwan. The second author was
supported by both the Program for New Century Excellent Talents in
University (\#NCET-05-0358) and the National Outstanding Young
Scientist Foundation (\#10625103)}
\newcommand{\bC}{{\mathbb C}}
\newcommand{\bQ}{{\mathbb Q}}
\newcommand{\bP}{{\mathbb P}}
\newcommand{\roundup}[1]{\lceil{#1}\rceil}
\newcommand{\rounddown}[1]{\lfloor{#1}\rfloor}

\newcommand\Vol{\text{\rm Vol}}

\newcommand\OO{{\mathcal{O}}}

\newtheorem{thm}{Theorem}[section]
\newtheorem{lem}[thm]{Lemma}
\newtheorem{cor}[thm]{Corollary}
\newtheorem{prop}[thm]{Proposition}

\theoremstyle{definition}
\newtheorem{defn}[thm]{Definition}
\newtheorem{setup}[thm]{}

\newtheorem{exmp}[thm]{Example}

\newtheorem{rem}[thm]{Remark}
\theoremstyle{remark}

\newtheorem{Prbm}{\bf Problem}
\begin{document}
\begin{abstract}
Let $V$ be a complex nonsingular projective 3-fold  of general type.
We prove $P_{12}(V)>0$ and $P_{24}(V)>1$. We also prove that the
canonical volume has an universal lower bound $\Vol(V) \geq 1/2660$
and that the pluri-canonical map $\varphi_m$ is birational onto its
image for all $m\geq 77$. As an application of our method, we prove
Fletcher's conjecture on weighted hyper-surface 3-folds with
terminal quotient singularities. Another featured result is the
optimal lower bound $\Vol(V)\geq \frac{1}{420}$ among all those
3-folds $V$ with $\chi(\OO_V)\leq 1$.

\end{abstract}
\maketitle
\pagestyle{myheadings} \markboth{\hfill J. A. Chen and M. Chen
\hfill}{\hfill Explicit birational geometry of threefolds\hfill}

\tableofcontents

{\bf Notations}

\begin{tabbing}
 \= aaaaaaaaaaaaaaaaaaaa\= bbbbbbbbbbbbbbbbbbbbbbbbbbbbbbb \kill
\> $Y$    \> a nonsingular projective variety of general type\\
\> $V$  \> a nonsingular projective 3-fold of general type\\
\>$X$ \> a minimal projective 3-fold of general type\\
\> $\Vol(V)$, $K^3$ \> the canonical volume\\
\> $\varphi_m=\Phi_{|mK|}$ \> pluricanonical maps\\
\> $P_m(V)$, $P_{m}(X)$ \> plurigenus of $V$, $X$\\
\> $\pi:X'\rightarrow X$ \> nonsingular birational modification\\
\> $m_1$ \> minimal positive integer with $P_m>0$ for\\
\> \>\ \ all $m\geq m_0$\\
\> B in Sections 2, 3 \> the base of the induced fibration
from $\varphi_{m_0}$\\
\>B, in Sections 4 $\sim$ 7 \> a basket, usually a geometric
basket\\
\> $\rounddown{\cdot}$ \> round-down, taking the integral part\\
\>$\roundup{\cdot}$ \> means $-\rounddown{-\cdot}$\\
\end{tabbing}

\section{\bf Introduction}
Let $Y$ be a non-singular projective variety of dimension $n$. It
is said to be of general type if the pluricanonical map
$\varphi_m:=\Phi_{|mK_Y|}$ corresponding to the linear system
$|mK_Y|$ is birational into a projective space for $m \gg 0$. It is
thus natural and important to ask:
\begin{Prbm}
Does there exist a constant $c(n)$,  so that $\varphi_{m}$ is
birational for all $m \ge c(n)$ and for all $Y$ with dimension $n$?
\end{Prbm}

When $\dim Y=1$, it was classically known that $|mK_Y|$ gives an
embedding of $Y$ into a projective space if $m \ge 3$. When $\dim
Y=2$, Bombieri's theorem \cite{Bom} says that $|mK_Y|$ gives a
birational map onto the image for $m\geq 5$. This theorem has forever
established the canonical classification theory for nonsingular
projective surfaces of general type.

A possible approach to this problem in any dimension is to use the cohomological method via vanishing
theorems. This amounts to estimating the positivity of $K_Y$, which
is usually measured by the canonical volume
$$\Vol (Y):=\limsup_{\{m\in {\mathbb Z}^{+}\}}
\{\frac{n!}{m^n}\dim_{\bC}H^0(Y, \OO_Y(mK_Y))\}.$$ The volume is an
integer when $\dim Y \le 2$. However it's only a rational number in
higher dimensions. In fact, it is almost an equivalent question to
study the lower bound of the canonical volume.

\begin{Prbm} Does there exist a constant $c'(n)$ such that $\Vol(Y) \geq
c'(n)$ for all varieties $Y$ of general type with dimension $n$?
\end{Prbm}

Notice that a recent remarkable result of Hacon and
M$^{\text{c}}$Kernan \cite{H-M}, Takayama \cite{Ta} and Tsuji
\cite{Tsuji} implies the affirmative answer to both Problems.
However, they did not give explicit numerical bounds or the bound
could only be far from realistic.

 We would like to prove some explicit bounds for $c(3)$ and
$c'(3)$ for 3-folds $Y$
of general type in this paper.

\begin{thm}[=Theorem \ref{birat}] Let $Y$ be a nonsingular projective
3-fold of general type. Then $\varphi_{m}$ is birational for $m \ge
77$.
\end{thm}

\begin{thm}[=Theorem \ref{vol}] Let $Y$ be a nonsingular projective
3-fold of general type. Then $\Vol(Y) \geq \frac{1}{2660}$.
\end{thm}

Yet another approach for 3-folds is to study $\varphi_{m_0}$ for
some positive integer $m_0$ with $\varphi_{m_0}$ non-constant. This
program was first proposed by  Koll\'ar \cite{Kol}, and then
improved by the second author.

\begin{thm}\label{56}\cite[Theorem 0.1]{JPAA} Let $Y$ be a nonsingular projective
3-fold of general type. If $P_{m_0} \ge 2$, then $\varphi_{m}$ is
birational for all $m \ge 5m_0+6$.
\end{thm}

Therefore, it is of fundamental importance to know the non-vanishing
of plurigenera. In fact,  Koll\'ar and  Mori proposed some related
problems (see, for example, the last question of 7.74 in \cite{K-M})
including the following one:
\begin{Prbm}
Does there exist a constant $c''(n)$ such that $P_{m} \ge 2 $ (or
$\geq 1$) for some $m \le c''(n)$ for all nonsingular projective
varieties $Y$ of general type with dimension $n$?
\end{Prbm}

We are able to answer these questions for 3-folds.

\begin{thm}[=Theorems \ref{p12}, \ref{p24}] Let $Y$ be a nonsingular projective
3-fold of general type. Then $ P_{12} \ge 1$ and $P_{24} \ge 2$.
\end{thm}

An interesting application of our method is that we are able to
prove Iano-Fletcher's conjecture (see \cite[15.1, 15.2]{C-R}) as the
following:

\begin{thm}\label{cod=1}
There are exactly 23 families of quasi-smooth weighted hyper-surface
3-folds $X$ with only terminal quotient singularities and
$\omega_X\cong \OO_X(1)$.
\end{thm}

We now explain the idea for the proofs. The key new ingredient is,
in some sense,  the {\it classification of Reid's baskets of
singularities}. Recall that for a 3-fold with canonical
singularities, Reid \cite{C-3f} introduced the notion of baskets of
singularities to compute the plurigenera. The upshot is that given a
minimal 3-fold $X$, Reid's ``virtue'' baskets are uniquely
determined by $X$. Thus to determine those baskets is a very
important step.

We will introduce the notion of {\it packing of baskets} on certain
given 3-fold, a  partial ordering between baskets, which allows us
to study baskets in a systematic way. In fact, there is a more
refined framework which tells that for any $m \ge 3$, the set of
baskets with given datum $(\chi,P_2,P_3,...,P_m)$ is finite.

We have discovered a key and new inequality (see inequality (5.3)):
$$2P_5+3P_6+P_8+P_{10}+P_{12} \ge \chi +10
P_2+4P_3+P_7+P_{11}+P_{13}+R,$$ with $R \geq 0$. Therefore if
$P_{m_0} \ge 2$ for some $m_0 \le 12$, then one can study
$\varphi_{m_0}$ and get effective results by Theorem \ref{56}. When
$P_{m} \leq 1$ for all $m \le 12$, the above inequality shows that
$\chi \leq 8$. So the set of baskets with these plurigenera are
finite. It is thus possible for us to classify those baskets
completely, which is basically what we did.

This article is organized as the following. In Section 2, we set up
some notations and generalities for the study of $\varphi_m$. In
Section 3, we study $\Vol$ and $P_m$ when $P_{m_0} \ge 2$ for some
$m_0>0$ using the technique developed in Section 2. The main new
ingredient starts from Section 4. We introduce the notion of packing
in Section 4. We also describe the structure between baskets by
using packings. Section 5 contains the description of baskets with
given datum $(\chi,P_2,P_3,...,P_m)$. We remark that it gives rise
to various inequalities.

Section 6 is the classification of baskets with $\chi=1$ and $P_m
\le 1$ for $m \le 6$. Together with the result in Section 3, we
prove that $\Vol \geq \frac{1}{420}$, \footnote{The authors were
informed by Lei Zhu that she obtained the same lower bound
independently.} which is sharp. Section 7 presents the list of
classification of baskets with $2 \le \chi \le 8$ and $P_m \le 1$
for $m \le 12$. Similarly, we get $\Vol \ge \frac{1}{2660}$ for
general 3-folds. With all these preparations, we prove our main
theorems in Section 8, including plurigenera and the birationality
for all 3-folds of general type. This is possible because for a
3-fold of general type, either $P_{m_0} \ge 2$ for some $m_0 \le
12$, or it's classified in Sections 6 or 7.  As a direct
application, we prove in Section 9 a conjecture of Iano-Fletcher
regarding hyper-surface 3-folds in weighted projective spaces.

There are some more applications of the techniques developed here that we will pursue in a future work.

Throughout, we work over the complex number field
$\bC$. We prefer to use $\sim$ to denote the linear equivalence
and $\equiv$ means numerical equivalence.
\bigskip

\noindent {\bf Acknowledgments.} We are indebted to Jiun-Cheng Chen,
Christopher Hacon, J\'anos Koll\'ar, Hui-Wen Lin and Chin-Lung Wang
for useful conversations and comments on this subject. We would like
to thank Pei-Yu Tsai and Hou-Yi Chen for helping us verifying datum.
After the first version was finished, we were kindly informed of Tie
Luo's paper \cite{Luo}, where there had been already considerable
calculations for a special case.

\section{\bf Technical preparations}
\begin{defn}\label{rat-map} Let $L$  be a divisor on a
nonsingular projective variety $Y$ with $n_L:=h^0(Y,\OO_Y(L))-1\geq
1$. Pick a basis $s_0,\cdots, s_{n_L}\in H^0(Y,\OO_Y(L))$. For any
point $x\in Y$, we define a rational map
$\Phi_{|L|}:Y\dashrightarrow \bP^{n_L}$ by sending $x$ to $[s_0(x),
\cdots, s_{n_L}(x)]$. $\Phi_{|L|}$ is usually said to be {\it the
rational map associated to} $|L|$.
\end{defn}

First of all we list the birationality principles which we will
frequently use in our arguments:

\begin{setup}\label{BP}{\bf Birationality principles.} Let $Y$ be a
nonsingular projective variety on which there are two divisors $D$
and $M$. Suppose $|M|$ is base point free. Take the Stein
factorization of $\Phi_{|M|}$:
$$Y\overset{f}\longrightarrow W\longrightarrow \bP^{h^0(Y,M)-1}$$
where $f$ is a fibration onto a normal variety $W$. Then the
rational map $\Phi_{|D+M|}$ is birational onto its image if one of
the following conditions is satisfied:
\begin{itemize}
\item [(i)] (\cite[Lemma 2]{T}) $\dim\Phi_{|M|}(Y)\geq 2$,
$|D|\neq \emptyset$ and $\Phi_{|D+M|}|_S$ is birational for a
general member $S$ of $|M|$.

\item [(ii)] (\cite[\S2.1]{MPCPS})
$\dim\Phi_{|M|}(Y)=1$, $\Phi_{|D+M|}$ can separate different general
fibers of $f$ and $\Phi_{|D+M|}|_F$ is birational for a general
fiber $F$ of $f$.
\end{itemize}
\end{setup}

\begin{rem}\label{separate} For the condition \ref{BP} (ii), one knows
that $\Phi_{|D+M|}$ can separate different general fibers of $f$
whenever $\dim\Phi_{|M|}(Y)=1$, $W$ is a rational curve and $D$ is
an effective divisor. In fact, since $|M|$ can separate different
fibers of $f$, so can $|D+M|$. We do not care too much about the
situation when $W$ is an irrational curve, since the results and
technique in \cite{JC-H} are sufficient for our purpose here.
\end{rem}

\begin{setup}\label{inv}{\bf Invariants of the fibration.} Let $V$
be a smooth projective 3-fold and $f:V\longrightarrow B$ a fibration
onto a nonsingular curve $B$. There is a spectral sequence,
$$E_2^{p,q}:=H^p(B,R^qf_*\omega_V)\Longrightarrow
E^n:=H^n(V,\omega_V).$$  By Serre duality and  \cite[Corollary 3.2,
Proposition 7.6]{Kol}, one has the torsion-freeness of the sheaves
$R^i f_* \omega_V$ and the following formulae:
$$h^2({\mathcal O}_V)=h^1(B,f_*\omega_V)+h^0(B,R^1f_*\omega_V),$$
$$q(V):=h^1({\mathcal O}_V)=g(B)+h^1(B,R^1f_*\omega_V).$$
\end{setup}

\begin{setup}\label{reduc}{\bf Reduction step}. Let $V$ be a
nonsingular projective 3-fold of general type. By the 3-dimensional
MMP (see for instance \cite{KMM,K-M,Reid83}), we can pick a minimal
model $X$ of $V$ and allow $X$ to have at worst $\bQ$-factorial
terminal singularities. Denote by $K_X$ a canonical divisor of $X$.
We recall the following birational invariants.
\begin{eqnarray*}
P_m(V)&:=&h^0(V,\OO_V(mK_V))=h^0(X,\OO_X(mK_X))=:P_m(X);\\
\chi(\OO_V)&=&\chi(\OO_X);\\
\Vol(V)&:=&\limsup \frac{3!}{m^3}h^0(V,mK_V)=K_X^3.
\end{eqnarray*}
Note  that the rational maps $\Phi_{|mK_V|}$ and
$\varphi_m:=\Phi_{|mK_X|}$ are birationally equivalent.
Sometimes we simply
denote by $K^3$ the canonical volume of $V$ and $X$.

{}From this point of view, it suffices to prove our main theorem
only for minimal 3-folds $X$.
\end{setup}
\medskip

($\heartsuit$) Throughout, $X$ will be an arbitrary minimal 3-fold
of general type with at worst $\bQ$-factorial terminal
singularities. The integer $m_0$ always denotes a positive (most
likely, minimal) integer with
$$P_{m_0}=P_{m_0}(X):=\dim_{\bC} H^0(X,\OO_X(m_0K_X))\geq 2$$
where $K_X$ is a canonical divisor of $X$. By minimal, we mean that
$K_X$ is nef. Define $r(X)$ to be the minimal positive integer such
that $r(X)K_X$ is Cartier. It is already known that $r(X)$ is
uniquely determined by the birational equivalence class of $X$. So
it is a birational invariant within the category of 3-folds having
at worst canonical singularities.


\begin{setup}\label{setup}{\bf Set up for $\varphi_{m_0}$.} We
study the $m_0$-canonical map of $X$:
$$\varphi_{m_0}:X\dashrightarrow \bP^{P_{m_0}-1}$$ which is
only a rational map. First of all we fix an effective Weil divisor
$K_{m_0}\sim m_0K_X$. By Hironaka's big theorem, we can take
successive blow-ups $\pi: X'\rightarrow X$ such that:
\begin{itemize}
\item [(i)] $X'$ is smooth; \item [(ii)] the movable part of
$|m_0K_{X'}|$ is base point free; \item [(iii)] the support of the
union of $\pi^*(K_{m_0})$ and the exceptional divisors is of
simple normal crossings.
\end{itemize}

Set $g_{m_0}:=\varphi_{m_0}\circ\pi$. Then $g_{m_0}$ is a morphism
by assumption. Let $X'\overset{f}\longrightarrow
B\overset{s}\longrightarrow W'$ be the Stein factorization of
$g_{m_0}$ with $W'$ the image of $X'$ through $g_{m_0}$. In
summary, we have the following commutative diagram:\medskip

\begin{picture}(50,80) \put(100,0){$X$} \put(100,60){$X'$}
\put(170,0){$W'$} \put(170,60){$B$} \put(112,65){\vector(1,0){53}}
\put(106,55){\vector(0,-1){41}} \put(175,55){\vector(0,-1){43}}
\put(114,58){\vector(1,-1){49}} \multiput(112,2.6)(5,0){11}{-}
\put(162,5){\vector(1,0){4}} \put(133,70){$f$} \put(180,30){$s$}
\put(95,30){$\pi$}
\put(130,-5){$\varphi_{m_0}$}\put(136,40){$g_{m_0}$}
\end{picture}
\bigskip

Let us recall the definition of $\pi^*(K_X)$. We can write
$r(X)K_{X'}=\pi^*(r(X)K_X)+E_{\pi}$ where $E_{\pi}$ is a sum of
exceptional divisors. We define
$$\pi^*(K_X):=K_{X'}-\frac{1}{r(X)}E_{\pi}.$$ So, whenever we take
the round-up of $m\pi^*(K_X)$, we always have
$$\roundup{m\pi^*(K_X)}\leq mK_{X'}$$ for any integer $m>0$. Denote by
$M_{m_0}$ the movable part of $|m_0K_{X'}|$. One has
$$m_0\pi^*(K_X)=M_{m_0}+E_{m_0}'$$
for an effective ${\mathbb Q}$-divisor $E_{m_0}'$ because
$$h^0(X',\rounddown{m_0\pi^*(K_X)})=h^0(X',\roundup{m_0\pi^*(K_X)})=P_{m_0}(X')=P_{m_0}(X).$$
On the other hand, one can write $m_0K_{X'}=
m_0\pi^*(K_X)+E_{m_0}$ where $E_{m_0}$ is an effective ${\mathbb
Q}$-divisor as a ${\mathbb Q}$-sum of distinct exceptional
divisors. Thus $m_0K_{X'}=M_{m_0}+Z_{m_0}$ where
$Z_{m_0}:=E_{m_0}'+E_{m_0}$ is exactly the fixed part of
$|m_0K_{X'}|$.

If $\dim(B)\geq 2$, a general member $S$ of $|M_{m_0}|$ is a
nonsingular projective surface of general type by Bertini's
theorem and by the easy addition formula for Kodaira dimension.

If $\dim(B)=1$, a general fiber $S$ of $f$ is an irreducible
smooth projective surface of general type, still by the easy
addition formula for Kodaira dimension. We may write
$$M_{m_0}=\underset{i=1}{\overset{a_{m_0}}\sum}S_i\equiv
a_{m_0}S$$ where the $S_i$ is a smooth fiber of $f$ for all $i$ and
$a_{m_0}\ge P_{m_0}(X)-1$ by considering the degree of a
non-degenerate curve in a projective space.

Define the positive integer
$$p=\begin{cases}   1 &\text{if}\ \dim (B)\geq 2\\
 a_{m_0} &\text{if}\ \dim(B)=1.
 \end{cases}$$
\begin{defn} In both cases regarding to $\dim(B)$, we call $S$ {\it a generic irreducible element
of} the linear system $|M_{m_0}|$. Denote by $\sigma:
S\longrightarrow S_0$ the blow-down onto the smooth minimal model
$S_0$.

By abuse of concepts we define {\it a generic irreducible element
of any given movable linear system} on any projective variety in a
similar way.
\end{defn}

\begin{setup}\label{type} { \bf Type of $f$}. To simplify our
statements, we say that the fibration $f$ induced from
$\varphi_{m_0}$ is of type III ( resp. II, I) if $ \dim B=3$ (resp.
$2,1$). In the case of type I, we distinguish into subcases as the
following:
$$
\begin{cases} I_q & \text{if}\  g(B)>0,\\
            I_3 & \text{if} \ P_{m_0} \ge 3,\\
            I_p& \text{if}\  g(B)=0,\ p_g(S) >0,\\
            I_n& \text{if}\  g(B)=0,\ p_g(S) =0.\\
\end{cases} $$
\end{setup}

\end{setup}

\begin{setup}\label{assumptions}{\bf Assumptions}. Keep the same
setup as in \ref{setup}. Let $m$ be a positive integer. We need some
assumptions to estimate $K^3$ and to study $\varphi_m$.
\begin{itemize}
\item [(1)] Take a generic irreducible element $S$ of $|M_{m_0}|$.
Assume
 that there is a base point free complete linear system $|G|$ on $S$.
Denote by $C$ a generic irreducible element of $|G|$. \item [(2)]
Assume there is a rational number $\beta>0$ such that
$\pi^*(K_X)|_S-\beta C$ is numerically equivalent to an effective
$\bQ$-divisor on $S$. \item [(3)] The linear system $|mK_{X'}|$
separates different generic irreducible elements of $|M_{m_0}|$
(namely, $\Phi_{|mK_{X'}|}(S')\neq \Phi_{|mK_{X'}|}(S'')$ for two
different irreducible elements $S'$, $S''$ of $|M_{m_0}|$). \item
[(4)] The linear system $|mK_{X'}|_{|S}$ on $S$ (as a sub-linear
system of $|{mK_{X'}}_{|S}|$) separates different generic
irreducible elements of $|G|$. Or sufficiently, the complete linear
system
$$|K_{S}
+\roundup{(m-1)\pi^*(K_X)-S-\frac{1}{p}E_{m_0}'}_{|S}|$$ separates
different generic irreducible elements of $|G|$.
\end{itemize}
\medskip

Set the following quantities:
\begin{eqnarray*}
&&\xi:=(\pi^*(K_X)\cdot C)_{X'};\\
&&\alpha:=(m-1-\frac{m_0}{p}-\frac{1}{\beta})\xi;\\
&&\alpha_0:=\roundup{\alpha}.
\end{eqnarray*}

Under Assumptions \ref{assumptions} (1), (2), clearly one has
$$K^3\geq \frac{p}{m_0}\pi^*(K_X)^2\cdot S\geq
\frac{p\beta}{m_0}(\pi^*(K_X)\cdot C)=\frac{p \beta}{m_0} \xi.
\eqno{(2.1)}$$

So it suffices to estimate the rational number
$\xi:=(\pi^*(K_X)\cdot C)_{X'}$ in order to obtain the lower bound
of $K^3$.
\end{setup}

\begin{setup} Suppose that $P_{m_0} \ge 2$ and Assumption \ref{assumptions}(1),
\ref{assumptions}(2) hold. Let $m$ be an integer such that $m >
1+\frac{m_0}{p}+ \frac{1}{\beta}$. Let $$\mathcal{
L}_m:=(m-1)\pi^*(K_X)-\frac{1}{p}E_{m_0}'  $$ be a $\mathbb{
Q}$-divisor on $X'$.  Clearly, we have
$$|K_{X'}+\roundup{\mathcal{L}_{m}}|\subset |mK_{X'}|.$$
Noting that $\mathcal{L}_m-S\equiv (m-1-\frac{m_0}{p})\pi^*(K_X)$
is nef and big, then the  Kawamata-Viehweg vanishing theorem
(\cite{KaV,V}) yields the surjective map
$$H^0(X', K_{X'}+\roundup{ \mathcal{L}_m}) \to H^0(S,
(K_{X'}+\roundup{ \mathcal{L}_m })_{|S}). \eqno(2.2)$$ Since $S$ is
a generic irreducible element of a free linear system, one has
$\roundup{*}|_S \geq \roundup{*_{|S} } $ for any divisor $*$. It
follows that
$$(K_{X'}+\roundup{ \mathcal{L}_m })_{|S} \geq {K_{X'}}_{|S}+\roundup{
{\mathcal{L}_m}_{|S} }=K_S+\roundup{(\mathcal{L}_m-S)_{|S}}.
\eqno(2.3)$$

By Assumption \ref{assumptions}(2), there is an effective ${\mathbb
Q}$-divisor $H$ on $S$ such that
$\frac{1}{\beta}\pi^*(K_X)|_{S}\equiv C+H$. We now consider
$$\mathcal{D}_m:= (\mathcal{L}_m -S)_{|S}-H$$ on $S$.
The divisor $\mathcal{D}_m-C\equiv (m-1-
\frac{m_0}{p}-\frac{1}{\beta})\pi^*(K_X)|_{S}$ is nef and big.
Thus the Kawamata-Viehweg vanishing theorem again gives the
following surjective map
$$H^0(S,
K_{S}+\roundup{\mathcal{D}_m}) \longrightarrow  H^0(C, K_C + D),
\eqno (2.4)$$ where $D:=\roundup{\mathcal{D}_m-C}|_C$ is a divisor
on $C$. Because $C$ is a generic irreducible element of a free
linear system, we have $D \ge \roundup{(\mathcal{D}_m-C)_{|C}}$
similarly. A simple calculation gives $$ \deg(D) \ge
(\mathcal{D}_m-C)\cdot C = (m-1-
\frac{m_0}{p}-\frac{1}{\beta})\xi= \alpha \eqno(2.5)
 $$
\end{setup}

\begin{prop} \label{easy-nv}
Keep the notation as above. Then $P_m(X) \ge 2$ for all integer $m
> 1+\frac{m_0}{p}+ \frac{1}{\beta}$.
\end{prop}

\begin{proof}
This is clear from the inclusion
$|K_{X'}+\roundup{\mathcal{L}_{m}}|\subset |mK_{X'}|$, surjections
$(2.2)$ and $(2.4)$, together with $(2.3)$, $(2.5)$ and the fact
that $g(C) \ge 2$.
\end{proof}

\begin{thm}\label{technical}  Let $m>0$ be an integer satisfying
Assumptions \ref{assumptions}(1), \ref{assumptions}(2). The
inequality
$$\xi\geq \frac{\deg(K_C)+\alpha_0}{m}$$
holds if one of the following conditions is satisfied:
\begin{itemize}
\item [(i)] $\alpha>1$; \item [(ii)] $\alpha>0$ and $C$ is an even
divisor on $S$.
\end{itemize}
\medskip

Furthermore if Assumptions \ref{assumptions} (1) through (4) are
satisfied. The map $\varphi_m$ is birational onto its image when
one of the following conditions is satisfied:
\begin{itemize}
\item [(i)] $\alpha > 2$;
 \item [(ii)] $\alpha \geq 2$ and $C$ is
not a hyper-elliptic curve on $S$.
\end{itemize}
\end{thm}

\begin{rem} \label{weak} In particular
the inequality $\xi\geq \frac{\deg(K_C)+\alpha_0}{m}$ in Theorem
\ref{technical} implies $ \xi \ge \frac{\deg(K_C)}{1+
\frac{m_0}{p}+\frac{1}{\beta}}$ since, whenever $m$ is big enough so
that $\alpha>1$,
$$ m \xi \ge \deg(K_C)+ \alpha_0 \ge \deg(K_C)+\alpha=\deg(K_C)+
(m-1-\frac{m_0}{p}-\frac{1}{\beta}) \xi.$$
\end{rem}

\begin{proof}[{\bf Proof of Theorem \ref{technical}}]
We keep the notation as above. Denote by $|M_m|$ the movable part
of $|mK_{X'}|$ and by $|M'_m|$ the movable part of
$|K_{X'}+\roundup{\mathcal{L}_m}|$. Clearly, one has $M_m \ge M'_m$
by definition.

Let $|N_m|$ be the movable part of
$|(K_{X'}+\roundup{\mathcal{L}_m})_{|S}|$. Applying Lemma 2.7 of
\cite{MPCPS} to the surjective map (2.2), we have $$M'_m|_{S}\ge
N_m.$$

We now claim that $|K_C+D|$ is base point free. To see this, if
$\alpha>1$, then $\deg(D)\geq \alpha_0\geq 2$. Thus $|K_C+D|$ is
free.  If $C$ is an even divisor, then $\deg(D)$ is even since
$\roundup{\mathcal{D}_m-C}$ is an integral divisor on $S$. If
moreover $\alpha >0$, then $ \deg(D) \ge 2$.
 Therefore $|K_C+D|$ is base point free. Hence in both cases, the movable part of $|K_C+D|$
is itself.

Let $|N'_m|$ be the movable part of
$|K_{S}+\roundup{\mathcal{D}_m}|$.   Applying Lemma 2.7 of
\cite{MPCPS} to surjective map (2.4), we have $$ N'_m|_C \ge
K_C+D.$$ Note that  $N_m \ge N'_m$. Combining all these gives
$$ m \pi^* K_X \cdot C \ge (N'_m \cdot C)_S \ge 2g(C)-2 +
\deg(D).$$ Therefore, $$\xi\geq \frac{\deg(K_C)+\alpha_0}{m}.$$

  Next we prove the birationality. Assumption
\ref{assumptions}.(3) says that $|mK_{X'}|$ can separate different
irreducible elements of $|M_{m_0}|$. The birationality principle
\ref{BP} permits us only to verify the birationality of
$|mK_{X'}|_{|S}$ on a generic irreducible element $S$ of
$|M_{m_0}|$.

By Assumptions \ref{assumptions}.(1) and \ref{assumptions}.(3),
there is a base point free linear system $|G|$ on $S$ and the linear
system $|mK_{X'}|_{|S}$ on $S$ separates different generic
irreducible elements of $|G|$. The birationality principle \ref{BP}
reduces the problem to verify the birationality of
$(|mK_{X'}|_{|S})_{|C}$ on a generic irreducible element $C$ of
$|G|$. In fact we will prove this for a sub-linear system of
$(|mK_{X'}|_{|S})_{|C}$.


{} From the above discussion, we only need to verify that
$|K_C+D|$ gives a birational map onto the image of $C$. This is
the case whenever either $\deg(D)\geq 3$ or $C$ is
non-hyperelliptic and $\deg(D)\geq 2$. This completes the proof.
\end{proof}

The following lemma has already appeared in a couple of unpublished
preprints of the second author. In order to make this paper more
self-contained we would like to collect the proof here. In fact, the
special case that $p_g \ge 2$ has been published in \cite[Lemma
3.7]{Chen-Zhang}.

\begin{lem}\label{beta} Keep the same notation as in \ref{setup}, \ref{assumptions} and
Theorem \ref{technical}. Assume $B=\mathbb{P}^1$. Let
$f:X'\longrightarrow \mathbb{P}^1$ be an induced fibration of
$\varphi_{m_0}$. Then one can find a sequence of rational numbers
$\{\beta_n\}$ with $\underset{n\mapsto +\infty}\lim \beta_n =
\frac{p}{m_0+p}$ such that $\pi^*(K_X)|_S-\beta_n\sigma^*(K_{S_0})$
is ${\bQ}$-linearly equivalent to an effective ${\bQ}$-divisor
$H_n$.
\end{lem}

\begin{proof} We use Koll\'ar's technique in \cite{Kol}.
One has $\mathcal {O}_{B}(p)\hookrightarrow {f}_*\omega_{X'}^{m_0}$.
The inclusion relation between divisors gives the inclusion of
sheaves:
$${f}_*\omega_{X'/B}^{t_0p}\hookrightarrow
{f}_*\omega_{X'}^{t_0p+2t_0m_0}$$ for any big integer $t_0$.

For any positive integer $k$, we know in \ref{setup} that $M_k$
denotes the movable part of $|kK_{X'}|$. Note that
${f}_*\omega_{X'/B}^{t_0p}$ is generated by global sections since it
is semi-positive according to Viehweg (\cite{VV}). So any local
section of ${f}_*\omega_{X'/B}^{t_0p}$ can be extended to be a
global one. On the other hand, $|t_0p\sigma^*(K_{S_0})|$ is base
point free and is exactly the movable part of $|t_0pK_S|$ by
Bombieri \cite{Bom} or Reider \cite{Reider}. Clearly one has the
following relation:
$$(t_0p+2t_0m_0)\pi^*(K_X)|_S\geq M_{t_0p+2t_0m_0}|_S\geq t_0p\sigma^*(K_{S_0}).$$
Set $a_0:=t_0p+2t_0m_0$ and $b_0:=t_0p$. Then there is an
effective $\mathbb{Q}$-divisor $I_0'$ on $S$ such that
$$a_0\pi^*(K_X)|_S=_{\bQ} b_0\sigma^*(K_{S_0})+I_0'.$$
Thus $\pi^*(K_X)|_S =_{\bQ} \frac{b_0}{a_0}\sigma^*(K_{S_0})+I_0$
with $I_0=\frac{1}{a_0}I_0'$ still an effective $\bQ$-divisor.
\medskip

{\bf Case 1}. First we consider the case $p\ge 2$.

We use an induction on the basis of the numbers $a_0$ and $b_0$.
Suppose that we have defined $a_l$ and $b_l$ such that the
following is satisfied with $l=n$:
$$a_{n}\pi^*(K_X)|_S \geq b_{n}\sigma^*(K_{S_0}).$$
We will define $a_{l+1}$ and $b_{l+1}$ inductively such that the
above inequality is satisfied with $l = n+1$. By assumption we
know that $a_n\pi^*(K_X)$ supports on a divisor with normal
crossings. Then the Kawamata-Viehweg vanishing theorem implies the
surjective map:
$$H^0(K_{X'}+\roundup{a_n\pi^*(K_X)}+S)\longrightarrow H^0(S, K_S+
\roundup{a_n\pi^*(K_X)}|_S).$$ One sees the following relations:
\begin{eqnarray*}
|K_{X'}+\roundup{a_n\pi^*(K_X)}+S||_S&=&|K_S+\roundup{a_n\pi^*(K_X)}|_S|\\
&\supset& |K_S+b_n\sigma^*(K_{S_0})|\\
&\supset& |(b_n+1)\sigma^*(K_{S_0})|.
\end{eqnarray*}
Denote by $M_{a_n+1}'$ the movable part of $|(a_n+1)K_{X'}+S|$.
Applying Lemma 2.7 of \cite{MPCPS}, one gets $M_{a_n+1}'|_S\geq
(b_n+1)\sigma^*(K_{S_0}).$ Re-modifying our original $\pi$, if
necessary, such that $|M_{a_n+1}'|$ is base point free. In
particular, $M_{a_n+1}'$ is nef. Since $X$ is of general type
$|mK_X|$ gives a birational map whenever $m$ is big enough. Thus
we see that $M_{a_n+1}'$ is big as we fix a very big $t_0$ in
advance.

Now the Kawamata-Viehweg vanishing theorem again gives
\begin{eqnarray*}
|K_{X'}+M_{a_n+1}'+S||_S&=&|K_S+M_{a_n+1}'|_S|\\
&\supset& |K_S+(b_n+1)\sigma^*(K_{S_0})|\\
&\supset& |(b_n+2)\sigma^*(K_{S_0})|.
\end{eqnarray*}

We may repeat the above procedure. Denote by $M_{a_n+t}'$ the
movable part of $|K_{X'}+M_{a_n+t-1}'+S|$ for $t\ge 2$. For the
same reason, we may assume $|M_{a_n+t}'|$ to be base point free.
Inductively one has:
$$M_{a_n+2}'|_S\geq (b_n+2)\sigma^*(K_{S_0})$$
and in general
$$M_{a_n+t}'|_S\geq (b_n+t)\sigma^*(K_{S_0})$$

Just take $t=p$ and set $a_{n+1}:=a_n+p+m_0$ and $b_{n+1}=b_n+p$.
Noting that
$$|K_{X'}+M_{a_n+p-1}'+S|\subset |(a_n+p+m_0)K_{X'}|$$
and applying Lemma 2.7 of \cite{MPCPS} again, one has
$$a_{n+1}\pi^*(K_X)|_S\geq M_{a_n+p+m_0}|_S\geq M'_{a_n+p}|_S\geq
b_{n+1} \sigma^*(K_{S_0}).$$  Set $\beta_n := \frac{b_{n}}{a_{n}}.$
Clearly $\underset{{n\mapsto +\infty}}\lim \beta_n =
\frac{p}{m_0+p}$. We have proved the lemma when $p\geq 2$.
\medskip

{\bf Case 2}. The lemma at the case $p=1$ can be proved similarly
with a simpler induction. We omit the proof and leave it to readers
as an exercise.
\end{proof}



\section{\bf Lower bound of the volume and non-vanishing}
In this section, we are going to utilize the general method
developed in Section 2.
\begin{thm}\label{volume} Let $V$ be a nonsingular projective
3-fold of general type with $P_{m_0} \ge 2$. Then
\begin{itemize}
\item [(i)] $\Vol(V)\geq \frac{10}{m_0^2(3m_0+2)}$ for type III.

\item [(ii)] $\Vol(V)\geq \frac{4}{m_0^2(3m_0+2)}$  for type II.

\item [(iii)] $\Vol(V) \geq \frac{36}{5m_0(m_0+2)^2}$ for type
I$_3$.

\item [(iv)] $\Vol(V)\geq \frac{11}{12m_0(m_0+1)^2}$ in general.
\end{itemize}
\end{thm}

\begin{proof} Take a minimal model $X$ of $V$. We study $|mK_X|$ on $X$.
Keep the same setup as in \ref{setup}. Then $\Vol(V)=K_X^3$ as we
have known.
\medskip

{\bf Part (i).} For type III, i.e. $\dim (B)=3$, we know that $p=1$
by definition. In this case we pick $S\sim M_{m_0}$ and that $|S|$
gives a generically finite morphism. Set $G:=S|_S$. Then $|G|$ is
base point free and $\varphi_{|G|}$ gives a generically finite map.
So a generic irreducible element $C$ of $|G|$ is a smooth curve.

If $\varphi_{|G|}$ gives a birational map, then $\dim
\varphi_{|G|}(C)=1$ for a general member $C$. The Riemann-Roch and
Clifford's theorem on $C$ says $C^2=G\cdot C\geq 2$. If
$\varphi_{|G|}$ gives a generically finite map of degree $\geq 2$,
since $h^0(S,G)\geq h^0(X',S)-1\geq 3$, \cite[Lemma 2.2]{JMSJ3}
gives $C^2\geq 2h^0(S,G)-4\geq 2$. Anyway we have $C^2\geq 2$. So
$\deg(K_C)=(K_S+C)\cdot C>2C^2\geq 4$. We see $\deg(K_C)\geq 6$
because it is even.

One may take $\beta=\frac{1}{m_0}$ since $m_0\pi^*(K_X)|_S\geq C$.
Now if we take  $m \gg 0$ such that $\alpha>1$ then Theorem
\ref{technical} gives:
$$m\xi\geq \deg(K_C)+(m-1-m_0-\frac{1}{\beta})\xi.$$
This gives $\xi\geq \frac{6}{2m_0+1}$. Take $m=3m_0+2$. Then
$\alpha=(m-2m_0-1)\xi>3$. So by Theorem \ref{technical} again,
$\xi\geq \frac{10}{3m_0+2}$. It follows that $K^3\geq
\frac{10}{(3m_0+2)m_0^2}$.

\medskip

{\bf Part (ii).} If $\dim(B)=2$, we pick $S \sim M_{m_0}$ and
$|G|:=|S|_S|$ is composed with a pencil of curves.

A generic irreducible element $C$ of $|G|$ is a smooth curve of
genus $\geq 2$, so $\deg(K_C)\geq 2$. Furthermore we have
$h^0(S,G)\geq h^0(X',S)-1\geq 2$. So $G\equiv \widetilde{a} C$ for
$\widetilde{a}\geq 1$. This means $m_0\pi^*(K_X)|_S\geq
S|_S\geq_{\text{num}} C$. So we may take $\beta = \frac{1}{m_0}$.

Now take a  $m \gg 0$. Remark \ref{weak} gives $\xi\geq
\frac{2}{2m_0+1}$. Take $m=3m_0+2$. Then $\alpha>1$. One gets
$\xi\geq \frac{4}{3m_0+2}$ by Theorem \ref{technical}. So $K^3\geq
\frac{4}{(3m_0+2)m_0^2}$.

{\bf Part (iii).}
Take $S \sim M_{m_0}$ and $G:=4\sigma^*(K_{S_0})$, where $S_0$ is
the minimal model of $S$. The surface theory tells us that $|G|$ is
base point free and a generic irreducible element $C$ of $|G|$ is a
smooth curve. Because
$$\deg(K_C)=(K_S+C)\cdot C\geq (\pi^*(K_X)|_S+C)\cdot C>C^2\geq 16,$$
again we see $\deg(K_C)\geq 18$.

We know that $\pi^*(K_X)|_S-\widetilde{\beta}_n\sigma^*(K_{S_0})$ is
numerically equivalent to an effective $\bQ$-divisor for a rational
number sequence $\{\widetilde{\beta}_n\}$ with
$\widetilde{\beta}_n\mapsto \frac{p}{p+m_0}\geq \frac{2}{m_0+2}$.
Take $\beta_n:=\frac{1}{4}\widetilde{\beta}_n$. Then
$\pi^*(K_X)|_S-\beta_n C$ is numerically equivalent to an effective
$\bQ$-divisor. We can take a rational number $\beta =
\frac{1}{2(m_0+2)} -\delta$ with $0 < \delta \ll 1$.

If we take  $m \gg 0$, then $\xi\geq
\frac{18}{1+\frac{m_0}{2}+\frac{1}{\beta}}$ by Remark \ref{weak}.
So $\xi\geq\frac{36}{5(m_0+2)}$ by letting $\delta$ go to $0$. One
gets $K^3\geq \frac{36}{5m_0(m_0+2)^2} \ge
\frac{4}{(3m_0+2)m_0^2}$.


\medskip

{\bf Part (iv).} It remains to  study the case $\dim (B)=1$. When
$q>0$, one has $K^3\geq \frac{1}{22}$ by \cite{Jungkai-Meng} and one
can easily verify the inequality $K^3\geq
\frac{11}{12m_0(m_0+1)^2}$. So we may assume $q=0$. So $B$ is a
rational curve. We set $G:=4\sigma^*(K_{S_0})$.  Again we see
$\deg(K_C)\geq 18$. We know that
$\pi^*(K_X)|_S-\widetilde{\beta}_n\sigma^*(K_{S_0})$ is numerically
equivalent to an effective $\bQ$-divisor for a rational number
sequence $\{\widetilde{\beta}_n\}$ with $\widetilde{\beta}_n\mapsto
\frac{p}{p+m_0}\geq \frac{1}{m_0+1}$. Take
$\beta_n:=\frac{1}{4}\widetilde{\beta}_n$. Then
$\pi^*(K_X)|_S-\beta_n C$ is numerically equivalent to an effective
$\bQ$-divisor. We know $\beta_n\mapsto \frac{1}{4m_0+4}$ whenever
$p=1$. We thus take $\beta=\frac{1}{4m_0+4}-\delta$ for some $0 <
\delta \ll 1$.

When $m \gg 0$, Remark \ref{weak} gives $\xi\geq \frac{18}{5m_0+5}$.
Take $m=6m_0+6$. Then $\alpha >3$ and  Theorem \ref{technical} gives
$\xi\geq \frac{11}{3m_0+3}$. So $K^3\geq \frac{11}{12m_0(m_0+1)^2}$.
\end{proof}

\begin{setup}{\bf Refinement of lower bounds of $K^3$.}
Indeed, for small $m_0$, we can improve the lower bound of the
canonical volume. We study some special cases that occur in our
paper.

For example, in type III, assume $m_0=11$. Then $\xi\geq
\frac{6}{23}$ by taking $m \gg 0$. Next take $m=27$. By Theorem
\ref{technical}, we have $\xi\geq \frac{8}{27}$. So inequality
(2.1) gives $K^3\geq \frac{8}{3267}$.

Let's assume  $m_0=8$ for type II as another example. Then
$\beta\geq \frac{1}{8}$. One has already $\xi\geq \frac{2}{17}$ by
taking $m \gg 0$. Take $m=26$. Then $\alpha\geq \frac{18}{17}>1$.
We get $\xi\geq \frac{2}{13}$. Take $m=24$. Then $\alpha>1$. One
gets $\xi\geq \frac{1}{6}$. So inequality (2.1) gives $K^3\geq
\frac{1}{384}$.

A patient reader should have no difficulty to check the following
table on the lower bound of $K^3$ for small $m_0$. We tag it as:
\medskip

\centerline{\bf Table A}
\smallskip

{\tiny
\begin{tabular}{c|ccccccccccc}
 $m_0$ & 2 & 3 & 4 & 5 &6 & 7& 8 & 9 & 10 & 11 &12\\
\hline
     III &$\frac{1}{3}$ &$\frac{8}{81}$ &$\frac{1}{22}$ &$\frac{8}{325}$ & $\frac{1}{72}$ &
     $\frac{4}{441}$&
     $\frac{1}{160}$ & $\frac{4}{891}$ & $\frac{2}{625}$ & $\frac{8}{3267}$ & $\frac{1}{522}$
     \\ \hline
     II &$\frac{1}{8}$ &$\frac{2}{45}$ &$\frac{1}{52}$ &$\frac{1}{100}$ & $\frac{1}{162}$ & $\frac{4}{1029}$
     &$\frac{1}{384}$ & $\frac{2}{1053}$ & $\frac{1}{725}$ & $\frac{1}{968}$ & $\frac{1}{1224}$
     \\ \hline
     I$_3$ &$\frac{1}{8}$ &$\frac{2}{45}$ &$\frac{1}{52}$ &$\frac{1}{100}$ & $\frac{1}{162}$ & $\frac{4}{1029}$
     &$\frac{1}{384}$ & $\frac{2}{1053}$ & $\frac{1}{725}$ & $\frac{1}{968}$ & $\frac{1}{1224}$
     \\ \hline
     general &$\frac{5}{96}$ &$\frac{5}{264}$ &$\frac{1}{108}$ &$\frac{1}{192}$ & $\frac{5}{1554}$ & $\frac{5}{2408}$
     &$\frac{5}{3456}$ & $\frac{1}{954}$ & $\frac{1}{1276}$ & $\frac{5}{8448}$ & $\frac{5}{10764}$    \\
\end{tabular}
}
\end{setup}
\medskip

We now study the non-vanishing problem of plurigenera.

\begin{defn} Let $X$ be a minimal projective 3-fold of
general type. Define $m_1:=m_1(X)$ to be the smallest positive
integer such that $P_{m}(X)>0$ for all $m\geq m_1(X)$. Clearly
$m_1(X)$ is a birational invariant of $X$. One knows
$m_1(X)<+\infty$ by Matsusaka's big theorem.
\end{defn}

 In fact, by Proposition \ref{easy-nv},
one has $m_1 \le 5m_0+6$ already. We will need the following easy
lemma to get a better bound.

\begin{lem}\label{>0} Let $S$ be a nonsingular projective surface
of general type. Denote by $\sigma:S\longrightarrow S_0$ the
blow-down onto its minimal model $S_0$. Let $Q$ be a $\bQ$-divisor
on $S$. Then $h^0(S,K_S+\roundup{Q})\geq 2$ under one of the
following conditions:
\begin{itemize}
\item[(i)] $p_g(S)>0$, $Q\equiv \sigma^*(K_{S_0})+Q_1$ for some
nef and big $\bQ$-divisor $Q_1$ on $S$;

\item[(ii)] $p_g(S)=0$, $Q\equiv 2\sigma^*(K_{S_0})+Q_2$ for some
nef and big $\bQ$-divisor $Q_2$ on $S$;
\end{itemize}
\end{lem}
\begin{proof} First of all $h^0(S, 2K_S)=h^0(S,2K_{S_0})>0$ by the
Riemann-Roch theorem on $S$, which is a surface of general type. Fix
an effective divisor $R_0\sim l\sigma^*(K_{S_0})$, where $l=1,2$ in
cases (i) and (ii) respectively. Then $R_0$ is nef and big and $R_0$
is 1-connected by \cite[Lemma 2.6]{Maga}. The Kawamata-Viehweg
vanishing theorem says $H^1(S, K_S+\roundup{Q}-R_0)=0$ which gives
the surjective map:
$$H^0(S, K_S+\roundup{Q})\longrightarrow
H^0(R_0,K_{R_0}+G_{R_0})$$ where $G_{R_0}:=(\roundup{Q}-R_0)_{|R_0}$
with $\deg(G_{R_0})\geq (Q-R_0)R_0=Q_l\cdot R_0>0$. The
1-connectedness of $R_0$ allows us to utilize the Riemann-Roch (see
Chapter II, \cite{BPV}) as in the usual way. Note that $S$ is of
general type. So $K_{S_0}^2>0$ and
$\deg(K_{R_0})=2p_a(R_0)-2=(K_S+R_0)R_0\geq 2$. By the Riemann-Roch
theorem on the 1-connected curve $R_0$, we have
$$h^0(R_0,K_{R_0}+G_{R_0})\geq \deg(K_{R_0}+G_{R_0})+1-p_a(R_0)\geq p_a(R_0)
\geq 2.$$
Hence $h^0(S,K_S+\roundup{Q})\geq 2$.
\end{proof}

\begin{prop}\label{nonvanishing} Let $X$ be a minimal projective
3-fold of general type with $P_{m_0}\geq 2$. Keep the same
notation as in \ref{setup}. Then $m_1$ has an upper bound under
each of the following situations:
\begin{itemize}
\item[(i)] $P_{m}\geq 2$ for all $m\geq 2m_0$ for type III;

 \item[(ii)] $P_{m}\geq 2$ for all $m\geq 2m_0$ for type II;

 \item[(iii)] $P_{m}\geq 2$ for all $m\geq 2m_0+3$  for type I$_{p}$;

 \item[(iv)] $P_{m}\geq 2$ for all $m\geq 3m_0+4$  for type I$_n$;

 \item[(v)]
$P_{m}\geq 2$ for all $m\geq \rounddown{\frac{3m_0}{2}}+4$  for type
I$_3$.
\end{itemize}
\end{prop}

\begin{proof} We keep the notation as in Section 2.

(i). For type III, one can pick $\beta=\frac{1}{m_0}$, thus by
Proposition \ref{easy-nv}, we have $P_m \geq 2$ for all $ m >
2m_0+1$. Now if $m=2m_0+1$, the surjection (2.2) and (2.3) lead us
to consider non-vanishing of $H^0(S, K_S+ \roundup{m_0
\pi^*K_X|_S})$. Let $L$ be an element in $|M_{m_0}|_{|S}$, then
clearly, $h^0(S, K_S+L)=\chi(S,K_S+L)\geq 2$ by Riemann-Roch
theorem. Hence $P_{2m_0+1}\geq 2$. Also, $P_{2m_0} \ge P_{m_0} \geq
2$. Therefore, we have $m_1 \le 2m_0$.

(ii). For type II, since we can take $\beta=\frac{1}{m_0}$, exactly
the same proof shows that $m_1 \le 2m_0$.

(iii). For type I,  Lemma \ref{beta} gives that  there is a
sequence of rational numbers $\{\beta_n\}$ with $\beta_n\mapsto
\frac{p}{m_0+p}\geq \frac{1}{m_0+1}$ such that
$$\pi^*(K_X)_{|S}-\beta_n\sigma^*(K_{S_0})\equiv H_n$$
for an effective $\bQ$-divisor $H_n$.

We consider
$$\mathcal{D}'_m:=(\mathcal{L}_m-S)_{|_S}-(m-1-\frac{m_0}{p}) H_n
\equiv (m-1-\frac{m_0}{p})\beta_n \sigma^*(K_{S_0}).$$

If $h^0(S, K_S+\roundup{\mathcal{D}'_m})\geq 2$, then so is $h^0(S,
K_S+\roundup{(\mathcal{L}_m-S)_{|_S}})$. It follows that $P_m \geq
2$ by (2.2) and the surjection (2.3).

In case I$_p$, we can pick $\beta_n = \frac{1}{m_0+1}-\delta$ for
some $0 < \delta \ll 1$. So when $m \ge 2m_0+3$,
$(m-1-\frac{m_0}{p})\beta_n
>1$. By Lemma \ref{>0}, we have $h^0(S,
K_S+\roundup{\mathcal{D}'_m})\geq 2$. Thus $m_1 \le 2m_0+3$.

In case I$_n$, similarly, we can pick $\beta_n =
\frac{1}{m_0+1}-\delta$ for some $0 < \delta \ll 1$. So when $m
\ge 3m_0+4$, $(m-1-\frac{m_0}{p})\beta_n
>2$. By Lemma \ref{>0}, we have $h^0(S,
K_S+\roundup{\mathcal{D}'_m})\geq 2$. Thus $m_1 \le 3m_0+4$.

In case I$_3$, we can pick $\beta_n = \frac{2}{m_0+2}-\delta$ for
some $0 < \delta \ll 1$. So when $m \ge
\rounddown{\frac{3m_0}{2}}+4$, we have $(m-1-\frac{m_0}{p})\beta_n
>2$ and $P_m\geq 2$. Thus $m_1 \le \rounddown{\frac{3m_0}{2}}+4$.

This completes the proof.
\end{proof}

\begin{rem}\label{chi(O)>1}
We would like to remark that the case I$_n$ implies $\chi \leq 1$.
To see this,
  we compute the invariant $\chi(\OO_X)$ under this
situation. We have an induced fibration $f:X'\longrightarrow B$
onto the smooth rational curve $B$. A general fiber $S$ of $f$ is
a nonsingular projective surface of general type with $p_g(S)=0$.
Because $\chi(\OO_S)>0$, we see $q(S)=0$. This means
$f_*\omega_{X'}=0$ and $R^1f_*\omega_{X'}=0$ since they are both
torsion free. Thus we get by \ref{inv} the following formulae:
$$h^2(\OO_X)=h^2(\OO_{X'})=h^1(f_*\omega_{X'})+h^0(R^1f_*\omega_{X'})=0;$$
$$q(X)=q(X')=g(B)+h^1(R^1f_*\omega_{X'})=0.$$
So we see $\chi(\OO_X)=1-q(X)+h^2(\OO_X)-p_g(X)\leq 1$.
\end{rem}

By a result of the first author and C. D. Hacon \cite{JC-H}, $P_m
> 0$ for all $m \ge 2$ if $q(V)>0$. Hence $m_1 \le 2$ for
irregular varieties of general type. It follows that, for a
threefold $V$ with $\chi(\mathcal{O}_V) \le 0$, one has $m_1 \le 2$.
We summarize the non-vanishing property as the following:

\begin{cor} Let $V$ be a nonsingular projective 3-fold of general type. Then
$m_1 \le 2$ if $\chi(\OO_V) \le 0$. Suppose furthermore that
$P_{m_0} \ge 2$ for some $m_0>0$, then $m_1 \le 3m_0+4$.  If
$P_{m_0}\geq 2$ and $\chi(\OO_V)>1$, then $m_1 \le 2 m_0 +3$.
\end{cor}

\section{\bf Baskets of singularities}

We always consider minimal projective 3-folds  of general type in
this section.

\begin{setup}{\bf Terminal quotient singularity and basket.} By {\it a 3-dimensional terminal quotient singularity $Q$ of
type $\frac{1}{r}(1,-1,b)$}, we mean a singularity which is
analytically isomorphic to the quotient of $(\bC^3, \text{O})$ by a
cyclic group action  $\varepsilon$:
$$\varepsilon(x,y,z)=(\varepsilon x,\varepsilon^{-1} y, \varepsilon^b z)$$
where $r$ is a positive integer, $\varepsilon$ is a fixed $r$-th
primitive root of $1$, the integer $b$ is coprime to $r$ and
$0<b<r$.
\end{setup}

\begin{setup}{\bf Convention.} By replacing $\varepsilon$ with
another primitive root of 1 and changing the ordering of
coordinates, we may even assume that $b \le \frac{r}{2}$.
\end{setup}
{\it A basket} $\mathscr{B}$ of singularities is a collection
(permitting weights) of terminal quotient singularities of type
$\frac{1}{r_i}(1,-1,b_i)$, $i\in I$ where $I$ is a finite index set.
{\it A single basket} means a single singularity $Q$ of type
$\frac{1}{r}(1,-1,b)$. For simplicity, we will always denote a
single basket by $(b,r)$. So we will simply write a basket as:
$$\mathscr{B}:=\{n_i\times (b_i,r_i)|i\in I,\ n_i\in {\mathbb Z}^+\}.$$

\begin{defn} When an integer $b$ is not coprime to another integer
$r$, we still call the symbol $(b,r)$ {\it a generalized single
basket} though it doesn't mean anything at this moment. {\it A
generalized basket} means a collection of single baskets and
generalized single baskets.
\end{defn}

\begin{setup}{\bf Plurigenera.}\label{pm}
Let us recall Reid's plurigenus formula (cf.  \cite{YPG}, p413) for
a minimal 3-fold $X$ of general type (with $\bQ$-factorial terminal
singularities):
\medskip

\noindent there exists a  ``virtual'' basket\footnote{Iano-Fletcher
\cite{Fletcher} has shown that Reid's virtual basket
$\mathscr{B}(X)$ is uniquely determined by $X$.} $\mathscr{B}(X)$
of terminal quotient singularities such that, for all $m>1$,
$$P_m(X)=\frac{1}{12}m(m-1)(2m-1)K_X^3-(2m-1)\chi(\mathcal
{O}_X)+l(m)  \eqno (4.1)$$ where the correction term $l(m)$ can be
computed as:
$$l(m):=\sum_{Q\in \mathscr{B}(X)}l_Q(m):=\sum_{Q\in \mathscr{B}(X)}\sum_{j=1}^{m-1}
\frac{\overline{jb_Q}(r_Q-\overline{jb_Q})}{2r_Q}$$ where the sum
$\sum_{Q}$ runs through all single baskets $Q$ of $\mathscr{B}(X)$
with type $\frac{1}{r_{Q}}(1,-1,b_{Q})$ and $\overline{jb_Q}$ means
the smallest residue of $jb_Q$ \text{mod} $r_Q$.
\medskip

We are going to analyze the above formula and Reid's virtual basket
$\mathscr{B}(X)$.
\end{setup}



\begin{setup}{\bf Invariants of baskets.} Given a generalized single basket
$(b,r)$ ($b$ not necessarily coprime to $r$) with $b\leq
\frac{r}{2}$ and a fixed integer $n>0$. Let $i:=\lfloor \frac{bn}{r}
\rfloor$. Then $\frac{i+1}{n}
> \frac{b}{r} \ge \frac{i}{n} $.  We define $$\Delta^n_{b,r}:=
ibn-\frac{(i^2+i)}{2}r.$$ One can see that $\Delta^n_{b,r}$ is a
non-negative integer. For a generalized basket $B= \{(b_i,r_i)|
i\in I\}$ and a fixed $n>0$, we define $\Delta^n(B):=
\underset{i\in I}\sum \Delta^n_{b_i,r_i}$. By definition,
$\Delta^2(B)=0$ for any basket $B$. By a direct calculation, one
gets the following relation:
$$\frac{\overline{jb_i}(r_i-\overline{jb_i})}{2r_i} - \frac{{jb_i}(r_i-{jb_i})}{2r_i} =
\Delta^j_{b_i,r_i}$$ for all $j>0$. Define $\sigma(B):=
\underset{i\in I}\sum b_i$ and $\sigma'(B):=\underset{i\in I}\sum
\frac{b_i^2}{r_i}$.
\end{setup}

\begin{setup}\label{del}{\bf Plurigenera in terms of $\Delta^m$.}
We can rewrite Reid's plurigenus formula as the following, where we
take $B={\mathscr{B}}(X)$ and $\Delta^m=\Delta^m(B)$:
$$(\partial)\left\{ \begin{array}{ll} P_2 &= \frac{1}{2} K_X^3 -3 \chi + \frac{1}{2} \sigma
-\frac{1}{2} \sigma',\\
 P_3-P_2 & = \frac{4}{2} K_X^3 -2 \chi + \frac{2}{2} \sigma
-\frac{4}{2} \sigma',\\
 P_{m+1}-P_m & = \frac{m^2}{2} K_X^3 -2 \chi + \frac{m}{2} \sigma
-\frac{m^2}{2} \sigma'+\Delta^m, \text{ for } m \geq 3.
\end{array} \right.
$$
\end{setup}

\begin{setup}\label{pk}{\bf Packing.}
Next we define a notion of ``packing''. Given a generalized basket
$$B=\{(b_1,r_1),(b_2,r_2),\cdots, (b_k,r_k)\},$$  we call the basket
$$B':=\{ (b_1+b_2,r_1+r_2), (b_3,r_3), \cdots, (b_k,r_k)
\}$$ a packing of $B$, written as $B \succ B'$. If furthermore
$b_1r_2-b_2r_1 =1$, we call $B \succ B'$ {\it a convenient
packing}.
\end{setup}

We have the following:
\begin{lem} \label{packing} Let $B \succ B'$ be any packing between generalized baskets.
Keep the same notations as above. Then:
\begin{itemize}
\item[(1)]
$\Delta^n(B) \ge \Delta^n(B')$ for all $n \ge 2$;
\item[(2)] the equality in (1) holds if and only if $\frac{i}{n} \leq
\frac{b_1}{r_1}, \frac{b_2}{r_2} \leq \frac{i+1}{n}$ for
some $i$;
\item[(3)]$\sigma(B')=\sigma(B)$ and $\sigma'(B) = \sigma'(B') +
\frac{(r_1b_2-r_2b_1)^2}{r_1r_2(r_1+r_2)} \geq
\sigma'(B').$ Thus equality holds only when
$\frac{b_1}{r_1}= \frac{b_2}{r_2}$.
\end{itemize}
\end{lem}

\begin{proof} First, if $\frac{i}{n} \le \frac{b_1}{r_1}, \frac{b_2}{r_2}
\le \frac{i+1}{n}$ for some $i$, then a direct calculation shows
$\Delta^n(B) = \Delta^n(B')$.

Suppose, for some $i>j$,
$$\frac{i+1}{n}> \frac{b_2}{r_2} \geq \frac{i}{n} \geq
\frac{j+1}{n} > \frac{b_1}{r_1} \geq \frac{j}{n}$$ and
 $\frac{j_1+1}{n}>\frac{b_1+b_2}{r_1+r_2} \geq \frac{j_1}{n}$
for some $j_1\in [j,i]$. Then
\begin{eqnarray*}
\Delta^n_{b_1+b_2,r_1+r_2}&=&
j_1n(b_1+b_2)-\frac{1}{2}(j_1^2+j_1)(r_1+r_2)\\
&=& \Delta^n_{b_2,r_2}+\Delta^n_{b_1,r_1}+\nabla_2+\nabla_1,
\end{eqnarray*}
where
$\nabla_2=(j_1-i)nb_2+\frac{1}{2}(i^2+i-j_1^2-j_1)r_2$ and
$\nabla_1=(j_1-j)nb_1+\frac{1}{2}(j^2+j-j_1^2-j_1)r_1$. Now
since $nb_2\geq ir_2$, one gets
$$\nabla_2\leq
\frac{1}{2}(i-j_1)(j_1+1-i)r_2.$$ When $j_1=i$,
$\nabla_2=0$; when $j_1=i-1$, $\nabla_2=-nb_1+ir_2\leq 0$;
when $j_1<i-1$, $\nabla_2<0$.

Similarly the relation $nb_1<(j+1)r_1$ implies
$$\nabla_1\leq \frac{1}{2}(j_1-j)(j+1-j_1)r_1.$$
When $j_1=j$, $\nabla_1=0$; when $j_1=j+1$,
$\nabla_1=nb_1-(j+1)r_1<0$; when $j_1>j+1$, $\nabla_1<0$.

Thus in any case, we see $\Delta^n(B) \geq \Delta^n(B')$, which
implies (1). Furthermore we see $\Delta^n(B) = \Delta^n(B')=0$ if
and  only if $\nabla_2=\nabla_1$, if and only if $j_1=j$ and
$i=j_1+1=j+1$. We have proved (2).

The inequality (3) is obtained by a direct calculation.
\end{proof}

\begin{cor}\label{addition} If $B=\{m\times (b,r)|\ b\leq \frac{r}{2},\ b \text{ coprime to } r\}$
and $B'=\{(mb, mr)\}$ for an integer $m>1$, then
\begin{itemize}
\item[(i)] $\sigma(B')=\sigma(B)$; $\sigma'(B')=\sigma'(B)$;
\item[(ii)] $\Delta^n(B')=\Delta^n(B)$ for any $n>0$.
\end{itemize}
\end{cor}
\begin{proof} This can be obtained by the definition of $\sigma$
and Lemma \ref{packing}.
\end{proof}

\begin{rem} The additive properties in Corollary \ref{addition}
allow us to view the generalized single basket $(mb,mr)$ as a
basket $\{m\times (b,r)\}$.
\end{rem}

Besides, a convenient packing has the following basic properties:

\begin{lem}
Let $B \succ B'$ be a convenient packing as in \ref{pk}, i.e.
$b_1r_2-b_2 r_1 =1$.
Then $\Delta^{r_1+r_2}_{b_1+b_2, r_1+r_2}=\Delta^{r_1+r_2}_{b_1,
r_1}+\Delta^{r_1+r_2}_{b_2, r_2}-1$.
\end{lem}

\begin{proof} \label{combini}
 When $b_1r_2-b_2r_1=1$, since $r_1> 1, r_2>  1$, one
 has
$$\frac{b_1+ b_2+1}{ r_1+ r_2} >
\frac{b_1}{r_1}
> \frac{ b_1+ b_2}{ r_1+ r_2} > \frac{b_2}{r_2}>
 \frac{ b_1+ b_2-1}{ r_1+ r_2}.$$
  We set $n=r_1+ r_2$. A direct calculation gives the equality
$$\Delta^{n}_{b_1+b_2, r_1+r_2}=\Delta^{n}_{b_1,
r_1}+\Delta^{n}_{b_2, r_2}-1.$$
\end{proof}

\begin{setup}{\bf Initial basket and limiting process.}
Given a basket $B=\{(b_i,r_i)|\ i\in \text{I}, b_i \text{ coprime
to } r_i, b_i\leq \frac{r_i}{2}\}$ with $I$ a finite set, we define
a sequence of baskets $\{\mathscr{B}^{(n)}(B)\}$.

Take a set $S^{(0)}:=\{\frac{1}{n}\}_{n \ge 2}$. For any single
basket $B_i=(b_i,r_i) \in B$, we can find a unique $n>0$ such that
$\frac{1}{n} > \frac{b_i}{r_i} \geq \frac{1}{n+1}$.
The single basket $(b_i,r_i)$ can be regarded as  successive
packings  via finite steps beginning from the basket
$B_i^{(0)}:=\{(nb_i+b_i-r_i) \times (1,n), (r_i-nb_i) \times
(1,n+1)\}$. Adding up those $B_i^{(0)}$, one obtains the basket
$\mathscr{B}^{(0)}(B)=\{ n_{1,2} \times (1,2), n_{1,3} \times
(1,3),\cdots, n_{1,r} \times (1,r)\}$, called {\it the initial
basket} of $B$. Clearly $\mathscr{B}^{(0)}(B)\succ B$. Defined in
this way, $\mathscr{B}^{(0)}(B)$ is uniquely determined by the given
basket $B$.

We begin to construct related baskets $\{\mathscr{B}^{(n)}(B)\}$ for
$n\geq 1$. Consider the sets
$S^{(1)}=S^{(2)}=S^{(3)}=S^{(4)}=S^{(0)}$ and
$$S^{(5)}:=S^{(0)} \cup \{\frac{2}{5}\}$$
and inductively,
$S^{(n)}=S^{({n-1})} \cup \{\frac{i}{n} \}_{i=2,...,\lfloor
\frac{n}{2} \rfloor}$. Reordering elements in $S^{(n)}$ and writing
$S^{(n)}=\{w^{(n)}_i\}_{i\in I}$
 such that $w^{(n)}_i > w^{(n)}_{i+1}$ for all $i$, then we see that
 the interval $(0,\frac{1}{2}]=\cup_i[w^{(n)}_{i+1}, w^{(n)}_i]$.
 Note that $w^{(n)}_i=\frac{q_i}{p_i}$ with  $p_i$ coprime to $q_i$ and
 $p_i \leq n$ unless $w^{(n)}_i=\frac{1}{m}$ for some $m >n$. First
 we prove the following:
\medskip

\noindent {\bf Claim A.} {\em $p_{i+1}q_i-p_iq_{i+1}=1$ for any two
endpoints of
 $[w^{(n)}_{i+1}, w^{(n)}_i]=[\frac{q_{i+1}}{p_{i+1}}, \frac{q_i}{p_i}]$.}

\begin{proof}
We can prove this inductively. Suppose that this property holds for
$S^{({n-1})}$. Now, for any $\frac{j}{n} \in S^{(n)}-S^{(n-1)}$,
$\frac{j}{n} \in [w^{n-1}_{i+1}, w^{n-1}_i]$ for some $i$. Thus $
\frac{q_{i+1}}{p_{i+1}} < \frac{j}{n} < \frac{q_{i}}{p_{i}}$. If
$p_i \ge n$, then $\frac{q_i}{p_i}=\frac{1}{m}$ and
$\frac{q_{i+1}}{p_{i+1}}=\frac{1}{m+1}$ for some $m \ge n$ which
contradicts to $\frac{j}{n} < \frac{q_i}{p_i}$. Therefore, we must
have  $p_{i} <n$. Then we consider $\frac{j-q_i}{n-p_i}$ and it's
easy to see that
$$
\frac{q_{i+1}}{p_{i+1}}\le \frac{j-q_i}{n-p_i} < \frac{j}{n} <
\frac{q_{i}}{p_{i}}.$$ Clearly, $\frac{j-q_i}{n-p_i} \in S^{(n-1)}$
and hence $\frac{j-q_i}{n-p_i}=\frac{q_{i+1}}{p_{i+1}}$. It follows
that $n=p_i+\alpha p_{i+1}, j=q_i+\alpha q_{i+1}$ for some integer
$\alpha>0$.

If $\alpha \ge 2$, then $\frac{q_{i+1}}{p_{i+1}}<
\frac{q_i+(\alpha-1)q_{i+1}}{p_i+(\alpha-1)p_{i+1}}  <
\frac{q_{i}}{p_{i}}$, and
$\frac{q_i+(\alpha-1)q_{i+1}}{p_i+(\alpha-1)p_{i+1}}  \in
S^{(n-1)}$, which is absurd. Thus $\alpha=1$ and then $n=p_i+
p_{i+1}, j=q_i+ q_{i+1}$. It's then clear  that $\frac{j}{n}$
is the only element of $S^{(n)}$ inside the interval
$[\frac{q_{i+1}}{p_{i+1}},\frac{q_{i}}{p_{i}}]$. Moreover, $j
p_{i+1}-n q_{i+1}=1, n q_i -j p_i=1$. This completes the proof of
the claim.
\end{proof}

Now for a single basket $B_i=(b_i,r_i) \in B$, if $\frac{b_i}{r_i}
\in S^{(n)}$, then we set $B^{(n)}_i:=\{(b_i,r_i)\}$. If
$\frac{b_i}{r_i}\not\in S^{(n)}$, then $\frac{q_1}{p_1} <
\frac{b_i}{r_i} < \frac{q_2}{p_2}$ for some interval
$[\frac{q_1}{p_1},\frac{q_2}{p_2}]$ due to $S^{(n)}$.  In this
situation, we can unpack $(b_i,r_i)$ to $B^{(n)}_i:=\{(r_i
q_2-b_ip_2) \times (q_1,p_1),(-r_i q_1+b_i p_1) \times (q_2,p_2)\}$.
Adding up those $B^{(n)}_i$, we get a new basket
$\mathscr{B}^{(n)}(B)$. $\mathscr{B}^{(n)}(B)$ is uniquely defined
according to our construction and $\mathscr{B}^{(n)}(B) \succ B$ for
all $n$.
\end{setup}

\noindent{\bf Claim B.} {\em
$\mathscr{B}^{(n-1)}(B)=\mathscr{B}^{(n-1)}(\mathscr{B}^{(n)}(B))
\succ \mathscr{B}^{(n)}(B)$ for all $n\geq 1$.}
\begin{proof} It's clear that $\mathscr{B}^{(n-1)}(\mathscr{B}^{(n)}(B))
\succ \mathscr{B}^{(n)}(B)$. Thus it suffices to show the first
equality. By the definition of $\mathscr{B}^{(n)}$, we only need to
prove for each single basket $B_i=(b_i,r_i) \in B$ and $n \ge 5$.

If $\frac{b_i}{ r_i} \in S^{(n-1)} \subset S^{(n)}$, then there is
nothing to prove since the equality follows from the definition of
$\mathscr{B}^{(n)}$ and $\mathscr{B}^{(n-1)}$.

If $\frac{b_i}{ r_i} \in  S^{(n)}-S^{(n-1)}$, then this is also
clear since $\mathscr{B}^{(n)}(B_i) = B_i$.

Suppose finally that  $\frac{b_i}{ r_i} \not \in  S^{(n)}$. Then
$\frac{q_1}{p_1} < \frac{b_i}{r_i} < \frac{q_2}{p_2}$ for some
$\frac{q_1}{p_1}=w^{(n)}_{i+1}$ and $\frac{q_2}{p_2}=w^{(n)}_i$.

{\bf Subcase (i).}  If both of $\frac{q_1}{p_1},\ \frac{q_2}{p_2}$
are in $S^{(n)}- S^{(n-1)}$, then $p_1=p_2=n$ and hence
$p_1q_2-p_2q_1 \neq 1$, a contradiction to Claim A.

{\bf Subcase (ii).} If both $\frac{q_1}{p_1}$ and $\frac{q_2}{p_2}$
are in $S^{(n-1)}$, then by definition
$$\mathscr{B}^{(n-1)}(B_i)=\mathscr{B}^{(n)}(B_i)=\mathscr{B}^{(n-1)}(\mathscr{B}^{(n)}(B_i)).$$

{\bf Subcase (iii).} We are left to consider the situation that one
of the $\frac{q_1}{p_1},\ \frac{q_2}{p_2}$ is in $S^{(n-1)}$, but
another one is in $S^{(n)}- S^{(n-1)}$. Let us assume, for example,
$\frac{q_1}{p_1}=w^{(n-1)}_{j+1} \in S^{(n-1)}$. Then
$\frac{q_2}{p_2} < w^{(n-1)}_j = \frac{q}{p}\in S^{(n-1)}$. The
proof for the other case is similar. Notice that by the proof of
Claim A, we have $q_2=q_1+q, p_2=p_1+p$. By definition,
\begin{eqnarray*}
&&\mathscr{B}^{(n)}(B_i)=\{(r_i q_2-b_ip_2) \times (q_1,p_1),(-r_i
q_1+b_i
p_1) \times (q_2,p_2)\},\\
&&\mathscr{B}^{(n-1)}(B_i)=\{(r_i q-b_ip) \times (q_1,p_1),(-r_i
q_1+b_i p_1) \times (q,p)\}.
\end{eqnarray*}
Since $\mathscr{B}^{(n-1)}(q_2,p_2)=\{(q_1,p_1),(q,p)\}$, we get the
following by computation: {\small
\begin{eqnarray*}
\mathscr{B}^{(n-1)}(\mathscr{B}^{(n)}(B_i))&=& \{(r_i q_2-b_ip_2)
\times
(q_1,p_1)\}+\{(-r_i q_1+b_ip_1) \times (q_1,p_1),\\
&&(-r_i q_1+b_ip_1) \times (q,p)\}\\
&=&\{ (r_i q-b_ip) \times (q_1,p_1),(-r_i q_1+b_i p_1) \times
(q,p)\}.
\end{eqnarray*}}
So we can see
$\mathscr{B}^{(n-1)}(B_i)=\mathscr{B}^{(n-1)}(\mathscr{B}^{(n)}(B_i))$.
We are done.
\end{proof}

By Claim B, we have obtained a chain $\{\mathscr{B}^{(n)}(B)\}$ of
baskets with the following relation:
$$ \mathscr{B}^{(0)}(B)=\ldots=\mathscr{B}^{(4)}(B)
\succ \mathscr{B}^{(5)}(B) \succ ... \succ \mathscr{B}^{(n)}(B)
\succ ... \succ B. \eqno{(4.2)}$$ Clearly $B=\mathscr{B}^{(n)}(B)$
for some $n \gg 0$ for a given finite basket $B$. Thus, in some
sense, $B$ can be realized as the limit of the sequence
$\{\mathscr{B}^{(n)}(B)\}$.

Another direct consequence of Claim B is the property:
$$\mathscr{B}^{(i)}(\mathscr{B}^{(j)}(B))=\mathscr{B}^{(i)}(B) \eqno{(4.3)}$$
for $i\leq j$.

\begin{setup}{\bf The quantity $\epsilon_n(B)$.} Now let us consider the step
$\mathscr{B}^{(n-1)}(B) \succ \mathscr{B}^{(n)}(B)$. For an element
$w \in S^{(n)}$, let $m(w)$ be the number of basket $(b,r)$ in
$\mathscr{B}^{(n)}(B)$ with $b$ coprime to $r$ and $\frac{b}{r}=w$.
Thus we can write $\mathscr{B}^{(n)}(B)=\{m(w) \times
(b,r)\}_{w=\frac{b}{r} \in S^{(n)}}$.

Suppose that $S^{(n)}-S^{(n-1)}=\{\frac{j_s}{n}\}_{s=1,...,t}$.
We have $w^{(n-1)}_{i_s}=\frac{q_{i_s}}{p_{i_s}} > \frac{j_s}{n}
> w^{(n-1)}_{i_s+1}=\frac{q_{i_s+1}}{p_{i_s+1}} $ for some $i_s$.  We remark
that by the proof of Claim A, $j_s= q_{i_s}+q_{i_s+1}$, $n=
p_{i_s}+p_{i_s+1}.$ Since $\mathscr{B}^{(n-1)}(B) =
\mathscr{B}^{(n-1)}(\mathscr{B}^{(n)}(B))$ by Claim B, we may write
$$\mathscr{B}^{(n)}(B)=\{m(w) \times (b,r)\}_{w=\frac{b}{r} \in
S^{(n-1)}}+\{m(\frac{j_s}{n}) \times (j_s,n)\}_{\frac{j_s}{n}} ,$$
where "+" means collecting baskets of the same type. Then
\begin{eqnarray*}
\mathscr{B}^{(n-1)}(B)=&&\{m(w) \times (b,r)\}_{w=\frac{b}{r} \in
S^{(n-1)}}+\{m(\frac{j_s}{n}) \times (q_{i_s},p_{i_s}),\\
&&m(\frac{j_s}{n}) \times (q_{i_s+1},p_{i_s+1})\}_{\frac{j_s}{n}}.
\end{eqnarray*}

We define $\epsilon_n(B):= \sum_{s=1}^t m(\frac{j_s}{n}),$ which is
the number of type $(j_s,n)$ baskets with $\frac{j_s}{n} \not \in
S^{(n-1)} $. In other words, $\epsilon_n(B)$ counts the number of
those single baskets $(j_s,n)$ in $\mathscr{B}^{(n)}(B)$ with
$(j_s,n)=1$ and $j_s
>1$. This is going to be an important quantity in our arguments.
\end{setup}


\begin{setup} {\bf Notation.} When no confusion is likely, we will
simply write $B^{(n)}$ for $\mathscr{B}^{(n)}(B)$.

\end{setup}

\begin{lem}\label{delta} For the sequence $\{B^{(n)}\}$, the following statements are true:
\begin{itemize}
\item[(i)]
$\Delta^j(B^{(0)})= \Delta^j(B)$ for $j=3,4$;
\item[(ii)]
$\Delta^j(B^{(n-1)})= \Delta^j(B^{(n)})$ for all $j <n$;
\item[(iii)]
$\Delta^n(B^{(n-1)})= \Delta^n(B^{(n)})+\epsilon_n(B)$.
\item[(iv)]
$\Delta^n(B^{(n)})=\Delta^n(B)$.
\end{itemize}
\end{lem}

\begin{proof}
From $B^{(0)}$ to $B$, via $B^{(n)}$, the whole process can be
realized through a composition of finite number of convenient
packings. Each step is of the form $\{ (q_1,p_1),(q_2,p_2) \} \succ
\{(q_1+q_2,p_1+p_2)\}$. Notice that either
$\frac{q_1}{p_1},\frac{q_2}{p_2} \leq \frac{1}{3}$ or
$\frac{q_1}{p_1},\frac{q_2}{p_2} \geq \frac{1}{3}$. By Lemma
\ref{packing}(2), one gets $\Delta^3(B^{(0)})=\Delta^3(B)$. The
proof for $\Delta^4$ is similar.

We consider now the step $B^{(n-1)} \succ B^{(n)}$. A direct
computation shows that
$$\begin{array} {l}\Delta^n(B^{(n-1)})-\Delta^n( B^{(n)})\\
= \sum_{s=1}^t m(\frac{j_s}{n})
(\Delta^n_{q_{i_s},p_{i_s}}+\Delta^n_{q_{i_s+1},p_{i_s+1}}-\Delta^n_{j_s,n})\\
=\sum_{s=1}^t m(\frac{j_s}{n})
(\Delta^n_{q_{i_s},p_{i_s}}+\Delta^n_{q_{i_s+1},p_{i_s+1}}-\Delta^n_{q_{i_s}+q_{i_s+1},p_{i_s}+p_{i_s+1}})\\
 =\sum_{s=1}^t m(\frac{j_s}{n})\\
=\epsilon_n(B). \end{array}$$

Finally, for any $j <n$, and suppose that $\frac{k+1}{j} \ge
\frac{q_{i_s}}{p_{i_s}}=w^{(n-1)}_{i_s} > \frac{k}{j}$ for some
$k$. Then $\frac{k+1}{j} \in S^{(n-1)}$ by definition. Thus
$\frac{q_{i_s+1}}{p_{i_s+1}}=w^{(n-1)}_{i_s+1} \ge \frac{k}{j}$. By
Lemma \ref{packing}, we have
$$
\Delta^j_{q_{i_s},p_{i_s}}+\Delta^j_{q_{i_s+1},p_{i_s+1}}=\Delta^j_{q_{i_s}+q_{i_s+1},p_{i_s}+p_{i_s+1}}.$$

The last statement is due to (ii) and the fact that $B=B^{(n)}$ for
a sufficiently large $n$. This completes the proof.
\end{proof}

Let us go back to the sequence (4.2)
$$ B^{(0)} \succ B^{(5)} \succ ... \succ B^{(n)} \succ ... \succ B.$$
We see that $\Delta^j (B^{(n)})= \Delta^j (B)$ for all $j <n$. Thus
$B^{(n)}$ can be viewed as an approximation of $B$ of degree $n$.
Also each approximation step $B^{(n-1)} \succ B^{(n)}$ is nothing
but the convenient packings of $\epsilon_n$ pairs of baskets of
type $(b,n)$ with $b$ coprime to $n$, $b\leq \frac{r}{2}$ and
$b>1$.
\medskip

The whole strategy of our method is that, given a basket $B$, we can
almost determine $B^{(n)}$ (for small $n$) in terms of  $P_m$ and
$\chi(\OO_X)$. Then we are able to recover $B$ from $B^{(n)}$
because there are only finitely many baskets dominated by $B^{(n)}$.
Finally we check whether those recovered baskets satisfy some
geometric constrains. This works very effectively as seen in next
sections.

\section{\bf Formal baskets}
Given a minimal 3-fold $X$ of general type, there is an associated basket $B:={\mathscr{B}}(X)$.
In Section 4, we have defined  the
invariants $\sigma(B)$, $\sigma'(B)$ and $\Delta^m(B)$ which satisfy the
equalities $(\partial)$. The main purpose of our studying baskets is
to classify ${\mathscr{B}}(X)$.
To this end, we will study in a slightly general way. {}From now on
within this section, we assume
$$B=\{(b_i,r_i)|i=1,\cdots, t; b_i, r_i >0 ;
b_i\leq \frac{r_i}{2}\}.$$

\begin{defn} Any basket $B$ as above is called {\it a normal basket.}
\end{defn}

Notice that, by the equalities $(\partial)$, all $P_m$ are
determined by $\sigma, \sigma'-K^3, \chi, \Delta^j$ for all $j <m$.
These, in turn, are determined by $B, \chi$ and $P_2$ by virtue of
the first equality in $(\partial)$. This leads us to consider  a
more general setting.

\begin{defn} Assume that $B$ is a normal basket, $\tilde{\chi}$ and  $\tilde{P_2}$ ($\geq
0$) are integers. The triple  ${\bf B}:=\{B, \tilde{\chi},
\tilde{P_2}\}$ is called a {\it formal basket}.
\end{defn}

First we define $P_2({\bf B}):=\tilde{P}_2$, $$P_3({\bf
B}):=-\sigma(B)+ 10 \tilde{\chi}+5\tilde{P}_2$$ and the volume
\begin{eqnarray*}
K^3({\bf B})&:=&\sigma'({B})-4 \tilde{\chi} -3\tilde{P}_2+P_3({\bf
B})\\
&=& -\sigma+\sigma'+6\tilde{\chi}+2\tilde{P}_2.
\end{eqnarray*}
For $m \geq 4$, the plurigenus $P_m({\bf B})$  is defined
inductively by {\small
$$P_{m+1}({\bf B})-P_m({\bf B}):=
\frac{m^2}{2}(K^3({\bf B})-\sigma'({B}))-2\tilde{\chi}
+\frac{m}{2} \sigma({B})+\Delta^m(B).\eqno{(5.1)}$$}

Clearly, by definition, $P_m({\bf B})$ is an integer for all $m\geq
4$ because $K^3({\bf B})-\sigma'({B})=-4 \tilde{\chi}
-3\tilde{P}_2+P_3({\bf B})$ and $\sigma=10
\tilde{\chi}+5\tilde{P}_2-P_3({\bf B})$ have the same parity.
Sometimes we even use the notations $K^3(B)$ and $P_m(B)$ to denote
the volume and plurigenus.

Given a minimal 3-fold $X$, one can associate to $X$ a triple ${\bf
B}(X):=\{B,\tilde{\chi}, \tilde{P_2}\} $ where $B=\mathscr{B}(X)$,
$\tilde{\chi}=\chi(\mathcal{O}_X)$ and $\tilde{P_2}= P_2(X)$. It's
clear that such a triple is a formal basket.

\begin{defn} A formal basket ${\bf B}$ is said to be {\it positive} if
$K^3({\bf B}) >0$. ${\bf B}$ is called {\it admissible} if
$P_m({\bf B})\geq 0$ for all $m\geq 2$. ${\bf B}$ is said
to be {\it geometric} if ${\bf B}:={\bf B}(X)$ for some
minimal 3-fold $X$.
\end{defn}


Let $X$ be a minimal 3-fold of general type. Because $\chi(\OO_X)$
is an integer, $K^3({\bf B}(X))=K_X^3>0$ and $P_m({\bf
B}(X))=P_m(X)\geq 0$ for all $m\geq 2$. It's clear that the formal
basket ${\bf B }(X)$ is admissible and positive. Therefore  it is
sufficient for us to classify admissible and positive formal
baskets. Indeed, it's enough to consider admissible and positive
formal basket with some additionally imposed geometric conditions.

The point of view of packing baskets allows us to classify
admissible baskets in an effective way.
In what follows, we only consider packings in the approximation
consideration as in 4.12. Thus all packings are convenient unless otherwise
stated.

\begin{setup}{\bf Notations.}
We assume that ${\bf B}=\{B, \tilde{\chi}, \tilde{P}_2\}$ is an
admissible and positive formal basket. For simplicity, we denote
$P_{m}({\bf B})$ by $\tilde{P}_m$ for all $m\geq 4$. Also denote
$K^3({\bf B})$ by $\tilde{K}^3$, $\sigma=\sigma(B)$,
$\sigma'=\sigma'(B)$ and $\Delta^m=\Delta^m(B)$.
\end{setup}

In what follows, we would like to classify positive admissible
formal baskets with given datum
$(\tilde{\chi},\tilde{P_2},\tilde{P_3},...,\tilde{P_m})$.

First of all, by the definition of $\tilde{K}^3$ and
$\tilde{P}_m$, we get:
$$\begin{array}{rl} \tau:=\sigma'-\tilde{K}^3&=   4 \tilde{\chi}+3\tilde{P}_2-\tilde{P}_3,\\
\sigma &=  10 \tilde{\chi} +5 \tilde{P}_2-\tilde{P}_3\\
\Delta^3&=  5 \tilde{\chi} +6\tilde{P}_2-4\tilde{P}_3+\tilde{P}_4\\
\Delta^4&=  14 \tilde{\chi} +14\tilde{P}_2-6\tilde{P}_3-\tilde{P}_4+\tilde{P}_5\\
\Delta^5 &=27 \tilde{\chi} +25\tilde{P}_2- 10 \tilde{P}_3 - \tilde{P}_5 + \tilde{P}_6  \\
\Delta^6 &= 44 \tilde{\chi} +39\tilde{P}_2- 15 \tilde{P}_3- \tilde{P}_6+\tilde{P}_7    \\
\Delta^7&= 65 \tilde{\chi} +56\tilde{P}_2- 21 \tilde{P}_3 - \tilde{P}_7 + \tilde{P}_8 \\
\Delta^8&= 90 \tilde{\chi} +76\tilde{P}_2- 28 \tilde{P}_3 - \tilde{P}_8 + \tilde{P}_9 \\
\Delta^9&= 119 \tilde{\chi} +99\tilde{P}_2- 36 \tilde{P}_3  - \tilde{P}_9 +\tilde{P}_{10} \\
\Delta^{10}&= 152 \tilde{\chi} +125\tilde{P}_2- 45 \tilde{P}_3 - \tilde{P}_{10} + \tilde{P}_{11} \\
\Delta^{11}&= 189 \tilde{\chi} +154\tilde{P}_2- 55 \tilde{P}_3 - \tilde{P}_{11} +\tilde{P}_{12}  \\
\Delta^{12}&=  230 \tilde{\chi} +186\tilde{P}_2- 66 \tilde{P}_3- \tilde{P}_{12} +\tilde{P}_{13}  \\
\end{array}
$$
Recall that $B^{(0)}=\{n^0_{1,2} \times (1,2),\cdots,
n^0_{1,r} \times (1,r)\}$ is the initial basket of $B$.
Then by Lemma \ref{delta} and the definition of
$\sigma(B)$, we have
$$ \begin{array}{l} \sigma(B)= \sigma(B^{(0)})=\sum n^0_{1,r},\\
\Delta^3(B)= \Delta^3(B^{(0)})= n^0_{1,2} \\
\Delta^4(B)=\Delta^4(B^{(0)})=2 n^0_{1,2}+ n^0_{1,3}
\end{array}$$
Therefore, the initial basket has the coefficients:
$$B^{(0)} \left\{ \begin{array}{l}
n^0_{1,2}=5 \tilde{\chi}+ 6 \tilde{P}_2 -4 \tilde{P}_3 +\tilde{P}_4\\
n^0_{1,3}=4 \tilde{\chi} +2 \tilde{P}_2+2\tilde{P}_3-3 \tilde{P}_4 +\tilde{P}_5\\
n^0_{1,4}=\tilde{\chi} -3 \tilde{P}_2+\tilde{P}_3+2\tilde{P}_4-\tilde{P}_5- \sum_{r \ge 5} n^0_{1,r}\\
n^0_{1,r}=n^0_{1,r}, r \ge 5
\end{array} \right.
$$
By Lemma \ref{delta}, we see
\begin{eqnarray*}\epsilon_5&:=&\Delta^5(B^{(0)})-\Delta^{5}(B)=4
n^0_{1,2}+2n^0_{1,3}+n^0_{1,4}-\Delta^{5}(B)\\
&=&2 \tilde{\chi}- \tilde{P}_3 +2 \tilde{P}_5
-\tilde{P}_6-\sigma_5,
\end{eqnarray*}
where $\sigma_5:= \sum_{r \ge 5} n^0_{1,r}$. Thus we can
write {\footnotesize
$$B^{(5)}=\{ n^5_{1,2} \times (1,2),n^5_{2,5} \times (2,5), n^5_{1,3}
\times (1,3), n^5_{1,4}  \times (1,4), n^5_{1,5} \times
(1,5),\cdots\}$$} with
$$B^{(5)}\left\{
 \begin{array}{l}
  n^5_{1,2}=3 \tilde{\chi} +6\tilde{P}_2- 3 \tilde{P}_3 + \tilde{P}_4 - 2 \tilde{P}_5 + \tilde{P}_6+ \sigma_5 ,\\
 n^5_{2,5}= 2 \tilde{\chi}-\tilde{P}_3 +   2 \tilde{P}_5  - \tilde{P}_6- \sigma_5 \\
 n^5_{1,3}= 2 \tilde{\chi}+2\tilde{P}_2+ 3 \tilde{P}_3- 3 \tilde{P}_4 -\tilde{P}_5 + \tilde{P}_6   + \sigma_5 , \\
 n^5_{1,4}= \tilde{\chi} -3\tilde{P}_2 +  \tilde{P}_3 +2 \tilde{P}_4 -\tilde{P}_5 -\sigma_5 \\
  n^5_{1,r}=n^0_{1,r}, r \ge 5
\end{array} \right.$$ noting that this is obtained by
convenient packing $\{(1,2), (1,3)\}\succ \{(2,5)\}$.

Clearly, $B^{(5)}=B^{(6)}$ by our construction. Thus by Lemma
\ref{delta} we have
$\Delta^6(B^{(5)})=\Delta^6(B^{(6)})=\Delta^6(B)$. Computation
shows that
\begin{eqnarray*}
\Delta^6(B^{(5)})&=&6n^5_{1,2}+9n^5_{2,5}+3n^5_{1,3}+2n^5_{1,4}+n^5_{1,5}\\
&=&
44\tilde{\chi}+36\tilde{P}_2-16\tilde{P}_3+\tilde{P}_4+
\tilde{P}_5-\epsilon,
\end{eqnarray*}
where $$\epsilon:=n^0_{1,5}+2 \sum_{r \ge 6} n^0_{1,r} = 2
\sigma_5-n^0_{1,5}.$$ So we see
$$
\epsilon_6:=-3\tilde{P}_2-\tilde{P}_3+\tilde{P}_4+\tilde{P}_5+\tilde{P}_6-\tilde{P}_7-\epsilon=0.
\eqno(5.2)$$
Next, we compute
$$\begin{array}{ll}\epsilon_7:&=\Delta^7(B^{(6)})-\Delta^7(B)=\Delta^7(B^{(5)})-\Delta^7(B)\\
&=9n^5_{1,2}+13n^5_{2,5}+5n^5_{1,3}+3n^5_{1,4}+2n^5_{1,5}+n^5_{1,6}-\Delta^7(B)\\
&=\tilde{\chi}-\tilde{P}_2-\tilde{P}_3+\tilde{P}_6+\tilde{P}_7-\tilde{P}_8-2\sigma_5+2n^0_{1,5}+n^0_{1,6}.
\end{array}$$
Since $S^{(7)}-S^{(6)}=\{\frac{2}{7},\frac{3}{7}\}$, there are two
ways of packing into basket of type $(b,7)$. Let $\eta \geq 0$ be
the number of packing $\{(1,3),(1,4)\} \succ \{(2,7)\}$. Then
$\epsilon_7-\eta \ge 0$ is the number of packing $\{(1,2),(2,5)\}
\succ \{(3,7)\}$. Thus we can write $B^{(7)}=\{ n^7_{b,r} \times
(b,r)\}_{\frac{b}{r} \in S^{(7)}}$ with {\small
$$ B^{(7)}\left\{
\begin{array}{l}
n^7_{1,2} =2 \tilde{\chi} +7\tilde{P}_2 - 2 \tilde{P}_3 +\tilde{P}_4  - 2 \tilde{P}_5 - \tilde{P}_7+ \tilde{P}_8+3 \sigma_5-2n^0_{1,5}-n^0_{1,6} + \eta  \\
 n^7_{3,7} =\tilde{\chi}-\tilde{P}_2-\tilde{P}_3+\tilde{P}_6+\tilde{P}_7-\tilde{P}_8-2\sigma_5+2n^0_{1,5}+n^0_{1,6}- \eta\\
 n^7_{2,5} = \tilde{\chi} +\tilde{P}_2+ 2 \tilde{P}_5 - 2 \tilde{P}_6 - \tilde{P}_7+ \tilde{P}_8 + \sigma_5-2n^0_{1,5}-n^0_{1,6}  + \eta \\
 n^7_{1,3} =2 \tilde{\chi} +2\tilde{P}_2+ 3 \tilde{P}_3 - 3 \tilde{P}_4 -\tilde{P}_5 + \tilde{P}_6  + \sigma_5  - \eta\\
n^7_{2,7} = \eta\\
n^7_{1,4} = \tilde{\chi} -3\tilde{P}_2+ \tilde{P}_3 + 2 \tilde{P}_4 - \tilde{P}_5 - \sigma_5 - \eta\\
n^7_{1,r}=n^0_{1,r}, r \ge 5
\end{array} \right.$$}

From $B^{(7)}$, we can compute $\epsilon_8$ and then $B^{(8)}$, and
inductively $B^{(n)}$ for all $n\geq 9$. But notice that one can
even compute $\epsilon_9$, $\epsilon_{10}$ and $\epsilon_{12}$
directly from $B^{(7)}$, thanks to Lemma \ref{packing}.

To see this, let's consider
$\epsilon_9:=\Delta^9(B^{(8)})-\Delta^9(B)$ for example. Note that
$B^{(7)} \succ B^{(8)}$ is obtained by some convenient packing into
$\{(b,8)\}$, which is $\{(3,8)\}$. And every such packing, which is
$\{(2,5),(1,3)\} \succ \{(3,8)\}$, happens inside a closed interval
$[ \frac{3}{9}, \frac{4}{9}]$. Thus by Lemma \ref{packing}(2),
$\Delta^9(B^{(8)})=\Delta^9(B^{(7)})$. Similarly we can see
$\Delta^{10}(B^{(9)})=\Delta^{10}(B^{(7)})$ and
$\Delta^{12}(B^{(10)})=\Delta^{12}(B^{(7)})$. Unfortunately,
$B^{11}(B^{(10)})\neq \Delta^{11}(B^{(7)})$.

In summary, we have:{\small
$$ \begin{array}{lll}
\Delta^8(B^{(7)})& =&12 n^7_{1,2}+30 n^7_{3,7}+18 n^7_{2,5}+ 7
n^7_{1,3}+11
n^7_{2,7}+4 n^7_{1,4}\\
&&+3 n^7_{1,5}+2 n^7_{1,6}+n^7_{1,7} \\
&=& 90 \tilde{\chi}+74\tilde{P}_2 -29\tilde{P}_3-\tilde{P}_4+\tilde{P}_5+\tilde{P}_6 -3\sigma_5\\
&&+3 n^0_{1,5}+2 n^0_{1,6}+n^0_{1,7}; \\
\Delta^9(B^{(8)})& =&\Delta^9(B^{(7)})\\
& =&16 n^7_{1,2}+39 n^7_{3,7}+24 n^7_{2,5}+ 9 n^7_{1,3}+15
n^7_{2,7}+6 n^7_{1,4} \\
&&+4 n^7_{1,5}+3n^7_{1,6}+2n^7_{1,7}+n^7_{1,8} \\
&= &119\tilde{\chi}+97\tilde{P}_2 -38\tilde{P}_3+\tilde{P}_4+\tilde{P}_5-\tilde{P}_7+\tilde{P}_8-3 \sigma_5+\eta\\
&&+2 n^0_{1,5}+2n^0_{1,6}+2n^0_{1,7}+n^0_{1,8}; \\
\Delta^{10}(B^{(9)}) &=&\Delta^{10}(B^{(8)})=\Delta^{10}(B^{(7)})\\
& =&20 n^7_{1,2}+50 n^7_{3,7}+30 n^7_{2,5}+ 12 n^7_{1,3}+19
n^7_{2,7}+8 n^7_{1,4}\\
&&+5 n^7_{1,5}+4n^7_{1,6}+3n^7_{1,7}+2n^7_{1,8}+n^7_{1,9} \\
&=& 152\tilde{\chi} +120\tilde{P}_2-46\tilde{P}_3+2\tilde{P}_6-6\sigma_5-\eta \\
&& +5n^0_{1,5}+4n^0_{1,6}+3n^0_{1,7}+2n^0_{1,8}+n^0_{1,9};\\
\Delta^{12}(B^{(11)})&=&\Delta^{12}(B^{(10)})=...=\Delta^{12}(B^{(7)})\\
& =&30 n^7_{1,2}+75 n^7_{3,7}+46 n^7_{2,5}+ 18 n^7_{1,3}+30
n^7_{2,7}+12 n^7_{1,4}\\
&&+9 n^7_{1,5}+6n^7_{1,6}+5n^7_{1,7}+4n^7_{1,8}+3n^7_{1,9}+2n^7_{1,10}+n^7_{1,11} \\
&=& 229\tilde{\chi}+181\tilde{P}_2 -69\tilde{P}_3+2\tilde{P}_5+\tilde{P}_6-\tilde{P}_7+\tilde{P}_8-8 \sigma_5+\eta\\
&&+7 n^0_{1,5}+5n^0_{1,6}+5n^0_{1,7}+4n^0_{1,8}+3n^0_{1,9}+2n^0_{1,10}+n^0_{1,11}. \\
\end{array} $$}

We thus have:{\small
$$
\begin{array}{ll}
 \epsilon_{8} =& -2\tilde{P}_2-\tilde{P}_3-\tilde{P}_4+\tilde{P}_5+\tilde{P}_6+\tilde{P}_8-\tilde{P}_9-3\sigma_5\\
&+3 n^0_{1,5}+2 n^0_{1,6}+n^0_{1,7}; \\
\epsilon_{9} =&-2\tilde{P}_2-2\tilde{P}_3+\tilde{P}_4+\tilde{P}_5-\tilde{P}_7+\tilde{P}_8+\tilde{P}_9-\tilde{P}_{10} -3 \sigma_5+\eta\\
&+2 n^0_{1,5}+2n^0_{1,6}+2n^0_{1,7}+n^0_{1,8}; \\
 \epsilon_{10} =& -5\tilde{P}_2-\tilde{P}_3+2\tilde{P}_6+\tilde{P}_{10}-\tilde{P}_{11}-6\sigma_5 -\eta\\
 &+5n^0_{1,5}+4n^0_{1,6}+3n^0_{1,7}+2n^0_{1,8}+n^0_{1,9};\\
\epsilon_{12}=
&-\tilde{\chi}-5\tilde{P}_2-3\tilde{P}_3+2\tilde{P}_5+\tilde{P}_6-\tilde{P}_7+\tilde{P}_8+\tilde{P}_{12}-\tilde{P}_{13}-8
  \sigma_5+\eta;\\
  &+7 n^0_{1,5}+5n^0_{1,6}+5n^0_{1,7}+4n^0_{1,8}+3n^0_{1,9}+2n^0_{1,10}+n^0_{1,11}. \\
\end{array}
$$}\\
Since both $\epsilon_{10}$ and $\epsilon_{12}$ are
non-negative, we have $\epsilon_{10}+\epsilon_{12} \geq 0$.
This gives rise to:
$$2\tilde{P}_5+3\tilde{P}_6+\tilde{P}_8+\tilde{P}_{10}+\tilde{P}_{12} \ge \tilde{\chi} +10
\tilde{P}_2+4\tilde{P}_3+\tilde{P}_7+\tilde{P}_{11}+\tilde{P}_{13}+R,
\eqno{(5.3)}$$ where {\small
$$R:=14 \sigma_5-12 n^0_{1,5}-9
n^0_{1,6}-8n^0_{1,7}-6n^0_{1,8}-4n^0_{1,9}-2n^0_{1,10}-n^0_{1,11}$$
$$ =
2n^0_{1,5}+5n^0_{1,6}+6n^0_{1,7}+8n^0_{1,8}+10n^0_{1,9}+12n^0_{1,10}+13n^0_{1,11}+14
\sum_{r \ge 12} n^0_{1,r}.$$}

The equation $(5.2)$ and the inequality $(5.3)$ will play
important roles in the following classification.

In practice, we will frequently end up with situations satisfying
the following assumption and then our computation will be
comparatively simpler.

\begin{setup}\label{r6} {\bf Assumption.} $\tilde{P}_2=0$ and
$n^0_{1,r}=0$ for all $r \ge 6$.
\end{setup}

Under Assumption \ref{r6}, we list our datum in details as
follows. First,
$$\epsilon_7=\tilde{\chi}-\tilde{P}_3+\tilde{P}_6+
\tilde{P}_7-\tilde{P}_8
$$ and
$B^{(7)}=\{ n^7_{b,r} \times (b,r)\}_{\frac{b}{r}
\in S^{(7)}}$
 has coefficients:
$$B^{(7)} \left\{
\begin{array}{l}
n^7_{1,2} =2 \tilde{\chi}  - 2 \tilde{P}_3 +\tilde{P}_4  - 2 \tilde{P}_5 - \tilde{P}_7+ \tilde{P}_8+ n^0_{1,5} + \eta  \\
 n^7_{3,7} =\tilde{\chi}-\tilde{P}_3+\tilde{P}_6+\tilde{P}_7-\tilde{P}_8- \eta\\
 n^7_{2,5} = \tilde{\chi}+ 2 \tilde{P}_5 - 2 \tilde{P}_6 - \tilde{P}_7+ \tilde{P}_8 -n^0_{1,5}  + \eta \\
 n^7_{1,3} =2 \tilde{\chi} + 3 \tilde{P}_3 - 3 \tilde{P}_4 -\tilde{P}_5 + \tilde{P}_6  + n^0_{1,5}  - \eta\\
n^7_{2,7} = \eta\\
n^7_{1,4} = \tilde{\chi} + \tilde{P}_3 + 2 \tilde{P}_4 - \tilde{P}_5 - n^0_{1,5} - \eta\\
n^7_{1,5}=n^0_{1,5}.
\end{array} \right.$$
We have already known
$$\epsilon_8 =
-\tilde{P}_3-\tilde{P}_4+\tilde{P}_5+\tilde{P}_6+\tilde{P}_8-\tilde{P}_9.$$
Thus, taking some convenient packing into account,
$B^{(8)}=\{ n^8_{b,r} \times (b,r)\}_{\frac{b}{r} \in
S^{(8)}}$ has the coefficients:
$$B^{(8)} \left\{
\begin{array}{l}
n^8_{1,2} =2 \tilde{\chi}  - 2 \tilde{P}_3 +\tilde{P}_4  - 2 \tilde{P}_5 - \tilde{P}_7+ \tilde{P}_8+n^0_{1,5} + \eta  \\
 n^8_{3,7} =\tilde{\chi}  -\tilde{P}_3 +  \tilde{P}_6 + \tilde{P}_7- \tilde{P}_8  - \eta\\
 n^8_{2,5} = \tilde{\chi} +\tilde{P}_3+\tilde{P}_4+  \tilde{P}_5 - 3 \tilde{P}_6 - \tilde{P}_7+ \tilde{P}_9- n^0_{1,5}  + \eta \\
 n^8_{3,8}=-\tilde{P}_3-\tilde{P}_4+\tilde{P}_5+\tilde{P}_6+\tilde{P}_8-\tilde{P}_9\\
 n^8_{1,3} =2 \tilde{\chi} + 4 \tilde{P}_3 - 2 \tilde{P}_4 -2\tilde{P}_5 - \tilde{P}_8 +\tilde{P}_9 + n^0_{1,5}  - \eta\\
n^8_{2,7} = \eta\\
n^8_{1,4} = \tilde{\chi} + \tilde{P}_3 + 2 \tilde{P}_4 - \tilde{P}_5 -n^0_{1,5} - \eta\\
n^8_{1,5}=n^0_{1,5}.
\end{array} \right.$$
We know that
$$ \epsilon_9=-2\tilde{P}_3+\tilde{P}_4+\tilde{P}_5
-\tilde{P}_7+\tilde{P}_8+\tilde{P}_9-\tilde{P}_{10}
-n^0_{1,5}+\eta.$$ Moreover
$S^{(9)}-S^{(8)}=\{\frac{4}{9},\frac{2}{9}\}$. Let $\zeta$ be the
number of packings $\{(1,2),(3,7)\} \succ \{(4,9)\}$, then the
number of $\{(1,4),(1,5)\} \succ \{(2,9)\}$ packings is
$\epsilon_9-\zeta$ . We can get $B^{(9)}$ consisting of the
following coefficients.
$$B^{(9)} \left\{
\begin{array}{l}
n^9_{1,2} =2 \tilde{\chi}  - 2 \tilde{P}_3 +\tilde{P}_4  - 2 \tilde{P}_5 - \tilde{P}_7+ \tilde{P}_8+n^0_{1,5} + \eta-\zeta  \\
n^9_{4,9}=\zeta \\
 n^9_{3,7} =\tilde{\chi}  -\tilde{P}_3 +  \tilde{P}_6 + \tilde{P}_7- \tilde{P}_8  - \eta-\zeta\\
 n^9_{2,5} = \tilde{\chi} +\tilde{P}_3+\tilde{P}_4+  \tilde{P}_5 - 3 \tilde{P}_6 - \tilde{P}_7+ \tilde{P}_9- n^0_{1,5}  + \eta \\
 n^9_{3,8}=-\tilde{P}_3-\tilde{P}_4+\tilde{P}_5+\tilde{P}_6+\tilde{P}_8-\tilde{P}_9\\
 n^9_{1,3} =2 \tilde{\chi} + 4 \tilde{P}_3 - 2 \tilde{P}_4 -2\tilde{P}_5 - \tilde{P}_8 +\tilde{P}_9 + n^0_{1,5}  - \eta\\
n^9_{2,7} = \eta\\
n^9_{1,4} = \tilde{\chi} + 3\tilde{P}_3 +  \tilde{P}_4 - 2\tilde{P}_5+\tilde{P}_7-\tilde{P}_8 -\tilde{P}_9+\tilde{P}_{10} - 2\eta+\zeta\\
n^9_{2,9}= -2\tilde{P}_3+\tilde{P}_4+\tilde{P}_5-\tilde{P}_7+\tilde{P}_8+\tilde{P}_9-\tilde{P}_{10} -n^0_{1,5}+\eta - \zeta\\
n^9_{1,5}=2\tilde{P}_3-\tilde{P}_4-
\tilde{P}_5+\tilde{P}_7- \tilde{P}_8 -\tilde{P}_9+
\tilde{P}_{10}+2n^0_{1,5}-\eta+\zeta
\end{array} \right.$$

One has $$ \epsilon_{10} =
-\tilde{P}_3+2\tilde{P}_6+\tilde{P}_{10}-\tilde{P}_{11}-n^0_{1,5}
-\eta$$ and then $B^{(10)}$ consists of following
coefficients:
$$B^{(10)} \left\{
\begin{array}{l}
n^{10}_{1,2} =2 \tilde{\chi}  - 2 \tilde{P}_3 +\tilde{P}_4  - 2 \tilde{P}_5 - \tilde{P}_7+ \tilde{P}_8+n^0_{1,5} + \eta-\zeta  \\
n^{10}_{4,9}=\zeta \\
 n^{10}_{3,7} =\tilde{\chi}  -\tilde{P}_3 +  \tilde{P}_6 + \tilde{P}_7- \tilde{P}_8  - \eta-\zeta\\
 n^{10}_{2,5} = \tilde{\chi} +\tilde{P}_3+\tilde{P}_4+  \tilde{P}_5 - 3 \tilde{P}_6 - \tilde{P}_7+ \tilde{P}_9- n^0_{1,5} + \eta \\
 n^{10}_{3,8}=-\tilde{P}_3-\tilde{P}_4+\tilde{P}_5+\tilde{P}_6+\tilde{P}_8-\tilde{P}_9\\
 n^{10}_{1,3} =2 \tilde{\chi} + 5 \tilde{P}_3 - 2 \tilde{P}_4 -2\tilde{P}_5 -2\tilde{P}_6- \tilde{P}_8 +\tilde{P}_9 -\tilde{P}_{10}+\tilde{P}_{11}+ 2n^0_{1,5} \\
 n^{10}_{3,10}=-\tilde{P}_3+2\tilde{P}_6+\tilde{P}_{10}-\tilde{P}_{11} -n^0_{1,5}-\eta \\
n^{10}_{2,7} = \tilde{P}_3-2\tilde{P}_6-\tilde{P}_{10}+\tilde{P}_{11} +n^0_{1,5}+2\eta\\
n^{10}_{1,4} = \tilde{\chi} + 3\tilde{P}_3 +\tilde{P}_4 -2\tilde{P}_5 +\tilde{P}_7-\tilde{P}_8-\tilde{P}_9+\tilde{P}_{10} -2\eta + \zeta\\
n^{10}_{2,9}= -2\tilde{P}_3+\tilde{P}_4+\tilde{P}_5-\tilde{P}_7+\tilde{P}_8+\tilde{P}_9-\tilde{P}_{10} -n^0_{1,5}+\eta - \zeta\\
n^{10}_{1,5}=2\tilde{P}_3-\tilde{P}_4-
\tilde{P}_5+\tilde{P}_7- \tilde{P}_8 -\tilde{P}_9+
\tilde{P}_{10}+2n^0_{1,5}-\eta+\zeta
\end{array} \right.$$

By computing $\Delta^{11}(B^{(10)})$, we get
$$ \epsilon_{11}= \tilde{\chi}-\tilde{P}_3+\tilde{P}_4-\tilde{P}_7+\tilde{P}_9+\tilde{P}_{11}-\tilde{P}_{12}-n^0_{1,5}-\zeta.$$
Let $\alpha$ be the number of packing $\{(1,2),(4,9)\} \succ
\{(5,11)\}$ and $\beta$ be the number of packing $\{(1,3),(3,8)\}
\succ \{(4,11)\}$. Then we get $B^{(11)}$ with{\footnotesize
$$B^{(11)}\left\{
\begin{array}{l}
 n^{11}_{1,2} =  2 \tilde{\chi}- 2 \tilde{P}_3+\tilde{P}_4- 2\tilde{P}_5 - \tilde{P}_7+ \tilde{P}_8+ n^0_{1,5} + \eta   - \zeta  - \alpha \\
  n^{11}_{5,11} = \alpha \\
  n^{11}_{4,9} = \zeta - \alpha\\
  n^{11}_{3,7} = \tilde{\chi}- \tilde{P}_3 +\tilde{P}_6+  \tilde{P}_7 -  \tilde{P}_8-  \eta - \zeta     \\
    n^{11}_{2,5} =  \tilde{\chi} + \tilde{P}_3+ \tilde{P}_4+ \tilde{P}_5 - 3 \tilde{P}_6-\tilde{P}_7  + \tilde{P}_9  -n^0_{1,5}+ \eta \\
  n^{11}_{3,8} = - \tilde{P}_3-\tilde{P}_4  + \tilde{P}_5 + \tilde{P}_6 + \tilde{P}_8- \tilde{P}_9  - \beta\\
  n^{11}_{4,11} = \beta\\
  n^{11}_{1,3} = 2 \tilde{\chi} +5 \tilde{P}_3 - 2 \tilde{P}_4-2 \tilde{P}_5- 2 \tilde{P}_6 - \tilde{P}_8 + \tilde{P}_9- \tilde{P}_{10}+ \tilde{P}_{11}+ 2 n^0_{1,5}-\beta\\
n^{11}_{3,10}= -\tilde{P}_3 + 2 \tilde{P}_6 + \tilde{P}_{10}- \tilde{P}_{11} - n^0_{1,5} - \eta \\
  n^{11}_{2,7} = - \tilde{\chi}+ 2 \tilde{P}_3- \tilde{P}_4- 2 \tilde{P}_6 + \tilde{P}_7-\tilde{P}_9- \tilde{P}_{10}+ \tilde{P}_{12}+ 2 n^0_{1,5}+ 2\eta+\zeta+ \alpha + \beta\\
  n^{11}_{3,11} = \tilde{\chi} - \tilde{P}_3 + \tilde{P}_4 - \tilde{P}_7 + \tilde{P}_9 + \tilde{P}_{11}- \tilde{P}_{12}- n^0_{1,5} - \zeta   - \alpha - \beta\\
  n^{11}_{1,4} = 4 \tilde{P}_3- 2 \tilde{P}_5+2\tilde{P}_7-\tilde{P}_8-2\tilde{P}_9+ \tilde{P}_{10}-\tilde{P}_{11}+ \tilde{P}_{12}+n^0_{1,5}- 2\eta + 2\zeta + \alpha + \beta\\
 n^{11}_{2,9}=  - 2 \tilde{P}_3 + \tilde{P}_4+ \tilde{P}_5- \tilde{P}_7 + \tilde{P}_8+ \tilde{P}_9- \tilde{P}_{10}- n^0_{1,5}+\eta -\zeta\\
n^{11}_{1,5}=2\tilde{P}_3-
\tilde{P}_4-\tilde{P}_5+\tilde{P}_7- \tilde{P}_8
-\tilde{P}_9+ \tilde{P}_{10}+2n^0_{1,5}-\eta+\zeta
\end{array}
\right. $$}

Finally since $$
  \epsilon_{12}= -\tilde{\chi}-3\tilde{P}_3+2\tilde{P}_5+\tilde{P}_6
  -\tilde{P}_7+\tilde{P}_8+\tilde{P}_{12}-\tilde{P}_{13}-
  n_{1,5}^0+\eta.$$ we get $B^{(12)}$ with{\footnotesize
$$B^{(12)}\left\{
\begin{array}{l}
n^{12}_{1,2} =  2 \tilde{\chi}- 2 \tilde{P}_3+\tilde{P}_4- 2\tilde{P}_5 - \tilde{P}_7+ \tilde{P}_8+ n^0_{1,5} + \eta   - \zeta  - \alpha \\
  n^{12}_{5,11} = \alpha \\
  n^{12}_{4,9} = \zeta - \alpha\\
  n^{12}_{3,7} = 2 \tilde{\chi}+2 \tilde{P}_3 - 2 \tilde{P}_5+ 2 \tilde{P}_7 - 2 \tilde{P}_8-\tilde{P}_{12}+ \tilde{P}_{13}- 2 \eta - \zeta  + n^0_{1,5}   \\
  n^{12}_{5,12} =  -\tilde{\chi} - 3 \tilde{P}_3+ 2 \tilde{P}_5  + \tilde{P}_6 - \tilde{P}_7 + \tilde{P}_8 + \tilde{P}_{12}-\tilde{P}_{13}+ \eta -n^0_{1,5} \\
  n^{12}_{2,5} = 2 \tilde{\chi} +4 \tilde{P}_3+ \tilde{P}_4- \tilde{P}_5 - 4 \tilde{P}_6 - \tilde{P}_8 + \tilde{P}_9 - \tilde{P}_{12} + \tilde{P}_{13} \\
  n^{12}_{3,8} = - \tilde{P}_3-\tilde{P}_4  + \tilde{P}_5 + \tilde{P}_6 + \tilde{P}_8- \tilde{P}_9  - \beta\\
  n^{12}_{4,11} = \beta\\
  n^{12}_{1,3} = 2 \tilde{\chi} +5 \tilde{P}_3 - 2 \tilde{P}_4-2 \tilde{P}_5- 2 \tilde{P}_6 - \tilde{P}_8 + \tilde{P}_9- \tilde{P}_{10}+ \tilde{P}_{11}+ 2 n^0_{1,5}-\beta\\
n^{12}_{3,10}= -\tilde{P}_3 + 2 \tilde{P}_6 + \tilde{P}_{10}- \tilde{P}_{11} - n^0_{1,5}- \eta  \\
  n^{12}_{2,7} = - \tilde{\chi}+ 2 \tilde{P}_3- \tilde{P}_4- 2 \tilde{P}_6 + \tilde{P}_7-\tilde{P}_9- \tilde{P}_{10}+ \tilde{P}_{12}+ 2 n^0_{1,5}+ 2\eta+\zeta+ \alpha + \beta\\
  n^{12}_{3,11} = \tilde{\chi} - \tilde{P}_3 + \tilde{P}_4 - \tilde{P}_7 + \tilde{P}_9 + \tilde{P}_{11}- \tilde{P}_{12}- n^0_{1,5} - \zeta   - \alpha - \beta\\
  n^{12}_{1,4} = 4 \tilde{P}_3- 2 \tilde{P}_5+2\tilde{P}_7-\tilde{P}_8-2\tilde{P}_9+ \tilde{P}_{10}-\tilde{P}_{11}+ \tilde{P}_{12}+n^0_{1,5}- 2\eta + 2\zeta + \alpha + \beta\\
 n^{12}_{2,9}=  - 2 \tilde{P}_3 + \tilde{P}_4+ \tilde{P}_5- \tilde{P}_7 + \tilde{P}_8+ \tilde{P}_9- \tilde{P}_{10}- n^0_{1,5}+\eta -\zeta\\
n^{12}_{1,5}=2\tilde{P}_3- \tilde{P}_4-\tilde{P}_5+\tilde{P}_7-
\tilde{P}_8 -\tilde{P}_9+ \tilde{P}_{10}+2n^0_{1,5}-\eta+\zeta
\end{array}
\right. $$}

To recall the meaning of several symbols, $\eta$ is the
number of packing $\{(1,3),(1,4)\} \succ \{(2,7)\}$,
$\zeta$ is the number of packing $\{(1,2),(3,7)\} \succ
\{(4,9)\}$,  $\alpha$ is the number of packing
$\{(1,2),(4,9)\} \succ \{(5,11)\}$ and $\beta$ is the
number of packing $\{(1,3),(3,8)\} \succ \{(4,11)\}$.

\section{\bf Classification of baskets  with $\chi=1$}
Assume that $X$ is a minimal 3-fold of general type with
$\chi(\mathcal{O}_X) =1$. By Theorem \ref{volume}, the canonical
volume $K_X^3$ is bounded from below when $P_{m}\geq 2$ for some $m\leq 6$.
It is thus sufficient to consider geometric formal basket under the following:

\begin{setup}\label{p6}{\bf Assumption.} Assume that $B$ is
a geometric basket on $X$ and $P_m(X) \leq 1$ for $m \leq
6$.
\end{setup}

In fact, one has the following geometric condition.

\begin{lem}\label{p2m} Let $X$ be a minimal 3-fold of general type with
$\chi(\mathcal{O}_X) =1$. Then $P_{m+2}\geq P_m+P_2$ for all $m\geq
2$.
\end{lem}
\begin{proof} By Reid's formula (4.1), we have
$$P_{m+2}-P_m-P_2=(m^2+m)K_X^3-\chi(\OO_X)+(l(m+2)-l(m)-l(2)).$$
By \cite[Lemma 3.1]{Flt}, one sees $l(m+2)-l(m)-l(2)\geq 0$. Because
$K_X^3>0$ and $\chi(\OO_X)=1$, we have $P_{m+2}- P_m-P_2
> -1$. The Lemma now follows.
\end{proof}

We consider the formal basket
$${\bf B}:=\{B, \chi(\OO_X), P_2\}.$$
Because $B$ is geometric and $K^3({\bf B})=K_X^3>0$, the formal
basket ${\bf B}$ is admissible and positive. We may apply the
argument in Section 5. Again because $B$ is geometric, one sees
$\tilde{P}_m=P_{m}({\bf B})=P_m(X)$ for all $m\geq 2$ and
$\tilde{\chi}=\chi(\OO_X)=1$.

By Lemma \ref{p2m}, we see $P_{4}\geq 2$ if $P_2>0$. Thus
under Assumption \ref{p6}, we have $P_2=0$. We can also get
 $P_{m+2}>0$ whenever
$P_{m}>0$. Thus, in practice, we only need to study the
following types: $P_2=0$ and
\begin{eqnarray*}
(P_3,P_4,P_5,P_6)&=& (0,0,0,0),(0,0,0,1),(0,0,1,0),
(0,0,1,1),\\
&&(0,1,0,1),(0,1,1,1),(1,0,1,1),(1,1,1,1).
\end{eqnarray*}

If $P_2, P_3, P_4, P_5$ and $P_6$ are given, we are able to
determine $\mathscr{B}^{(5)}(B)$. Our main task is to
search all possible minimal (with regard to $\succ$)
positive baskets dominated by $\mathscr{B}^{(5)}(B)$.

\begin{setup}{\bf Notations and Conventions.}
Throughout this section, for a given basket $B$, we can consider the associated formal basket $\bar{B}:=\{B,1,0\}$.
We might abuse the notation of $B$ and its associated formal basket $\bar{B}$ in this section.

A basket with no further convenient packing is called {\it minimal}.
A positive basket with no further convenient packing into a positive
basket is called a {\it minimal positive basket}.

We let $m_0$ denote the minimal integer
such that $P_{m_0} \ge 2$ from now on.
\end{setup}

Now let us begin to classify all minimal positive geometric
baskets.

\begin{setup}\label{0000}{\bf The case  I: $P_3=P_4=P_5=P_6=0$.} \\
Now we have $\sigma=10, \tau=4, \Delta^3=5, \Delta^4=14,
\epsilon=0, \sigma_5=0$ and $\epsilon_5=2$. The only
possible initial basket is
 $\{5 \times (1,2), 4 \times (1,3), (1,4)\}$.
And $B^{(5)}=\{3\times (1,2), 2 \times (2,5), 2 \times
(1,3), (1,4)\}$ with $K^3=\frac{1}{60}$.

We now classify baskets $B$ with $\mathscr{B}^{(5)}(B)$ as above.
This basket can only be obtained by successive convenient packings.

If we pack $\{ (1,2),(2,5) \} $ to $\{(3,7)\}$. Then we get\\
{\bf I-1.} $\{  (2,4), (3,7),(2,5),  (2,6), (1,4)  \}$, $K^3=\frac{1}{420}$, $m_0=18$,   \\
which admits no further convenient packing into positive baskets.
Hence it's minimal positive.

Thus it remains to consider baskets with $(1,2)$ remained unpacked
because otherwise it's dominated by the basket of Case I-1.
So we consider the  packing:\\
$\{(3,6),(2,5),(3,8),(1,3),(1,4)\}$ with $K^3=\frac{1}{120}$. This
allows further packing to minimal positive ones:\\
{\bf I-2.} $\{ (3,6), (2,5),(4,11) , (1,4) \}$ , $K^3=\frac{1}{220}$ , $m_0=14$. \\
{\bf I-3.} $\{ (3,6), (5,13),(1,3) , (1,4) \}$,
$K^3=\frac{1}{156}$, $m_0=12$. \\
It remains to consider the case that both $(1,2)$ and $(2,5)$ are
remained unpacked. We can have the following packing which is
indeed minimal positive:\\
{\bf I-4.} $\{  (3,6),  (4,10),(1,3), (2,7) \}$ , $K^3=\frac{1}{210}$ , $m_0=14$. \\
This gives a complete list of minimal positive baskets satisfying
$P_2=...=P_6=0$.
\end{setup}

\begin{setup}\label{0001}{\bf The case  II: $P_3=P_4=P_5=0,P_6=1$.} \\
Now we have $\sigma=10, \tau=4, \Delta^3=5, \Delta^4=14,
\epsilon \le 1$. If $\epsilon=0$, then $\epsilon_5=1$ and
if $\epsilon=1$, then $\epsilon_5=0$. Thus the only
possible initial baskets and
$B^{(5)}$ are:\\
{\bf II-i.} $B^{(0)}=\{5 \times (1,2), 4 \times (1,3), (1,4)\} \succ
B^{(5)}=\{4 \times (1,2),(2,5), 3 \times (1,3),
 (1,4)\}$, with $K^3(B^{(5)})=\frac{1}{20}$.\\
{\bf II-ii.}  $B^{(0)}=\{5 \times (1,2), 4 \times (1,3), (1,5)\} \succ
B^{(5)}=\{5 \times (1,2), 4 \times (1,3), (1,5)\}$,with
$K^3(B^{(5)})=\frac{1}{30}$.

In Case II-i, we first consider that all the basket $(1,2)$
are packed with $(2,5)$ :\\
 $\{(6,13), 3 \times (1,3),(1,4)  \}$. Further packing gives
 a minimal positive basket:\\
{\bf II-1.} $\{(6,13), (1,3),  (3,10)  \}$ , $K^3=\frac{1}{390}$, $m_0=9$. \\
We then consider the case that at least one basket $(1,2)$ is
remained unpacked. Then we reach:\\
{\bf II-2.} $\{(1,2), (5,11),  (4,13)  \}$ ,
$K^3=\frac{1}{286}$, $m_0=9$, which is minimal positive.\\
Notice that if $\{3 \times (1,2),(3,7), 3 \times (1,3),
(1,4)\} \succ B$, then $B$ dominate the basket in Case
II-2. Thus it remains to consider the case that all basket
$(1,2)$ are remain unpacked and $(2,5)$ must be packed with
some $(1,3)$. So we
have minimal positive baskets:\\
{\bf II-3.} $\{  (4,8), (3,8),  (3,10)  \}$ , $K^3=\frac{1}{40}$, $m_0=8$. \\
{\bf II-4.} $\{  (4,8), (4,11), (2,7) \}$ , $K^3=\frac{2}{77}$, $m_0=8$.  \\
{\bf II-5.}  $\{ (4,8), (5,14), (1,4) \}$ , $K^3=\frac{1}{28}$, $m_0=8$. \\

In case II-ii, $B^{(5)}$ admits no further convenient packing.
Thus it's minimal positive.\\
{\bf II-6.} $\{ (5,10),(4,12),(1,5)\}$, $K^3=\frac{1}{30}$,
$m_0=8$.
\end{setup}

\begin{setup}\label{0010}{\bf The case  III: $P_3=P_4=0,P_5=1,P_6=0$.} \\
Now we have $\sigma=10, \tau=4, \Delta^3=5, \Delta^4=15$.
Moreover, $P_7 \ge 1$, hence $\epsilon=0, \sigma_5=0$ and
$\epsilon_5=4$. Thus the only possible initial baskets and
$B^{(5)}$ are:\\
$B^{(0)}=\{5 \times (1,2), 5 \times (1,3) \} \succ B^{(5)}=\{(1,2),4
\times (2,5), (1,3)\}$.\\
So the minimal positive baskets are:\\
{\bf III-1.} $\{(9,22),(1,3)\}$, $K^3= \frac{1}{66}$, $m_0=9$. \\
{\bf III-2.} $\{(7,17), (3,8)\}$, $K^3= \frac{1}{136}$, $m_0=10$. \\
{\bf III-3.} $\{(5,12), (5,13)\}$, $K^3= \frac{1}{156}$, $m_0=10$. \\
{\bf III-4.} $\{(3,7),(7,18) \}$, $K^3= \frac{1}{126}$, $m_0=10$. \\
{\bf III-5.} $\{(1,2),(9,23),\}$, $K^3= \frac{1}{46}$, $m_0=8$. \\
\end{setup}

\begin{setup}\label{0011}{\bf The case  IV: $P_3=P_4=0,P_5=1,P_6=1$.} \\
Now we have $\sigma=10, \tau=4, \Delta^3=5, \Delta^4=15$.
Moreover, the initial basket must have $n^0_{1,2}=n^0_{1,3}=5$,
hence $n^0_{1,r}=0$ for all $r \ge 4$. It follows that
$\epsilon=0$, $\sigma_5=0$ and $\epsilon_5=3$. Thus the
only possible initial baskets and
$B^{(5)}$ are:\\
$B^{(0)}=\{5 \times (1,2), 5 \times (1,3) \} \succ B^{(5)}=\{2 \times
(1,2),3
\times (2,5), 2 \times (1,3)\}$.\\
So the minimal positive baskets are:\\
{\bf IV-1.} $\{(8,19),(2,6)\} \succ \text{III}-1$. \\
{\bf IV-2.} $\{(6,14), (4,11)\} \succ \text{III}-4$. \\
{\bf IV-3.} $\{(4,9), (6,16)\} \succ \text{III}-2$. \\
{\bf IV-4.} $\{(2,4),(8,21) \} \succ \text{III}-5$. \\
Hence the lower bound of $K^3$ can only be better.
\end{setup}

\begin{setup}
\label{0101}{\bf The case  V: $P_3=0,P_4=1,P_5=0,P_6=1$.} \\
Now we have $\sigma=10, \tau=4, \Delta^3=6, \Delta^4=13$
and $\sigma_5 \le \epsilon \le 2$.
The initial baskets could be :\\
{\bf V-i.} $\{6 \times (1,2),  (1,3), 3 \times (1,4)\}$\\
{\bf V-ii.} $\{6 \times (1,2),  (1,3), 2 \times (1,4), (1,5)\}$\\
{\bf V-iii.} $\{6 \times (1,2),  (1,3),  (1,4),2 \times (1,5)\}$\\
{\bf V-iv.} $\{6 \times (1,2),  (1,3), 2 \times (1,4), (1,r)\}$ with $r \ge 6$\\
The Case V-iii, V-iv are impossible since $K^3 \le 0$. For Case
V-i, we have $\epsilon_5=1$ and for case V-ii, we have
$\epsilon_5=0$.
Hence $B^{(5)}$ could be \\
{\bf V-i.} $\{ (5,10), (2,5), (3,12)\}$\\
{\bf V-ii.} $\{ (6,12),  (1,3),  (2,8), (1,5)\}$\\
So the minimal positive baskets are:\\
{\bf V-1.} $\{(7,15),(3,12)\}$, $K^3= \frac{1}{60}$, $m_0=7$. \\
{\bf V-2.} $\{(6,12),(1,3),(3,13)\}$, $K^3= \frac{1}{39}$, $m_0=8$. \\
{\bf V-3.} $\{(6,12),(3,11),(1,5)\}$, $K^3= \frac{1}{55}$, $m_0=8$. \\
\end{setup}

\begin{setup}
\label{0111}{\bf The case  VI: $P_3=0,P_4=P_5=P_6=1$.} \\
Now we have $\sigma=10, \tau=4, \Delta^3=6, \Delta^4=14$.
Also $P_7 \ge 1$ and hence $\sigma_5 \le \epsilon \le 2$.
The initial baskets could be :\\
{\bf VI-i.} $\{6 \times (1,2),  2 \times (1,3), 2 \times (1,4)\}$\\
{\bf VI-ii.} $\{6 \times (1,2),  2 \times (1,3), (1,4), (1,5)\}$\\
{\bf VI-iii.} $\{6 \times (1,2), 2 \times  (1,3),  2 \times (1,5)\}$\\
{\bf VI-iv.} $\{6 \times (1,2), 2 \times  (1,3), (1,4), (1,r)\}$
with $r \ge 6$

Since there are only 2 baskets of $(1,3)$, we have
$\epsilon_5 =3-\sigma_5 \le 2$. Hence $\sigma_5 >0$ and
$\epsilon >0$. Therefore, case VI-i is impossible.

For Case VI-ii, $\epsilon_5=2$, hence\\
 {\bf VI-ii.} $B^{(5)}=\{ 4 \times (1,2),2 \times (2,5), (1,4),(1,5)
 \}$. \\
Similarly, one can compute the minimal positive baskets:\\
{\bf IV-1.} $\{ (1,2),(7,16), (2,9)\}$, $K^3=\frac{1}{144}$, $m_0=7$.\\
{\bf IV-2.} $\{ (6,13),(2,5), (2,9)\}$, $K^3=\frac{8}{585}$, $m_0=7$.\\
{\bf IV-3.} $\{ (8,18),(1,4), (1,5)\}$, $K^3=\frac{1}{180}$,$m_0=7$.\\
For Case VI-iii, $\epsilon_5=1$, hence\\
 {\bf VI-ii.} $B^{(5)}=\{ 5 \times (1,2), (2,5),(1,3) , 2 \times (1,5)
 \}$. \\
Then minimal positive baskets are:\\
{\bf IV-4.} $\{ (1,2),(6,13),(1,3),(2,10)\}$, $K^3=\frac{1}{390}$, $m_0=8$.\\
{\bf IV-5.} $\{ (5,10),(3,8), (2,10)\}$, $K^3=\frac{1}{40}$, $m_0=8$.\\

For Case VI-iv, $\epsilon_5=2$, hence\\
 {\bf VI-iv.} $B^{(5)}=\{ 4 \times (1,2),2 \times (2,5), (1,4),(1,r)
 \}$ with $r \ge 6$.  \\
 Since $K^3(B^{(5)})>0$, we must have
 $r=6$. Then the minimal positive basket is:\\
{\bf IV-6.} $\{ (3,6),(3,7),(2,5) (1,4),(1,6)\}$, $K^3=\frac{1}{420}$, $m_0=10$.\\
\end{setup}

\begin{setup}
\label{1011}{\bf The case  VII: $P_3=1,P_4=0,P_5=P_6=1$.} \\
Now we have $\sigma=9, \tau=3, \Delta^3=1, \Delta^4=9$.
Moreover, $P_7 \geq 1$ and hence $\epsilon=0$. It follows
that $\sigma_5=0$ and $\epsilon_5=2$. The
initial baskets is :\\
 $B^{(0)}=\{ (1,2), 7 \times (1,3),(1,4) \}$\\
 Note that there is only one basket of type $(1,2)$, while
 $\epsilon_5=2$ gives a contradiction since $n_{2,5}^5=\epsilon_5=2$.
\end{setup}

\begin{setup}
\label{1111}{\bf The case  VIII: $P_3=P_4=P_5=P_6=1$.} \\
Now we have $\sigma=9, \tau=3, \Delta^3=2, \Delta^4=8$.
Moreover, $P_7 \geq 1$ and then $\epsilon \le 1$. If
$\epsilon=1$, then $\sigma_5=1$ and $\epsilon_5=1$. If
$\epsilon=0$, then $\sigma_5=0$ and $\epsilon_5=2$.
The initial baskets and $B^{(5)}$ could be:\\
{\bf VIII-i.} $B^{(0)}=\{2 \times (1,2), 4 \times (1,3), 3 \times
(1,4)\} \succ B^{(5)}=\{2 \times (2,5), 2 \times (1,3), 3 \times
 (1,4)\}$, with $K^3(B^{(5)})=\frac{1}{60}$.\\
{\bf VIII-ii.}  $B^{(0)}=\{2 \times (1,2), 4 \times (1,3), 2
\times (1,4),(1,5)\} \succ B^{(5)}=\{(1,2)$, $(2,5), 3
\times (1,3), 2 \times (1,4), (1,5)\}$, with
$K^3(B^{(5)})=0$.

It's clear that Case VIII-ii is impossible since it is not
positive.

For Case VIII-i, we first consider
$\{(2,5),(3,8),(1,3),(4,12)\}$. It follows that it
dominates two possible minimal positive
baskets:\\
{\bf VIII-1.} $ \{(5,13),(1,3),(3,12)\}$, $K^3=\frac{1}{156}$, $m_0=7$.\\
{\bf VIII-2.} $ \{(2,5),(4,11),(3,12)\}$, $K^3=\frac{1}{220}$, $m_0=7$.\\

Now it remains to consider the case where $(2,5)$ remains
unpacked. We then consider $\{(4,10),(1,3),(2,7),(2,8)\}$
with $K^3=\frac{1}{210}$. It allows the following minimal
positive basket which is a one-step packing:
{\bf VIII-3.}$\{(4,10),(1,3),(3,11),(1,4)\}$, $K^3=\frac{1}{660}$, $m_0=7$. \\
\end{setup}

For the canonical volume, we conclude the following:

\begin{thm}\label{volume_chi=1}
Every minimal 3-fold $X$ of general type with
$\chi(\OO_X)=1$ has $K^3 \geq \frac{1}{420}$.
\end{thm}

\begin{proof}
By Theorem \ref{volume}, we may assume that $P_m \leq 1 $
for $m \le 6$. Take the geometric formal basket  ${\bf
B}(X)=\{B, 1, P_2\}$ with $B=\tilde{\mathscr{B}}(X)$. Then
by the definition of geometric basket, we have
$K_X^3=K^3({\bf B})$. By our classification in above, we
know $B\succ B_{min}$ for a minimal positive basket
$B_{min}$ of type I-1 through VIII-3. Because
$\sigma(B_{min})=\sigma(B)$ and so
$\sigma(B_{min})=10+5P_2-P_3$, we see that
$\bar{B}_{min}:=\{B_{min}, 1, P_2\}$ is a formal basket by
definition. Furthermore we see $P_3(\bar{B}_{min})=P_3(X)$.
Because Lemma \ref{packing}(3) says $\sigma'(B)\geq
\sigma'(B_{min})$, we see that $K_X^3=K^3({\bf B})\geq
K^3(\bar{B}_{min})$.

Now by above classification of minimal baskets, we have
$K^3(\bar{B}_{min}) \geq \frac{1}{420}$ unless $B_{min}$ is of
Case VIII-3. Notice that in Case VIII-3, $P_7(X) =2$ by
direct computation. Thus by Table A in Section 3, we have
$K_X^3=K^3({\bf B}) \geq \frac{5}{2408} > \frac{1}{660}$.
So this means that the type VIII-3 minimal basket is not
geometric. Notice that the type VIII-3 minimal basket is
obtained by one-step packing from
$$B_{210}:=\{(4,10),(1,3),(2,7),(2,8)\}$$ with
$K^3(B_{210})=\frac{1}{210}
> \frac{1}{420}$.  Furthermore, $B_{210}$ is the only intermediate basket.
 Thus either $B$ dominates a minimal basket of other types or
$B_{210}$. So at any case we have seen $K_X^3\geq \frac{1}{420}$.
\end{proof}

The proof of the last theorem gives the following:
\begin{cor} Let $X$ be a minimal 3-fold of general type.
Any geometric basket $B$ on $X$ either dominates a minimal
basket of type different from VIII-3  or dominates the
basket $$B_{210}:=\{(4,10),(1,3),(2,7),(2,8)\}.$$
\end{cor}
\medskip

The lower bound of $K^3$ in Theorem \ref{volume_chi=1} is
optimal. Iano-Fletcher has already found the following
example:
\begin{exmp} (\cite[p151, No.23]{C-R} ) The
canonical hypersurface $X_{46}\subset \bP(4,5,6,7,23)$ has 7
terminal quotient singularities and the canonical volume
$K^3_{X_{46}}=\frac{1}{420}$. Because $p_g(X_{46})=
q(X_{46})=h^2(\OO_{X_{46}})=0$, one sees $\chi(\OO_{X_{46}})=1$.
\end{exmp}

\section{\bf Classification of baskets with $\chi>1$}
In order to study the case with $\chi \geq 2$, we need to
go further to develop the basket packing technique.  We
always consider the formal basket ${\bf B}=\{B,\chi(\OO_X),
P_2\}$ which is admissible and positive. Therefore we
can apply our general theory in Section 5, just replacing
$\tilde{\chi}$,  and $\tilde{P}_2$ in Section
5 by $\chi(\OO_X)$ and $P_2(X)$ respectively. All
other symbols are also replaced accordingly.

By Theorem \ref{volume}, we only need to study varieties under the following:

\begin{setup} \label{p12<2}
{\bf Assumption.} Assume $P_m \leq 1$ for all $m \leq 12$.
\end{setup}

Recalling equation (5.2), we have:  $$
\epsilon_6:=-3P_2-P_3+P_4+P_5+P_6-P_7-\epsilon=0
$$
which is equivalent to
$$ P_4+P_5+P_6 = 3 P_2+P_3+P_7+\epsilon. \eqno{(7.1)}$$

Notice that, under Assumption \ref{p12<2} and if $P_2=1$, then by equation (7.1),
$P_4=P_5=P_6=1$ and
$P_3=P_7=\epsilon=0$. But this is impossible since $P_2=P_5=1$
implies $P_7 \geq 1$. Thus we may assume that $P_2=0$ in the
following discussion.

Assumption \ref{p12<2} allows us to compute $B^{(12)}$. One can
see that $\chi$ is bounded by $P_m$ for $m \leq 12$ by the
inequality $(5.3)$. Hence it is possible to classify all possible
basket $B^{(12)}$ under Assumption \ref{p12<2}.

Note that the equality (5.3) will be as:
$$2P_5+3P_6+P_8+P_{10}+P_{12} \ge \chi +10
P_2+4P_3+P_7+P_{11}+P_{13}+R, \eqno{(7.2)}$$ where {\small $$R:=14
\sigma_5-12 n^0_{1,5}-9
n^0_{1,6}-8n^0_{1,7}-6n^0_{1,8}-4n^0_{1,9}-2n^0_{1,10}-n^0_{1,11}$$
$$ =
2n^0_{1,5}+5n^0_{1,6}+6n^0_{1,7}+8n^0_{1,8}+10n^0_{1,9}+12n^0_{1,10}+13n^0_{1,11}+14
\sum_{r \ge 12} n^0_{1,r}.$$} and $\sigma_5=\sum_{r \geq 5}
n^0_{1,r}$.

With all these preparations, we are able to prove the following:
\begin{thm}
Let $X$ be a minimal 3-fold of general type with $\chi(\OO_X) \geq
2$. Then either $\chi(\OO_X) \leq 6$ or $P_m \geq 2$ for some $m
\leq 12$.
\end{thm}

\begin{proof}
If $P_m \le 1$ for all $m \le 12$, we have seen $P_2=0$. Then by
$(7.2)$, we get $8 \ge \chi=\chi(\OO_X)$. If $\chi =7 $ or $8$,
then $P_5=P_6=1$. It follows that $P_{11}=1$. Hence $8 \ge \chi +1$
gives $\chi =7$ and $P_8=1$ as well. Then $P_{13}=1$. This leads to
$8 \ge \chi+2 =9$, a contradiction.
\end{proof}

\begin{setup}\label{r>5}{\bf Verifying Assumption \ref{r6}}. Assume
$P_m\leq 1$ for all $m\leq 12$ and $\chi(\OO_X)\geq 2$. Then we have
seen $P_2=0$. We study $n_{1,r}^0$ when $r\geq 6$. If there exists a
number $r\geq 6$ such that $n^0_{1,r} \ne 0$, then $R \geq 5$ by the
definition of $R$. Now $(7.2)$ gives
$$8 \geq \chi+5 \geq 7.$$
This implies that  $P_5=P_6=1$. Hence $P_{11}=1$. Now $(7.2)$ reads $5+P_8+P_{10}+P_{12} \ge 8+P_7+P_{13}$.
One then has $P_8=P_{10}=P_{12}=1$ and $P_7=P_{13}=0$. This leads to a contradiction since $P_{13} \ge P_5 P_8=1$.
 So we conclude
$n^0_{1,r}=0$ for all $r\geq 6$. In other words, Assumption
\ref{r6} is satisfied.

This allows us to utilize those
classifications in the last part of Section 5.
\end{setup}

\begin{setup}{\bf Classifying admissible baskets under extra conditions.}
We hope to classify all positive admissible baskets
$\mathscr{B}^{(12)}(B)$ under Assumption \ref{p12<2} and
$\chi(\OO_X)\geq 2$. Note that, for all $0<m,n\leq 12$, and $m+n
\le 13$,
$$P_{m+n} \geq P_m P_n \eqno{(7.3)}$$
naturally holds since $P_m, P_n\leq 1$ for geometric formal
basket. Our main result is Table B which is a complete list
of all possibilities of $B^{(12)}$ and can be obtained by a
simple computer program, or even by a direct, but time
consuming calculation.

In fact, first we preset $P_m=0,1$ for $m=3,\cdots, 11$. Then
$\epsilon_6=0$ gives the value of $\epsilon$. So we know the value
of $n_{1,5}^0$. By the inequality (5.2) we get the upper bound of
$\chi$ since $P_{13}\geq 0$. Since $n_{1,4}^7\geq 0$, we get the
upper bound of $\eta$. Similarly $n^9_{2,9}\geq 0$ gives the upper
bound of $\zeta$. Also $n_{4,9}^{11}\geq 0$ yields $\alpha\leq
\zeta$. Finally $n_{3,8}^{11}\geq 0$ gives the upper bound of
$\beta$. Now we set $P_{12}=0,1$. Then the inequality (5.3) again
gives the upper bound of $P_{13}$, noting that $\chi\geq 2$. Clearly
there are  only finite many solutions. With inequality (7.3) imposed
we can get only about 80 cases. An important relation to recall is
$B^{(12)}\succ B$. So we see $K^3(B^{(12)})\geq K^3(B)=K_X^3>0$. The
final imposed property eventually outputs 63 cases which is exactly
Table B.

All minimal positive baskets dominated by ${B}^{(12)}$ are also
collected in Table B.
\medskip

If one would like to take a direct calculation by hand, it
is of course possible. Consider no.2 case in Table B as an
example. Those formulae in Section 5 gives us enough
information to compute $B^{(12)}$. Because $P_2=0$,
$P_3=\cdots =P_7=0$, $P_8=1$ and $P_9=P_{10} =P_{11}=0$,
(7.1) tells $\epsilon=0$ and thus $\sigma_5=0$, which means
$R=0$. (7.2) gives $P_{12}+1\geq \chi+P_{13}\geq 2$. So
$P_{12}=1$, $\chi=2$ and $P_{13}=0$. Now the formula for
$\epsilon_{10}$ gives $\epsilon_{10}=-\eta\geq 0$, which
means $\eta=0$. Similarly $n^9_{1,5}=\zeta-1\geq 0$. On the
other hand, $n_{3,7}^9=1-\zeta\geq 0$. Thus $\zeta=1$. Now
$n_{4,9}^{11}=\zeta-\alpha\geq 0$ gives $\alpha \leq 1$.
$n^{11}_{3,11}=1-\zeta-\alpha-\beta\geq$ gives
$\alpha=\beta=0$. Finally we get
$$\{n_{1,2},n_{5,12},...,n_{1,5}\}=\{4,0,1,0,0,2,1,0,3,0,0,0,2,0,0\}$$
That is $B^{(12)}=\{4 \times(1,2), (4,9),2 \times
(2,5),(3,8),3 \times (1,3), 2 \times (1,4)\}$.
\newpage

\centerline{\bf Table B.}

{\scriptsize
$$
\begin{array}{lcccccccc}
no. & (P_3,...,P_{11}) &P_{18}&P_{24}&m_0 & \chi   & B^{(12)}=(n_{1,2},n_{5,11},...,n_{1,5}) \text{ or } B_{min} & K^3(B)\\
\hline

1 &(0,0,0, 0, 0, 0, 0, 1, 0) &4&8&14&  2& (5,0,0,1,0,3,0,0,3,0,0,1,0,0,0) & \frac{3}{770}\\
2 &( 0, 0, 0, 0, 0, 1, 0, 0, 0)&3&7&15&  2 &( 4, 0,1, 0,0, 2, 1,0, 3, 0, 0,0, 2,  0, 0) & \frac{1}{360} \\
2a&&2&3 &18&& \{(2,5),(3,8),*\} \succ\{(5,13),* \} & \frac{1}{1170}\\
3&( 0, 0, 0, 0, 0, 1, 0, 1, 0 ) &3&7&15&  3&(6,1,0, 0,0,4, 1,0,4, 0, 1,0,2,0,0) & \frac{23}{9240} \\
3a&&2&3 &18&&\{(2,5),(3,8),*\} \succ \{(5,13),* \} & \frac{17}{30030}\\
4&( 0, 0, 0, 0, 0, 1, 0, 1, 0 ) &4&9&14&  3&(7,0,1, 0,0,4, 0,1,3, 0, 1,0,2,0,0) & \frac{13}{3465} \\
4a&&1&2 &14&&\{(4,11),(2,6),*\} \succ \{(6,17),* \} & \frac{1}{5355}\\
5&( 0, 0, 0, 0, 0, 1, 0, 1, 0 ) &5&10&14&  3&(7,0,1, 0,0,4, 1,0,4, 0, 0,1,1,0,0) & \frac{17}{3960} \\
5a&&4&3&15&&\{(8,20),(3,8),*\} \succ \{(11,28),* \} & \frac{1}{1386}\\
5b&&3&3&15&& \{(5,13),(4,15),* \} & \frac{1}{1170}\\
6& (0, 0, 0, 1, 0, 0, 0, 1, 0 ) &3&6&14& 3 & (9,0,0, 2,0,1, 0,1,4, 0,2,0,0, 0, 1) & \frac{1}{462} \\
7& (0, 0, 0, 1, 0, 0, 1, 0, 0 ) &3&5&14& 2 & (5,0,1, 1,0,0, 0,0,5, 0,1,0,0, 0, 1) & \frac{1}{630} \\
7a&&2&3&14&&\{(4,9),(3,7),*\} \succ \{(7,16),* \} & \frac{1}{1680}\\
8& (0, 0, 0, 1, 0, 0, 1, 1, 0 ) &3&5&14& 3 & (7,1,0, 1,0,2, 0,0,6, 0,2,0,0, 0, 1) & \frac{1}{770} \\
9& (0, 0, 0, 1, 0, 1, 0, 0, 0 ) &2&2&14& 3 & (9,0,0, 2,0,0, 1,1,4, 0,1,0,0, 1, 0) & \frac{1}{5544} \\
10& (0, 0, 0, 1, 0, 1, 0, 0, 0 ) &3&6&14& 3 & (8,0,1, 1,0,0, 2,0,5, 0,1,0,1, 0, 1) & \frac{1}{630} \\
10a&&2&4&14&&\{(4,9),(3,7),*\} \succ \{(7,16),* \} & \frac{1}{1680}\\
11& (0, 0, 0, 1, 0, 1, 0, 1, 0 ) &2&4&14& 3 & (9,0,0, 2,0,0, 1,1,3, 1,0,0,1, 0, 1) & \frac{3}{3080}\\
11a&&2&3&14&&\{(3,8),(4,11),*\} \succ \{(7,19),* \} & \frac{1}{2660}\\
12& (0, 0, 0, 1, 0, 1, 0, 1, 0 ) &5&11&14& 3 & (9,0,1, 0,0,1, 2,0,4, 0,2,0,0, 0, 1) & \frac{1}{252} \\
12a&&4&6&14&&\{(2,5),(6,16),*\} \succ \{(8,21),* \} & \frac{1}{630}\\
13& (0, 0, 0, 1, 0, 1, 0, 1, 0 ) &3&4&14& 4 & (12,0,0, 2,0,2, 0,2,4, 0,2,0,0,1,0) & \frac{4}{3465} \\
14& (0, 0, 0, 1, 0, 1, 0, 1, 0 ) &3&6&14& 4 & (10,1,0, 1,0,2, 2,0,6,0,2,0,1, 0, 1) & \frac{1}{770} \\
15& (0, 0, 0, 1, 0, 1, 0, 1, 0 ) &4&8&14& 4 & (11,0,1, 1,0,2, 1,1,5,0,2,0,1, 0, 1) & \frac{71}{27720} \\
15a&&2&4&14&&\{(4,11),(1,3),*\} \succ \{(5,14),* \} & \frac{1}{2520}\\
15b&&3&4&14&&\{(2,5),(3,8),*\} \succ \{(5,13),* \} & \frac{23}{36036}\\
15c&&3&5&14&& \{(7,16),(7,19),* \} & \frac{31}{31920}\\
16& (0, 0, 0, 1, 0, 1, 0, 1, 0 ) &5&9&14& 4 & (11,0,1, 1,0,2, 2,0,6,0,1,1,0, 0, 1) & \frac{43}{13860} \\
16a&&4&3&14&&\{(4,10),(3,8),*\} \succ \{(7,18),* \} & \frac{1}{3080}\\
16b&&4&4&14&&\{(2,5),(6,16),*\} \succ \{(8,21),* \} & \frac{1}{1386}\\
16c&&3&3&14&& \{(7,16),(5,13),* \} & \frac{3}{16016}\\
17& (0, 0, 0, 1, 0, 1, 0, 1, 1 ) &3&6&14& 3 & (9,0,0, 2,0,0, 0,2,3, 0, 1,0,1,0, 1) &  \frac{3}{1540} \\
18& (0, 0, 0, 1, 0, 1, 0, 1, 1 ) &4&7&14& 3 & (9,0,0, 2,0,0, 1,1,4, 0, 0,1,0,0, 1) &  \frac{23}{9240} \\
18a&&2&3&14&&\{(4,11),(1,3),*\} \succ \{(5,14),* \} & \frac{1}{3080}\\
18b&&4&6&14&&\{(3,8),(4,11),*\} \succ \{(7,19),* \} & \frac{83}{43890}\\

19& (0, 0, 0, 1, 0, 1, 1, 0, 0 ) &3&3&14& 3 & (8,0,1,1,0,1,0,1,5,0,1,0,0,1,0) & \frac{2}{3465} \\
20&(0, 0, 0, 1, 0, 1, 1, 0, 0 ) &4&7&14& 3 & (7,0,2, 0,0,1, 1,0,6, 0,1,0,1,0, 1) & \frac{1}{504} \\
20a&&3&3&18&&\{(2,5),(3,8),*\} \succ \{(5,13),* \} & \frac{1}{16380}\\
21&( 0, 0, 0, 1, 0, 1, 1, 1, 0 ) &4&8&14& 2 & (6,0,1, 0,0,0, 1,0,3, 1,0,0,0, 0, 1 ) & \frac{1}{360} \\
21a&&2&3&16&&\{(1,3),(3,10),*\} \succ \{(4,13),* \} & \frac{1}{4680}\\
22&(0, 0, 0, 1, 0, 1, 1, 1, 0 ) &2&3&18& 3 & (7,1,0, 1,0,1, 1,0,5, 1,0,0,1,0,1) & \frac{1}{9240}\\
23&(0, 0, 0, 1, 0, 1, 1, 1, 0 ) &3&5&14& 3 & (8,0,1, 1,0,1, 0,1,4, 1,0,0,1,0,1) & \frac{19}{13860}\\
23a&&2&3&14&&\{(4,9),(3,7),*\} \succ \{(7,16),* \} & \frac{1}{2640}\\
24&(0, 0, 0, 1, 0, 1, 1, 1, 0 ) &3&3&14& 4 & (10,1,0, 1,0,3, 0,1,6, 0,2,0,0,1,0)& \frac{1}{3465} \\
25&(0, 0, 0, 1, 0, 1, 1, 1, 0 ) &4&7&14& 4 & (9,1,1, 0,0,3, 1,0,7, 0,2,0,1,0,1) & \frac{47}{27720} \\
25a&&4&6&14&&\{(5,11),(4,9),*\} \succ \{(9,20),* \} & \frac{1}{840}\\
26&(0, 0, 0, 1, 0, 1, 1, 1, 0, ) &5&9&14& 4 & (10,0,2, 0,0,3, 0,1,6,0,2,0,1,0,1) & \frac{41}{13860} \\
26a&&3&5&14&&\{(4,11),(1,3),*\} \succ \{(5,14),* \} & \frac{1}{1260}\\
27&(0, 0, 0, 1, 0, 1, 1, 1, 0 ) &6&10&14& 4 & (10,0,2, 0,0,3, 1,0,7, 0,1,1,0,0,1) & \frac{97}{27720} \\
27a&&5&3&14&&\{(6,15),(3,8),*\} \succ \{(9,23),* \} & \frac{19}{79695}\\
27b&&5&5&14&& \{(5,13),(5,18),* \} & \frac{1}{1170}\\
28&(0, 0, 0, 1, 0, 1, 1, 1, 1 ) &4&8&14& 2 & (5,1,0, 0,0,0, 1,0,4, 0,1,0,0,0,1) & \frac{23}{9240}\\
29&(0, 0, 0, 1, 0, 1, 1, 1, 1 ) &5&10&14& 2 & (6,0,1, 0,0,0, 0,1,3, 0,1,0,0,0,1) & \frac{13}{3465}\\
29a&&2&3&14&&\{(4,11),(2,6),*\} \succ \{(6,17),* \} & \frac{1}{5355}\\

\end{array}
$$

\newpage
$$
\begin{array}{lccccccc}
no. & (P_3,...,P_{11}) &P_{18}&P_{24}&m_0& \chi & (n_{1,2},n_{4,9},...,n_{1,5}) \text{ or } B_{min} & K^3(B)\\
\hline
30&(0, 0, 0, 1, 0, 1, 1, 1, 1 ) &3&5&14& 3 & (7,1,0, 1,0,1, 0,1,5, 0,1,0,1,0, 1) & \frac{1}{924} \\
31&(0, 0, 0, 1, 0, 1, 1, 1, 1 ) &4&6&14& 3 & (7,1,0, 1,0,1, 1,0,6, 0,0,1,0,0, 1) & \frac{1}{616} \\
32&(0, 0, 0, 1, 0, 1, 1, 1, 1 ) &5&8&14& 3 & (8,0,1, 1,0,1, 0,1,5, 0,0,1,0,0, 1) & \frac{2}{693} \\
32a&&4&6&14&&\{(4,9),(3,7),*\} \succ \{(7,16),* \} & \frac{1}{528}\\
32b&&2&2&14&&\{(4,11),(1,3),*\} \succ \{(5,14),* \} & \frac{1}{1386}\\
33 & (  0, 0, 0, 1, 1, 0, 0, 1, 0 ) &2&4&14 & 2 & (5,0, 0, 2,0,0,1,0,1,1,1,0, 0, 0, 0) & \frac{1}{840} \\
33a&&1&3&14&& \{(3,10),(2,7),*\} \succ \{(5,17),* \} & \frac{1}{2856}\\
34 & ( 0, 0, 0, 1, 1, 0, 0, 1, 0 ) &4&8& 14& 3 &( 7,0, 1, 1,0, 2, 1,0,3,0, 3,0,0,  0, 0) & \frac{1}{360} \\
34a&&3&6 &14&&\{(4,9),(3,7),*\} \succ \{(7,16),* \} & \frac{1}{560}\\
34b&&3& 4&14&&\{(2,5),(3,8),*\} \succ \{(5,13),* \} & \frac{1}{1170}\\
%
35& (  0, 0, 0, 1, 1, 0, 0, 1, 1 ) &3&6& 14& 2 &( 5,0, 0, 2,0,0,0,1,1,0,2,0,0, 0, 0) & \frac{1}{462}\\
36 & ( 0, 0, 0, 1, 1, 0, 1, 1, 0 ) &3&5&14& 2 & (4,0,1, 1,0,1, 0,0,2,1,1,0,0,  0, 0) & \frac{1}{630}\\
36a&&2&3&14&& \{(4,9),(3,7),*\} \succ\{(7,16),* \} & \frac{1}{1680}\\
36b&&2&4&14&&  \{(3,10),(2,7),*\} \succ\{(5,17),* \} & \frac{4}{5355}\\

37 & ( 0, 0, 0, 1, 1, 0, 1, 1, 0 ) &5&9&14 & 3 & (6,0,2, 0,0,3, 0,0,4, 0,3,0,0, 0, 0)&  \frac{1}{315} \\

38 & ( 0, 0, 0, 1, 1, 0, 1, 1, 1 ) &3&5&14& 2 &( 3,1,0, 1,0,1, 0,0,3,0,2,0,0, 0, 0) & \frac{1}{770}\\
39 & ( 0, 0, 0, 1, 1, 1, 0, 1, 0) &3&6&14 & 3 & (7,0,1, 1,0,1, 2,0,2,1,1,0,1,  0, 0) & \frac{1}{630} \\
39a&&2&4 &14&&\{(4,9),(3,7),*\} \succ \{(7,16),* \} & \frac{1}{1680}\\
39b&&2&5 &14&& \{(3,10),(2,7),*\} \succ \{(5,17),* \} & \frac{4}{5355}\\
40 & ( 0, 0, 0, 1, 1, 1, 0, 1, 0) &5&10&14 & 4 & (9,0,2, 0,0,3, 2,0,4,0,3,0,1,  0, 0)& \frac{1}{315} \\
40a&&4&4 &14&&\{(4,10),(3,8),*\} \succ \{(7,18),* \} & \frac{1}{2520}\\
40b&&4&5 &14&&\{(2,5),(6,16),*\} \succ \{(8,21),* \} & \frac{1}{1260}\\
41& ( 0, 0, 0, 1, 1, 1, 0, 1, 1 ) & 5 & 11& 13& 2& (5,0, 1, 0,0,
0,2,0, 1, 0, 2,0, 0, 0, 0) & \frac{1}{252} \\
42&  (0, 0, 0, 1, 1, 1, 0, 1, 1) &3&6&14 & 3 &   (6,1,0, 1,0,1, 2,0,3,0,2,0,1, 0, 0) & \frac{1}{770} \\
43 &  (0, 0, 0, 1, 1, 1, 0, 1, 1) &4&8&14 & 3 &   (7,0,1, 1,0,1,1,1,2,0,2,0,1, 0, 0) & \frac{71}{27720} \\
43a&&2&4&14&& \{(4,11),(1,3),*\} \succ\{(5,14),* \} & \frac{1}{2520}\\
43b&&3&4&14&&\{(2,5),(3,8),*\} \succ \{(5,13),* \} & \frac{23}{36036}\\
43c&&3&5&14&& \{(7,16),(7,19),* \} & \frac{31}{31920}\\
44 & ( 0, 0, 0, 1, 1, 1, 0, 1, 1) &5&9&14 & 3 & (7,0,1, 1,0,1, 2,0,3,0,1,1,0, 0, 0) & \frac{43}{13860} \\
44a&&4&4&14&& \{(2,5),(6,16),*\} \succ \{(8,21),* \} & \frac{1}{1386}\\
44b&&3&3&14&& \{(7,16),(5,13),* \} & \frac{3}{16016}\\
44c&&4&6&14&& \{(7,16),(5,18),* \} & \frac{1}{720}\\
44d&&4&4&14&& \{(5,13),(5,18),* \} & \frac{1}{2184}\\
45&  (0, 0, 0, 1, 1, 1, 1, 0, 1) &4&7&14& 2 & (3,0,2, 0,0,0, 1,0,3,0,1,0,1, 0, 0) & \frac{1}{504} \\
46& ( 0, 0, 0, 1, 1, 1, 1, 1, 0) &4&7&14& 3 &( 6,0,2, 0,0,2, 1,0,3,1,1,0,1, 0, 0)& \frac{1}{504} \\
46a&&3&3&16&&\{(2,5),(3,8),*\} \succ \{(5,13),* \} & \frac{1}{16380}\\
46b&&3&6&14&& \{(3,10),(2,7),*\} \succ \{(5,17),* \} & \frac{7}{6120}\\
47&  0, 0, 0, 1, 1, 1, 1, 1, 1) &2&3&16&  2 &  (3,1,0, 1,0,0, 1,0,2,1,0,0,1, 0,0 )& \frac{1}{9240} \\
48 &  0, 0, 0, 1, 1, 1, 1, 1, 1) &3&5&14&  2 & (4,0,1, 1,0,0, 0,1,1,1,0,0,1, 0,0 )& \frac{19}{13860} \\
48a&&2&3 &14&&\{(4,9),(3,7),*\} \succ \{(7,16),* \} & \frac{1}{2640}\\
49& (  0, 0, 0, 1, 1, 1, 1, 1, 1 ) &4&7&14& 3 & (5,1,1, 0,0,2, 1,0,4, 0,2,0,1,0,0) &\frac{47}{27720} \\
49a&&4&6&14&&\{(5,11),(4,9),*\} \succ \{(9,20),* \} & \frac{1}{840}\\
50& (  0, 0, 0, 1, 1, 1, 1, 1, 1 ) &5&9&14& 3 & (6,0,2, 0,0,2, 0,1,3, 0,2,0,1,0,0) &\frac{41}{13860} \\
50a&&3&5&14&&\{(4,11),(1,3),*\} \succ \{(5,14),* \} & \frac{1}{1260}\\
51& (  0, 0, 0, 1, 1, 1, 1, 1, 1 ) &6&10&14& 3 & (6,0,2, 0,0,2, 1,0,4, 0,1,1,0,0,0) &\frac{97}{27720} \\
51a&&5&4&14&&\{(4,10),(3,8),*\} \succ \{(7,18),* \} & \frac{1}{1386}\\
51b&&5&5&14&& \{(5,13),(5,18),* \} & \frac{1}{1170}\\

52&(0, 0, 1, 0, 0, 1, 0, 1, 0 ) &3&7&14& 2 & (4,0,0, 1,0,2, 2,0,2, 0,0,0,0,0, 1) & \frac{1}{420} \\
52a&&2&3&18&&\{(2,5),(3,8),*\} \succ \{(5,13),* \} & \frac{1}{2184}\\
53&(0, 0, 1, 0, 0, 1, 1, 1, 0 ) &4&8&14& 2 & (3,0,1, 0,0,3, 1,0,3, 0,0,0,0, 0, 1) & \frac{1}{360} \\
53a&&3&4&15&&\{(2,5),(3,8),*\} \succ \{(5,13),* \} & \frac{1}{1170}\\

54& (  0, 0, 1, 0, 1, 0, 0, 1, 0) &2&4&14& 2 &( 2,0,0, 2,0,3, 1,0,1, 0,1,0,0, 0, 0)&  \frac{1}{840} \\
55 & ( 0, 0, 1, 0, 1, 0, 0, 1, 0) &2&2&14& 3 & (4,0,0, 3,0,4, 1,0,3, 0,0,1,0, 0, 0) & \frac{1}{3080}\\
56 & ( 0, 0, 1, 0, 1, 0, 1, 1, 0 ) &3&5&14& 2 & (1,0,1, 1,0,4, 0,0,2, 0,1,0,0, 0, 0)& \frac{ 1}{630} \\
56a&&2&3&14&&\{(4,9),(3,7),*\} \succ \{(7,16),* \} & \frac{1}{1680}\\
\end{array}
$$

\newpage
$$
\begin{array}{lccccccc}
no. & (P_3,...,P_{11}) &P_{18}&P_{24}&m_0& \chi & (n_{1,2},n_{4,9},...,n_{1,5})\text{ or } B_{min} & K^3(B)\\
\hline

57& ( 0, 0, 1, 0, 1, 0, 1, 1, 0 ) &3&3&14& 3 & (3,0,1, 2,0,5, 0,0,4, 0,0,1,0, 0, 0) & \frac{1}{1386} \\
58& (0, 0, 1, 0, 1, 1, 0, 1, 0) &3&6&14& 3 & (4,0,1, 1,0,4, 2,0, 2, 0,1,0,1, 0, 0) & \frac{1}{630} \\
58a&&2&4&14&&\{(4,9),(3,7),*\} \succ \{(7,16),* \} & \frac{1}{1680}\\
59& (0, 0, 1, 0, 1, 1, 0, 1, 1) &2&4&14& 2 & (2,0,0, 2,0,2, 1,1,0, 0,0,0,1, 0, 0) &  \frac{3}{3080} \\
59a&&2&3&14&&\{(3,8),(4,11),*\} \succ \{(7,19),* \} & \frac{1}{2660}\\
60& (0, 0, 1, 0, 1, 1, 1, 1, 0) &4&7&14& 3 & (3,0, 2, 0,0,5, 1,0,3, 0,1,0,1, 0, 0 ) & \frac{1}{504} \\
60a&&3&3&15&&\{(2,5),(3,8),*\} \succ \{(5,13),* \} & \frac{1}{16380}\\
61& (0, 0, 1, 0, 1, 1, 1, 1, 1 ) &2&3&15& 2 & (0,1,0, 1,0, 3,1,0,2,0,0,0,1,0,0 )& \frac{1}{9240} \\
62& (0, 0, 1, 0, 1, 1, 1, 1, 1 ) &3&5&14& 2 & (1,0,1, 1,0,3, 0,1,1,0,0,0,1,0,0 )& \frac{19}{13860} \\
62a&&2&3&14&&\{(4,9),(3,7),*\} \succ \{(7,16),* \} & \frac{1}{2640}\\

63&(0, 0, 1, 1, 1, 1, 1, 1, 1 ) &3&4&14& 3 & (5,0,1, 2,0,1, 1,1,3, 0,1,0,0,0, 1) & \frac{1}{5544} \\
\end{array}
$$}

We see that there is only one packing $\{(2,5),(3,8)\} \succ
\{(5,13)\}$ to get a minimal positive basket $\{4 \times(1,2),
(4,9), (2,5),(5,13),3 \times (1,3), 2 \times (1,4)\}$. We simply
write this as $\{(5,13),*\}$ in Table B. It is now easy to calculate
$K^3$ for both $B^{(12)}$ and the minimal positive basket
$\{(5,13),*\}$. Finally we can directly calculate $P_{m}$. We use
$m_0$ to denote the minimal integer with $P_{m_0}\geq 2$. For the
need of our argument, we also display the value of
$P_{18}=P_{18}(X)$ and $P_{24}=P_{24}(X)$ in Table $B$. So
theoretically we can finish our classification by detailed
computations. We omit more details because all calculations are
similar.
\end{setup}

One will see that many positive minimal baskets in Table B are not
geometric. By Theorem \ref{volume} we know some effective lower
bounds of $K_X^3$. On the other hand, if we know a concrete $m_0$
and the volume of a minimal positive basket is smaller than the
lower bound predicted in Theorem \ref{volume}, then such a minimal
positive basket would not be a geometric one.





%

Looking through Table B, we have:
\medskip

\noindent{\bf Claim C}. {\em  Each minimal positive basket in Table
B, of cases 4a, 9, 16a, 16c, 18a, 20a, 21a, 22, 24, 27a, 29a, 33a,
44b, 46a, 47, 52a, 55, 60a, 61, 63 is not geometric.}

\begin{proof}
{\bf 1).} If $P_{14} \ge 2$, then $K^3 \ge \frac{11}{37800}
>\frac{1}{3437}$ by Theorem \ref{volume}. So the cases
4a, 9, 16c, 24, 27a, 29a, 44b, 63 are not geometric.

{\bf 2).} If $P_{15} \ge 2$, then $K^3 \ge \frac{11}{46080}
> \frac{1}{4190}$ by Theorem \ref{volume}, hence case 60a, 61 are not geometric.

{\bf 3).} If $P_{16} \ge 2$, then $K^3 \ge \frac{11}{55488} >
\frac{1}{5045}$ by Theorem \ref{volume}, hence the cases 46a, 47
are not geometric.

{\bf 4).} If $P_{18} \ge 2$, then $K^3 \ge
\frac{11}{77976}>\frac{1}{7089}$ by Theorem \ref{volume}. Thus the
cases 20a, 22 are not geometric.

{\bf 5).} The case 33a has $P_6=1,P_{16}=2$ but $P_{22}=1$, a
contradiction. So case 33a is not geometric.

{\bf 6).} The cases  16a, 18a, 21a, 52a, 55 have $P_{17}=0$.   In
case 21a, $P_8=P_9=1$, a contradiction. And in case 52a, 55,
$P_5=P_{12}=1$, a contradiction.
 For case 18a, $P_6=P_{11}=1$, again a contradiction. Finally in
case 16a, computation shows that $P_{19}=-1$. Hence each of these
cases is not geometric.
\end{proof}

\begin{thm}\label{vol}
The canonical volume $K^3 \ge \frac{1}{2660}$ for all projective
3-folds of general type.
\end{thm}

\begin{proof} It suffices to study minimal models. Let $X$ be a
minimal 3-fold $X$ of general type. If $P_{m} \ge 2$ for some $m \le
12$, then $K^3 \geq \frac{11}{24336}>\frac{1}{2213}$ by Theorem
\ref{volume}. Also notice that if $\chi(\mathcal{O}_X) \le 0$, then
$K^3\geq \frac{1}{30}$ by \cite[Theorem 1.1]{Jungkai-Meng}. When
$\chi=1$, we have seen that $K^3 \ge \frac{1}{420}$ by Theorem
\ref{volume_chi=1}. It remains to treat the case that
$\chi(\OO_X)>1$ and $P_m \leq 1$ for $m \leq 12$.

Recall that we can study a geometric basket $B$ on $X$ and the
corresponding formal basket ${\bf B}=\{B,\chi, P_2\}$. We have given
Table B for each possible minimal positive basket $B_{min}$
dominated by $B$. Clearly we have $K_X^3=K^3(B)\geq K^3(B_{min})$
since $B\succ B_{min}$ for some $B_{min}$ in Table B.

If $B_{min}$ is geometric, then by searching in Table B and
eliminate those non-geometric ones, we get $K^3 \ge
\frac{1}{2660}$, which happens in cases 11a and 59a.

We still need to treat the case that $B_{min}$ is non-geometric.
Notice that if $B^{(12)}$ is minimal and non-geometric. Then
$B^{(12)} \succ B \succ B_{min}$ gives $B^{(12)} = B$, which is
non-geometric. This is a contradiction.

Hence it remains to consider  those intermediate baskets between
$B^{(12)}$ and $B_{min}$  under the situation that $B^{(12)} \ne
B_{min}$ and $B_{min}$ is non-geometric.

Take case 4a for example. It's obtained by 2-steps packing
$$\{(2,6),(4,11)\} \succ \{(1,3),(5,14)\} \succ \{(6,17)\}.$$
It
doesn't really  matter whether the intermediate basket $B_{mid}$
is geometric or not. The intermediate basket has
$K^3(B_{mid})=\frac{1}{630} \geq \frac{1}{2660}$. Thus $K_X^3>
\frac{1}{2660}$.

One can see that the computation for case 29a is exactly the same.

Take case 33a for another example.  It's obtained by 1-step packing
$\{(3,10),(2,7)\}  \succ \{(5,17)\}$. Hence there is no intermediate
baskets. So $B= B^{(12)}$ and $K_X^3= K^3(B^{(12)})=\frac{1}{840}$.
Similar argument works for cases 18a, 20a, 21a, 46a,52a, 60a.

Indeed the remaining cases are  16a, 16c, 27a, 44b. In case 44b,
there are two intermediate baskets which dominates case 44c or 44d
respectively. Thus in particular $K_X^3> \frac{1}{2184}$. In case
27a, it's obtained from case 54 by 3-steps packing $\{3 \times
(2,5), (5,8),* \} \succ \{ 2 \times (2,5), (5,13),* \} \succ \{
(2,5),(7,18),* \} \succ \{ (9,23),*\}$. Every geometric basket
dominating the basket in case 54a must dominate the basket $\{
(2,5),(7,18),* \}$ with $K^3=\frac{1}{1386}$. Finally, we consider
cases 16a, 16c. The basket for case 16 is of the form
$\{(4,9),(3,7),(2,5),(2,5),(3,8),(3,8),*\}$. If we take 1-step
packing $\{(4,9),(3,7),(2,5),(5,13),(3,8),*\}$, then this is a
common intermediate basket $B_{mid}$ between $B^{(12)}$ and 16a or
16c. It has $K^3_{B_{mid}}=\frac{85}{72072}>\frac{1}{848}$. The only
remaining intermediate basket is
$\{(7,16),(2,5),(2,5),(3,8),(3,8),*\}$, which has the volume
$\frac{13}{6160}>\frac{1}{474}$.


Thus for all non-geometric basket cases, we still have
$K_X^3>\frac{1}{2660}$. We have proved the theorem.
\end{proof}

\section{\bf Birationality}
With the technique of studying pluricanonical maps and the
classification of baskets, we are able to study various
explicit birational geometry including plurigenera and the
pluricanonical birationality.

We will need the following:

\begin{lem}\label{p} Consider two formal baskets
${\bf B}_i=\{{ B_i}, \tilde{\chi}, \tilde{P}_2\}$ for $i=1,2$.
Assume ${ B_1} \succ {B_2}$ with $\mathscr{B}^{(0)}({
B_1})=\mathscr{B}^{(0)}({ B_2})$.  Then we have $P_m({\bf B}_1)
\geq P_m({\bf B}_2)$ for all $m\geq 2$.
\end{lem}

\begin{proof}
Since $\mathscr{B}^{(0)}({ B_1})=\mathscr{B}^{(0)}({ B_2})$, it
follows by definition that $\sigma({ B_1})=\sigma({ B_2})$ and
$\Delta^j({\bf B_1})=\Delta^j({\bf B_2})$ for $j=3,4$. Also we
know $K^3({\bf B}_1)-\sigma'({ B_1})=K^3({\bf B}_2)-\sigma'({
B_2})$. By Lemma \ref{packing}, we get $\Delta^m({\bf B_1})\geq
\Delta^m({\bf B_2})$ for all $m\geq 5$.

Therefore by the inductive definition of $P_m({\bf B}_i)$, we see
$P_m({\bf B}_1) \geq P_m({\bf B}_2)$ for all $m\geq 2$.
\end{proof}

Let us recall some known relevant results as follows. On irregular
3-folds there is already a practical result. The following theorem
was proved by the first author and C. D. Hacon.

\begin{thm}\label{q>0}  \cite{JC-H} Let $X$ be a minimal projective 3-fold of
general type with $q(X):=h^1(\OO_X)>0$. Then $P_m >0$ for
all $m \ge 2$ and $\varphi_m$ is birational for all $m\geq
7$.
\end{thm}

Therefore, we do not need to worry about irregular 3-folds in the
following discussion. The following result is due to  Koll\'ar.

\begin{thm}\label{Kollar} \cite[Corollary 4.8]{Kol}
Let $X$ be a minimal projective 3-fold of general type with
$P_{m_0}\geq 2$. Then $\varphi_{11m_0+5}$ is birational
onto its image.
\end{thm}

Koll\'ar's result was ever improved by the second author:

\begin{thm}\label{5k+6} \cite[Theorem 0.1]{JPAA} Let $X$ be a
minimal projective 3-fold of general type with $P_{m_0}\geq
2$. Then $\varphi_{m}$ is birational onto its image for all
$m\geq 5m_0+6$.
\end{thm}

When $\chi(\OO_X)<0$, Reid's formula (4.1) says $P_2\geq 4$
and $P_m > 0$ for all $m \ge 2$. So $\varphi_m$ is
birational for all $m\geq 16$ by Theorem \ref{5k+6}.

When $\chi(\OO_X)=0$, since one can verify $l_Q(3)\geq
l_Q(2)$ for any basket $Q$, Reid's formula (4.1) says:
$P_3(X)>P_2(X)>0$. Moreover, $P_{m+1} \ge P_m$ for all $m
\ge 2$.  Now $P_3(X)\geq 2$, so $\varphi_m$ is birational
for all $m\geq 21$ by Theorem \ref{5k+6}.

To make  a summary, we have the following result when $\chi \le 0$.
\begin{thm}\label{chi<1} let $X$ be a minimal projective 3-fold of
general type with $\chi(\OO_X)\leq 0$. Then

\begin{itemize}
\item[(1)]
$P_m >0$ for all $m \ge 2$;
\item[(2)]
$P_m \ge 2$ for all $m \ge 3$;
\item[(3)]
$\varphi_m$ is birational for all $m\geq 21$.
\end{itemize}
\end{thm}

\begin{rem} Under the same condition as that of Theorem
\ref{chi<1}, K. Zuo and the second author \cite{Chen-Zuo} have
actually proved that $\varphi_m$ is birational for all $m\geq 14$
(optimal). Since the mentioned paper is not published yet, we list
here Theorem \ref{chi<1} to make this paper more self-contained.
\end{rem}

Now we recall Fletcher's interesting result.
\begin{thm}\label{chi=1}(\cite{Flt}) Let $X$ be a minimal
projective 3-fold of general type with $\chi=1$. Then
$P_{12} \ge 1$, $P_{24} \ge 2$.
\end{thm}

We are able to prove a more general result for all 3-folds of
general type.

\begin{thm} \label{p12} Let $X$ be a minimal
projective 3-fold of general type. Then $P_{12}\geq 1$.
\end{thm}

\begin{proof} It suffices to prove this when $\chi \ge 2$ by Theorem \ref{chi<1} and \ref{chi=1}.
We assume $P_{12}=0$ and will deduce a contradiction. It's clear
that $P_2=P_3=P_4=P_6=0$. We consider the geometric formal basket
 ${\bf B}(X)=\{B, \chi, P_2, P_3\}$.
\medskip

\noindent {\bf Step 1.} If $P_5=0$, then the equality (5.2)
for $\epsilon_6$ gives $P_7=\epsilon=0$. This also means
$\sigma_5=0$. Hence Assumption \ref{r6} is satisfied. Now
since $\epsilon_7 \geq \eta$ and $\epsilon_{12} \ge 0$, one
gets
$$ \chi \ge P_8+\eta \ge \chi+P_{13}.$$
It follows that $\chi=P_8+\eta$, $\epsilon_7=\eta$ and
$n^7_{3,7}=0$. Since $n^9_{3,7}= -\zeta$, we have
$\zeta=0$. Now $n_{4,9}^{11}=\zeta-\alpha\geq 0$ gives
$\alpha=0$.

Hence since $n_{1,5}^0=0$ and so $n^9_{2,9}=-n^9_{1,5}=0$,
we have $n^9_{2,9}=0$ and $\epsilon_9=n^9_{2,9}+\zeta=0$
which gives $P_{10}=P_8+P_9+\eta$.

Now $n^{12}_{3,8}+n^{12}_{2,7} \ge 0$ gives $\eta \geq
\chi+3P_9 \ge \eta+P_8+3P_9$.
 Hence $P_8=P_9=0$, and also $P_{10}=\eta=\chi$.
However, $n^{12}_{3,8}+n^{12}_{1,4}=P_{10}-2\eta
-P_{11}=-\chi-P_{11}<0$, which is a contradiction.
\medskip

\noindent {\bf Step 2.} If $P_5 >0$, then we have $P_7=0$.
First of all, (5.2) gives $P_5=\epsilon:=n^0_{1,5}+2
\sum_{r \ge 6} n^0_{1,r}$. By definition of $\eta$, we see
$\epsilon_7 \geq \eta$. Because $\epsilon_{12} \ge 0$, we
get  $ \chi \ge P_8+ \eta +(2
\sigma_5-n^0_{1,5}-n^0_{1,6})$ and $ 2P_5+P_8+\eta \ge \chi
+P_{13} +(8
\sigma_5-7n^0_{1,5}-5n^0_{1,6}-5n^0_{1,7}-...-n^0_{1,11})$.
Combine these two inequalities, we get $$2 \epsilon
+P_8+\eta = 2P_5+P_8+\eta\geq P_8 +P_{13}+\eta + R',
\eqno(8.1)$$ where $R'=
2n^0_{1,5}+4n^0_{1,6}+5n^0_{1,7}+...+9 n^0_{1,11}+10
\sum_{r \ge 12} n^0_{1,r}\geq 0$. Notice that and $R' \ge 2
\epsilon=2 P_5$. It follows that $P_{13}=0$ and
$n^0_{1,r}=0$ for all $ r \ge 7$. Note also that $P_{13}=0$
and $P_{5}>0$ implies $P_8=0$.
\medskip

\noindent{\bf Step 3.} We summarize that $\sigma_5=n^0_{1,5}+n^0_{1,6}$ and $P_5=n^0_{1,5}+2n^0_{1,6}$.
Now $\epsilon_7 \ge \eta$ and $\epsilon_{12} \ge 0$ reads
$$ \chi \ge P_8+\eta+n^0_{1,5}+n^0_{1,6} \ge \chi.$$
It follows that $\chi=P_8+\eta+\sigma_5$. We then look at $n^7_{1,4}=\chi-P_5-\sigma_5-\eta$.
$n^7_{1,4} \ge 0$ implies that $P_8 \ge P_5 >0$, a contradiction.
%
%

\end{proof}

\begin{thm} \label{p24}
Let $X$ be a minimal projective 3-fold of general type.
Then $P_{24}\geq 2$.
\end{thm}

\begin{proof}
It suffices to prove the theorem for $\chi \geq 2$ by Theorem
\ref{chi<1} and \ref{chi=1}. Also we only have to study the
situation with $q(X)=0$ by \cite{JC-H}.

Suppose that $P_m \leq 1$ for all $m \le 12$. Then by our
classification of baskets in Table B, we have $P_{24}(B_{min})\geq
2$ for all minimal positive baskets. Thus by Lemma \ref{p}, we
have $P_{24}=P_{24}({\bf B})\geq P_{24}(B_{min}) \ge 2$, where
${\bf B}$ is a geometric formal basket on $X$.

It remains to consider the case that $\chi \ge 2$ and $P_m \ge 2$ for some
$m \leq 12$. Clearly, we only need to consider the cases
with $m=5,\ 7,\ 9,\ 10,\ 11$.

In what follows, we assume $P_{24}=1$ and will deduce a
contradiction. By Theorem \ref{p12}, we may and do assume that
$P_{12}=1$. We will frequently use the following easy observation
frequently:
 $$P_{m+n} \geq P_m P_n\ \text{ if either } P_m \text{ or } P_n \leq 1
 \eqno(8.2).$$

\noindent{\bf Step 1.} Suppose $P_{m_0} \geq 2$ for certain
$m_0\leq 10$. We can study $\varphi_{m_0}$. As we know that,
because $\chi(\OO_X)>1$, $\varphi_{m_0}$ is of type III, II, and
I$_p$. By Proposition \ref{nonvanishing}, we see that $P_m\geq 2$
for all $m\geq 2m_0+3\geq 23$. In particular, $P_{24} \ge 2$.
\medskip

\noindent{\bf Step 2.} Suppose $P_{11} \geq 2$ and $P_m \leq 1$ for
all $m \leq 10$. Clearly, $P_2=0$ since $P_{24}\geq P_{2}P_{2\times
11}$. First of all, $\epsilon_{10} \geq 0 $ gives {\small
$$2P_6+P_{10} \ge
P_{11}+P_3+\eta+(n^0_{1,5}+2n^0_{1,6}+3n^0_{1,7}+4n^0_{1,8}+5n^0_{1,9}+6
\sum_{r \ge 10} n^0_{1,r}) \ge 2.\eqno{(8.3)}$$}

If $P_6=0$ then $P_{10} \ge 2$ which is absurd. Thus we may assume
that $P_6=1$. Also $P_7=0$ since $P_{24}\geq P_6P_7P_{11}$.

If $P_3=1$, then $\epsilon_{10}$ gives $2+P_{10} \ge
P_{11}+P_3 \ge 3$. Hence $P_{10} = 1$. But then $P_{24} \ge
P_{11} P_{13} \ge 2$, a contradiction. Thus we may assume
$P_3=0$.

Now we look at $(8.3)$ again. Since  $3 \ge 2P_6+P_{10} $, together
with $P_{11} \ge 2$, we have $n^0_{1,r}=0$ for all $r \ge 6$ and
$n^0_{1,5} \le 1$. Hence $\epsilon:=n^0_{1,5}+\sum_{r \ge 6} 2
n^0_{1,r} \le 1$. Recall that $\epsilon_6=0$, which gives
$1+P_4+P_5=P_3+P_7+\epsilon \le 1$. It follows that $P_4=P_5=0$ and
$\epsilon=1$.

We make a summary: $P_2=P_3=P_4=P_5=P_7=0$, $P_6=1$.

Now $\epsilon_{10}$ gives
$$3 \ge 2+P_{10} \ge P_{11}+1+\eta \ge 3.$$ We thus have
$P_{10}=1$, $P_{11}=2$ and $\eta=0$.

We then look at $\epsilon_{12}$ which says: $$ 1+P_8+P_{12} \ge
\chi+P_{13}+1 \ge3.$$ Thus $\chi=2$, $P_8=P_{12}=1$ and $P_{13}=0$.
Also $\epsilon_9 \ge 0$ gives $P_9=1$.

With all above information, we see $n^8_{2,5}=-1<0$ which is a
contradiction.

This completes the proof.
\end{proof}


\begin{thm}\label{18} Let $X$ be a minimal projective 3-fold of general type.
Then there exists an integer $m_0\leq 18$ such that
$P_{m_0}\geq 2$.
\end{thm}

\begin{proof}
If $\chi \le 0$, then this is clear by Theorem \ref{chi<1}.
If $\chi=1$, then either $P_m \ge 2$ for some $m \le 6$ or
the geometric basket ${\bf B} \succ B_{min}$ for some
$B_{min}$ in the list of Section 6. Note that for all
baskets in the list of Section 6, we have $P_{m_0}(B_{min})
\ge 2 $ for some $m_0 \le 18$. By Lemma \ref{p}, we have
$P_{m_0}(X)=P_{m_0}({\bf B}) \geq P_{m_0}(B_{min}) \ge 2$
for some $m_0 \le 18$.

If $\chi >1$, then either $P_{m} \geq 2$ for some $m\leq 12$ or $B
\succ B_{min}$ for some $B_{min}$ in Table, Section 7. Similarly,
for all minimal basket in table B, we can verify that there is some
$m_0 \le 18$ with $P_{m_0}(B_{min}) \ge 2 $. Hence again by Lemma
\ref{p}, we have $P_{m_0}(X) \ge 2$ for some $m_0 \le 18$.
\end{proof}

\begin{thm}
For all minimal projective folds of general type, $m_1 \le 27$.
That is $P_m
>0$ for all $m \ge 27$.
\end{thm}

\begin{proof}
If $\chi <1$, this is clear by Theorem \ref{chi<1}.

If $\chi=1$, then either $P_{m_0} \ge 2$ for some $m_0 \le
6$ or it's classified in Section 6. By Proposition
\ref{nonvanishing}, $m_1 \le 3m_0+4\leq 22$ if $P_{m_0} \ge
2$ for some $m_0 \le 6$. For those minimal baskets
$B_{min}$ classified in Section 6, direct computation shows
that $P_{m}(B_{min})>0$ for $m\geq 14$. Thus
$P_m(X)=P_m({\bf B})\geq 0$ by lemma \ref{p} for all $m\geq
14$.

If $\chi \ge 2$, then either $P_{m_0} \ge 2$ for some $m_0
\le 12$ or the minimal basket $B_{min}$ of the geometric
basket $B$ is classified in Section 7. Notice that when
$P_{m_0} \ge 2$ for some $m_0 \le 12$, the induced map
$\varphi_{m_0}$ can not be of type I$_{n}$ as we pointed
out in Remark \ref{chi(O)>1}. Therefore, by Proposition
\ref{nonvanishing}, $m_1 \le 2m_0+3\leq 27$ in this
situation. It remains to consider those baskets classified
in Section 7. A direct but tedious computation shows that
$P_{m}(B_{min})>0$ for $m\geq 24$. Thus $P_m(X)=P_m({\bf
B})>0$ by Lemma \ref{p} for all $m\geq 14$. That is $m_1
\le 24$.
\end{proof}

We now study the birationality of pluricanonical maps. The
general strategy is to find a small $m_0$ such that
$P_{m_0} \ge 2$ or $ \ge 3$. Then we study the map
$\varphi_{m_0}$ according to its type by the method in
Section 2. We keep the same notation as in Section 2.

We will not study $\varphi_{m_0}$ of type I$_q$ for the
results of the first author and C. Hacon \cite{JC-H} are
effective enough. We will not study $\varphi_{m_0}$ of type
I$_n$ either, because we can not improve Theorem 0.1 of the
second author \cite{JPAA} (cf. Theorem \ref{5k+6}) in this
situation.

The structures of the proof for various situations are the same.
However, we need to pick different linear systems in different
situations.

\begin{lem}\label{(i)} Let $X$ be a minimal projective 3-fold of
general type with $P_{m_0}\geq 2$. Keep the same notation
as in \ref{setup}. Then Assumptions \ref{assumptions}(3)
holds whenever $m\geq m_0+m_1$ unless the induced map is of
type I$_q$.
\end{lem}
\begin{proof} We consider the linear system
$|K_{X'}+\roundup{t\pi^*(K_X)}+M_{m_0}|\subset
|(m_0+t+1)K_{X'}|$ for any $t>0$. Because
$K_{X'}+\roundup{t\pi^*(K_X)}\geq (t+1)\pi^*(K_X)$, we see
that $K_{X'}+\roundup{t\pi^*(K_X)}$ is effective whenever
$t+1\geq m_1$.

Now Remark \ref{separate} says
$|K_{X'}+\roundup{t\pi^*(K_X)}+M_{m_0}|$ can separate
different generic irreducible elements $S$ except when
$\dim(B)=1$ and $b=g(B)>0$.
\end{proof}
\begin{prop}\label{III}  Let $X$ be a minimal projective 3-fold of
general type with $P_{m_0}\geq 2$. Suppose that the induced map
from $\varphi_{m_0}$ is  of type III. Then $\varphi_{m}$ is
birational for all $m \ge 3m_0+1$.
\end{prop}

\begin{proof}
Pick a general member $S$. Take $G:=S_{|S}$. Then $|G|$ is base
point free and $\Phi_{|G|}$ gives a generically finite map. So a
general member $C\in |G|$ is a smooth curve. Recall that we have
$p=1$. Because $m_0\pi^*(K_X)_{|S}\geq S_{|S}\sim C$, we can take
$\beta=\frac{1}{m_0}$. So far, by Proposition
\ref{nonvanishing}(i) and Lemma \ref{(i)}(i), Assumptions
\ref{assumptions}(1), (2) and (3) are satisfied for all $m\geq
3m_0$.

We verify Assumptions \ref{assumptions}(4) under the condition
$m\geq 3m_0$. Because
\begin{eqnarray*}
&& K_S+\roundup{(m-1)\pi^*(K_X)-S-\frac{1}{p}E_{m_0}'}_{|S}\\
&\geq & K_S+(m-1)\pi^*(K_X)_{|S}-(S+E_{m_0}')_{|S}\\
&=& K_S+(m-m_0-1)\pi^*(K_X)_{|S}\geq
(m-m_0)\pi^*(K_X)_{|S}+C.
\end{eqnarray*}
and $(m-m_0)\pi^*(K_X)_{|S}\geq 0$ for $m-m_0\geq 2m_0$ by
Proposition \ref{nonvanishing}(i), we see that
$|K_S+\roundup{(m-1)\pi^*(K_X)-S-\frac{1}{p}E_{m_0}'}_{|S}|$
separates different $C$ since $|C|$ is not composed with any
pencils. So Assumptions \ref{assumptions}(4) is also satisfied.

Now we begin to verify the numerical conditions for $\alpha$.
Because $|G|$ is not composed with any pencils, we see
$C^2=G^2\geq 2$. Then $\deg(K_C)=(K_S+C)\cdot C\geq
(\pi^*(K_X)_{|S}+2C)\cdot C>4$. The evenness of $\deg(K_C)$ says
$\deg(K_C)\geq 6$. If we take a sufficiently big $m$ such that
$\alpha>1$, then Remark \ref{weak} gives
$\xi\geq
\frac{6}{2m_0+1}$. Take $m\geq 3m_0+1$. Then
$\alpha=(m-2m_0-1)\xi\geq \frac{6m_0}{2m_0+1}>2$. Theorem
\ref{technical}(II) says that $\varphi_m$ is birational for
all $m\geq 3m_0+1$.
\end{proof}

\begin{prop}\label{II}
 Let $X$ be a minimal projective 3-fold of
general type with $P_{m_0}\geq 2$. Suppose that the induced map
from $\varphi_{m_0}$ is of type II. Then $\varphi_{m}$ is
birational for all $m \ge 4m_0+2$.
\end{prop}

\begin{proof}
Pick a general member $S$. Take $G:=S_{|S}$. Then $|G|$ is base
point free and $|G|$ is composed with a pencil of curves, namely
$G^2=0$. A generic irreducible element $C$ of $|G|$ is a smooth
curve of genus $\geq 2$. Recall that we have $p=1$. For the same
reason we can take $\beta=\frac{1}{m_0}$ since
$m_0\pi^*(K_X)_{|S}\geq G$. So far, by Proposition
\ref{nonvanishing}(ii) and Lemma \ref{(i)}(ii), Assumptions
\ref{assumptions}(1), (2) and (3) are satisfied for all $m\geq
3m_0$.

We verify Assumptions \ref{assumptions}(4).  As we have
seen in (i), there are the relations:
\begin{eqnarray*}
&&K_S+\roundup{(m-1)\pi^*(K_X)-S-\frac{1}{p}E_{m_0}'}_{|S}\\
&\geq&
\pi^*(K_X)_{|S}+G+(m-m_0-1)\pi^*(K_X)_{|S}\\
&=&(m-m_0)\pi^*(K_X)_{|S}+G.
\end{eqnarray*}
Whenever $|G|$ is composed with a rational pencil and $m-m_0\geq
m_1$, e.g. $m\geq 3m_0$, the linear system
$|K_S+\roundup{(m-1)\pi^*(K_X)-S-\frac{1}{p}E_{m_0}'}_{|S}|$ can
separate different $C$. Whenever $|G|$ is an irrational pencil, we
know that $G\equiv rC$ for $r\geq 2$. We pick two different
generic irreducible elements $C_1$ and $C_2$ in $|G|$. Since
$m_0\pi^*(K_X)_{|S}-{E_{m_0}'}_{|S}\sim G\equiv rC$, we see
$\frac{2m_0}{r}\pi^*(K_X)_{|S}\equiv \frac{2}{r}{E_{m_0}'}_{|S}+
C_1+C_2$. Therefore, if we set
$Q_m:=((m-1)\pi^*(K_X)-S-\frac{1}{p}E_{m_0}')_{|S}$, we see that
$$Q_m-\frac{2}{r}{E_{m_0}'}_{|S}-C_1-C_2\equiv
(m-m_0-\frac{2m_0}{r}-1)\pi^*(K_X)_{|S}$$ is nef and big
whenever $m\geq 2m_0+2$. So the Kawamata-Viehweg vanishing
theorem gives $H^1(S,
K_S+\roundup{Q_m-\frac{2}{r}{E_{m_0}'}_{|S}}-C_1-C_2)=0$
and thus the surjective map:
\begin{eqnarray*}
&&H^0(S,K_S+\roundup{Q_m-\frac{2}{r}{E_{m_0}'}_{|S}})\\
&\longrightarrow& H^0(C_1, K_{C_1}+D_1)\oplus
H^0(C_2,K_{C_2}+D_2),
\end{eqnarray*}
where
$D_i:=\roundup{Q_m-\frac{2}{r}{E_{m_0}'}_{|S}-C_1-C_2}_{|C_i}$
for $i=1,2$ with
\begin{eqnarray*}
\deg(D_i)&\geq
&(Q_m-\frac{2}{r}{E_{m_0}'}_{|S}-C_1-C_2)\cdot
C_i\\
&=&(m-m_0-\frac{2m_0}{r}-1)\pi^*(K_X)_{|S}\cdot C_i>0.
\end{eqnarray*}
The Riemann-Roch on $C_i$ says $h^0(C_i,K_{C_i}+D_i)>0$. We thus see
that $|K_S+\roundup{Q_m-\frac{2}{r}{E_{m_0}'}_{|S}}|$ can separate
$C_1$ and $C_2$. To be a bigger linear system,
$|K_S+\roundup{(m-1)\pi^*(K_X)-S-\frac{1}{p}E_{m_0}'}_{|S}|$ can
also separate $C_1$ and $C_2$.

In a word, we have seen that Assumptions \ref{assumptions}(1),
(2), (3) and (4) are all satisfied for $m\geq \max \{3m_0,
2m_0+2\}$.

We begin to verify numerical conditions for $\alpha$. Now if we
take $m \gg 0$ such that $\alpha>1$, Remark \ref{weak} gives
$\xi\geq \frac{2}{2m_0+1}$. Take $m\geq 4m_0+3$. Then
$\alpha=(m-2m_0-1)\xi>2$. Theorem \ref{technical}(I) says $\xi\geq
\frac{5}{4m_0+3}$. Take $m=4m_0+2$. Then
$\alpha=\frac{10m_0+5}{4m_0+3}>2$. So $\varphi_{4m_0+2}$ is
birational by Theorem \ref{technical}(II).
\end{proof}

\begin{prop}\label{Ip}
 Let $X$ be a minimal projective 3-fold of
general type with $P_{m_0}\geq 2$. Suppose that the induced map
from $\varphi_{m_0}$ is of type I$_p$. Then $\varphi_{m}$ is
birational for all $m \ge 4m_0+5$.
\end{prop}

\begin{proof}
We have an induced fibration $f:X'\longrightarrow B$ with
$g(B)=0$. By assumption, $p_g(S)>0$ for a general fiber $S$
of $f$. An established theorem (see Bombieri \cite{Bom},
Reider \cite{Reider}, Catanese-Ciliberto \cite{C-C}, and P.
Francia \cite{Francia} or directly refer to Theorem 3.1 in
the survey article by Ciliberto \cite{Ci}), $|2K_{S_0}|$ is
always base point free whenever $p_g(S)>0$. Thus
$|2\sigma^*(K_{S_0})|$ is also base point free. We take
$G:=2\sigma^*(K_{S_0})$. A generic irreducible element $C$
is a smooth complete curve. On the other hand we have
already known there is a sequence of rational numbers
$\{\beta_n\}$ with $\beta_n\mapsto \frac{p}{m_0+p}$ such
that $\pi^*(K_X)_{|S}\geq
\beta_n\sigma^*(K_{S_0})=\frac{\beta_n}{2}C$. We can take a
$\beta\mapsto \frac{p}{2m_0+2p}\geq \frac{1}{2m_0+2}$ where
we know $p=a_{m_0}\geq 1$. So far, by Proposition
\ref{nonvanishing}(iii) and Lemma \ref{(i)}(iii),
Assumptions \ref{assumptions}(1), (2) and (3) are satisfied
for all $m\geq m_0+m_1$. In particular, $m\geq 3m_0+3$ will
do.

We begin to verify Assumptions \ref{assumptions}(4). Note
that
\begin{eqnarray*}
&&(m-m_0-1)\pi^*(K_X)_{|S}-(m-m_0-1)H_n-2\sigma^*(K_{S_0})\\
&\equiv&
\sigma^*(K_{S_0})+((m-m_0-1)\beta_n-3)\sigma^*(K_{S_0}).
\end{eqnarray*}
When $m\geq  4m_0+5$ and take a very big $n$,
$(m-m_0-1)\beta_n-3>0$. Lemma \ref{>0}(i) says
$$h^0(K_S+\roundup{(m-m_0-1)\pi^*(K_X)_{|S}-
(m-m_0-1)H_n}-2\sigma^*(K_{S_0}))>0.$$ Therefore
$$K_S+\roundup{(m-m_0-1)\pi^*(K_X)}_{|S}\geq
K_S+\roundup{(m-m_0-1)\pi^*(K_X)_{|S}} \geq C.$$ So
Assumptions \ref{assumptions}(4) is satisfied for $m\geq
4m_0+5$. Simultaneously Assumptions \ref{assumptions}(1),
(2) and (3) are also satisfied.

Now let us verify the numerical conditions for $\alpha$. We have
$\deg(K_C)=(K_S+C)\cdot C>4K_{S_0}^2\geq 4$. Because it is even,
$\deg(K_C)\geq 6$. We see $\xi=\pi^*(K_X)_{|S}\cdot C\geq 2\beta_n
K_{S_0}^2\geq 2\beta_n.$ Taking limits we have $\xi\geq
\frac{2}{m_0+1}$.

Take $m\geq 4m_0+5$. Then $\alpha\geq (m-3m_0-3)\xi\geq
\frac{2m_0+4}{m_0+1}>2$. So Theorem \ref{technical}(II)
says that $\varphi_m$ is birational for all $m\geq 4m_0+5$.
\end{proof}



We need the following lemma to prove another result.

\begin{lem}\label{pg=0} Let $S$ be a nonsingular projective surface of
general type. Denote by $\sigma:S\longrightarrow S_0$ the blow down
onto the minimal model $S_0$. Let $|G|$ be the movable part of
$|2K_S|$ and $C$ a generic irreducible element of $|G|$. Assume
that $|G|$ is base point free. Then $\sigma^*(K_{S_0})\cdot C\geq
2$.
\end{lem}
\begin{proof} We say that $S$ is of $(1,0)$ type if
$K_{S_0}^2=1$ and $p_g(S)=0$. When $p_g(S)>0$, then
$|2\sigma^*(K_{S_0})|$ is base point free as stated in the proof of
Proposition \ref{Ip}. So $G=2\sigma^*(K_{S_0})$ and
$\sigma^*(K_{S_0})\cdot C\geq 2$ follows. Thus we only have to
study a surface $S$ with $P_g(S)=0$.

(1) First we study the (1,0) type surface $S$. Because
$P_2(S)=K_{S_0}^2+\chi(\OO_{S_0})=2$, we see $h^0(S,G)=2$. On the
other hand a (1,0) type surface $S$ has $q(S)=0$ (see \cite{Bom}).
Thus $|G|$ is a rational pencil and $C\sim G$. Set
$\overline{C}=\sigma_*(C)$. Clearly $h^0(S_0, \overline{C})\geq
h^0(S,C)$. Thus $\overline{C}$ moves in a family. Because $|C|$ is
the movable part of $|2K_S|$, $|\overline{C}|$ must be the movable
part of $|2K_{S_0}|$ since $P_{2}(S)=P_2(S_0)\geq
h^0(S_0,\overline{C})$. We can write $2K_{S_0}\sim
\overline{C}+\overline{Z}_2$ where $\overline{Z}_2$ is the fixed
part.

If $\overline{C}^2=0$, then $|\overline{C}|$ is base point free and
$\overline{C}$ must be smooth and $\sigma^*(K_{S_0})\cdot
C=K_{S_0}\cdot \overline{C}\geq 2g(\overline{C})-2\geq 2$, noting
that $\overline{C}$ is movable in a family which means
$g(\overline{C})\geq 2$.

If $\overline{C}^2>0$ and $K_{S_0}\cdot \overline{C}=1$, then
$\overline{C}^2\leq \frac{(K_{S_0}\cdot
\overline{C})^2}{K_{S_0}^2}=1.$ Clearly $\overline{C}^2=1$ implies
that $\overline{C}$ is smooth. This already says
$G(C)=g(\overline{C})=2$. The Hodge index theorem says
$\overline{C}\equiv K_{S_0}$. So $Z_2\equiv K_{S_0}$. According to
Bombieri \cite{Bom} or \cite{BPV}, $|3K_{S_0}|$ gives a birational
map. So ${\Phi_{|3K_{S_0}|}}|_{\overline{C}}$ is birational for a
general $\overline{C}$. Because $Z_2\equiv K_{S_0}$ is nef and big,
one has $H^1(S_0, K_{S_0}+Z_2)=0$ by the Kodaira vanishing. So
there is the following surjective map:
$$H^0(S_0, 3K_{S_0})\longrightarrow H^0(\overline{C},
K_{\overline{C}}+{Z_2}|_{\overline{C}}).$$ Since $Z_2$ is effective
and $Z_2\cdot \overline{C}=K_{S_0}^2=1$, ${Z_2}|_{\overline{C}}$ is
a single point. So the Riemann-Roch on $\overline{C}$ gives
$h^0(K_{\overline{C}}+{Z_2}|_{\overline{C}})=2$. Thus the linear
system $|K_{\overline{C}}+{Z_2}|_{\overline{C}}|$ can only give a
finite map onto ${\mathbb P}^1$, a contradiction. Therefore
$\sigma^*(K_{S_0})\cdot C=K_{S_0}\cdot \overline{C}>1$.

(2) Assume $S$ is not of (1,0) type. Then $K_{S_0}^2\geq 2$. We
still keep the same notation as in (1). If $\overline{C}^2=0$, then
$\overline{C}$ must be smooth and $\sigma^*(K_{S_0})\cdot
C=K_{S_0}\cdot \overline{C}\geq 2g(C)-2\geq 2$.

If $\overline{C}^2>0$, then Hodge index theorem says
$$\sigma^*(K_{S_0})\cdot C=K_{S_0}\cdot \overline{C}\geq \sqrt{K_{S_0}^2}
>1.$$

We are done.
\end{proof}

\begin{prop}\label{I3}
 Let $X$ be a minimal projective 3-fold of
general type with $P_{m_0}\geq 3$. Suppose that the induced map
from $\varphi_{m_0}$ is of type I$_n$ or I$_p$. Then $\varphi_{m}$
is birational for all $m \ge 3m_0+6$.
\end{prop}

\begin{proof}
We still take $G$ to be the movable part of $|2K_{S_0}|$. A
different point from previous propositions is that $|G|$ is not
always base point free. But since we have the induced fibration
$f:X'\longrightarrow B$, we can consider the relative bi-canonical
map of $f$, namely the rational map $\Psi: X'\dashrightarrow {\bf
P}$ over B. First we can blow up the indeterminacy of $\Psi$ on
$X'$. Then we can assume, in the birational equivalence sense, that
$\Psi$ is a morphism over $B$. By further modifying $\pi$, we can
even finally {\bf assume} that $\pi$ dominates $\Psi$. With this
assumption, we see that $|G|$ is base point free since $|G|$ gives
the bicanonical morphism for each fiber $S$ of $f$. Under the
assumption $P_{m_0}\geq 3$, we have $p\geq 2$ and we can take a much
better $\beta$. In fact we take a sequence $\{\beta_n\}$ with
 $\beta_n\mapsto \frac{p}{m_0+p}\geq \frac{2}{m_0+2}$ which
should give us better bounds. So far, by Proposition
\ref{nonvanishing}(vi) and Lemma \ref{(i)}(vi), Assumptions
\ref{assumptions}(1), (2) and (3) are satisfied for all $m\geq
m_0+m_1$. In particular, $m\geq 3m_0+4$ will do.

We study Assumptions \ref{assumptions}(4). Note that
\begin{eqnarray*}
&&(m-m_0-1)\pi^*(K_X)_{|S}-(m-m_0-1)H_n-2\sigma^*(K_{S_0})\\
&\equiv&
2\sigma^*(K_{S_0})+((m-m_0-1)\beta_n-4)\sigma^*(K_{S_0}).
\end{eqnarray*}
When $m\geq  3m_0+6$ and take a very big $n$,
$(m-m_0-1)\beta_n-4>0$. Similar to the argument in (iv),
Lemma \ref{>0}(ii) and Remark \ref{separate} tells that
Assumptions \ref{assumptions}(4) is satisfied for $m\geq
3m_0+6$. Thus Assumptions \ref{assumptions}(1), (2) and (3)
are all satisfied for $m\geq 3m_0+6$.

Now we consider $\alpha$. By lemma \ref{pg=0}, we know
$\sigma^*(K_{S_0})\cdot C\geq 2$. Thus $\xi\geq
\beta_n\sigma^*(K_{S_0})\cdot C$. Taking limits one sees
$\xi\geq \frac{4}{m_0+2}$. Recall that we may take
$\beta\mapsto \frac{p}{2m_0+2p}\geq \frac{1}{m_0+2}$.

Take $m\geq 3m_0+6$. Then $\alpha>2$. So Theorem
\ref{technical}(II) says that $\varphi_m$ is birational for
all $m\geq 3m_0+6$.
\end{proof}

\begin{exmp}\label{II20}
In the case of type II, if $m_0 \geq 13$, then one can easily
verify that $\varphi_m$ is birational for $ m \ge 4m_0-6$.
\end{exmp}

\begin{thm}\label{birat}
Let $X$ be a minimal projective 3-fold of general type. Then
$\varphi_m$ is birational for all $m \ge 77$.
\end{thm}

\begin{proof} We proceed as the following steps.

{\bf Step 1.} By Theorem \ref{chi<1} and \cite{JC-H}, it remains
to consider the situation $\chi \geq 1$ and $q(X)=0$.

{\bf Step 2.} If $P_{m_0} \ge 2$ for some $m_0 \le 14$, then
$\varphi_m$ is birational for all $m \ge 76$ by Theorem
\ref{5k+6}.

{\bf Step 3.} If $\chi \ge 2$, there is an $m_0 \le 18$ with
$P_{m_0} \ge 2$ by Theorem \ref{18}. Notice that the induced
fibration can not be of type I$_n$  by Remark \ref{chi(O)>1}. Thus
$\varphi_m$ is birational for all $m \ge 77$ by Propositions
\ref{III}, \ref{II}, \ref{Ip}.

{\bf Step 4.} If $\chi =1$, then by the classification in Section
6, we have $P_{m_0} \ge 2$ for some $m_0 \leq 14$ except the case
I-1. Thus it remains to study the case I-1.

{\bf Step 5.} For the case I-1, computation shows that
$P_{20}(X)\geq P_{20}(B_{min}) =3$. We consider $f$ which is
induced from $\varphi_{20}$. If $f$ is of type III, then we get
birationality for $m \ge 61$ by Proposition \ref{III}. If $f$ is
of type II, then we get birationality for $m \ge 74$ by Example
\ref{II20}. If $f$ is of type $I_3$, then we get birationality for
$m \ge 66$ by Proposition \ref{I3}. This completes the proof.
\end{proof}

\section{\bf Fletcher's conjecture}
First let us recall some notations and definitions in \cite{C-R}.

Let $a_0,\cdots, a_n$ be positive integers. Define
$S=S(a_0,\cdots, a_n)$ to be the graded polynomial ring
$\bC[x_0,\cdots,x_n]$, graded by $\deg(x_i)=a_i$ for all $i$. The
weighted projective space  $\bP(a_0, a_1,\cdots, a_n)$ is defined
by $\text{Proj}(S)$. We only consider the {\it well formed}
weighted projective space  $\bP(a_0,\cdots,a_n)$, i.e.
$$\text{hcf}(a_0,\cdots,\hat{a_i},\cdots,a_n)=1\ \ \text{for each}\
i.$$ It is clear that the usual projective space
$\bP^n=\text{Proj}(T)$ where $T=\bC[y_0,\cdots,y_n]$ and the $y_i$
has weight 1 for all $i$. Consider the inclusion $S\hookrightarrow
T$ given by $x_i\mapsto y_i^{a_i}$ for all $i$. By
\cite[p108, 5.12]{C-R}, the induced quotient map $q: \bP^n\rightarrow \bP(a_0,
\cdots,a_n)$ is a ramified Galois covering with Galois group
$\oplus \mathbb{Z}_{a_i}.$ If $\{Y_i\}$ are the coordinates on
$\bP^n$, the map $q$ is defined as $[Y_0,\cdots,Y_n]\mapsto
[Y_0^{a_0},\cdots, Y_n^{a_n}].$

We consider a hypersurface $X$ of degree $d$ in
$\bP=\bP(a_0,\cdots,a_n)$. Then $X$ has only quotient singularities
(locally $\bC^n$ by a finite group action) and so the dualizing
sheaf $\omega_X\cong i_*\omega_{X_0}=\OO_X(K_X)$ where
$i:X^0\hookrightarrow X$ is the inclusion from the smooth part
$X^0$ and $K_X$ (a canonical Weil divisor) is a $\bQ$-Cartier
divisor. We would like to classify well formed hypersurfaces which
is equivalent to say, by  \cite[p110, 6.10]{C-R}:
\begin{itemize}
\item [(i)] $\text{hcf}(a_0,\cdots, \hat{a_i},\cdots,
\hat{a_j},\cdots, a_n)|d$; \item [(ii)] $\text{hcf}(a_0,\cdots,
\hat{a_i},\cdots, a_n)=1$, for all distinct $i,j$.
\end{itemize}

{}From now on we assume $\dim(X)=3$, $n=4$ and $d-\sum_i a_i=1$,
which are the assumptions of Fletcher. By 6.14, p111 in \cite{C-R},
one has $\omega_X\cong \OO_X(d-\sum a_i)=\OO_X(1)$.

If $X$ has only cyclic terminal quotient singularities, then
$K_X^3={\OO_X(1)}^3$ makes sense. We have
$q^*\OO_{\bP}(1)=\OO_{\bP^n}(1)$. Thus
${\OO_{\bP^n}(1)}^4=\frac{1}{a_0a_1a_2a_3a_4}$. This simply gives
$$K_X^3={\OO_X(1)}^3=d\cdot {\OO_{\bP^n}(1)}^4=
\frac{d}{a_0a_1a_2a_3a_4}.\eqno{(9.1)}$$

Let $A$ be the homogeneous coordinate ring of $X$. By Lemma
7.1 of \cite{C-R}, $$P_m(X)=h^0(X,\OO_X(m))= \dim_k A_m.
\eqno{(9.2)}$$ In particular, when $m < d$, then $P_m(X)=
\dim_k S_m$ clearly, where $A_m, S_m$ denotes the $m$-th
graded part of $A,S$ respectively.

We recall Fletcher's criterion:
\begin{thm}\label{F}  \cite[p145, 14.1]{C-R} Let $X_d$ be a
general hypersurface of degree $d$ in $\bP(a_0,a_1,a_2,a_3,a_4)$
and let $\alpha=d-\sum a_i$. Then $X_d$ is quasismooth with only
isolated terminal quotient singularities and is not a linear cone
if and only if all the following holds:
\smallskip

\noindent(1) For all $i$,
\begin{itemize}
\item[(i)] $d>a_i.$ \item[(ii)] there exists a monomial $x_i^mx_e$
of degree $d$ (that is, there exists $e$ such that $a_i|d-a_e)$.
\item[(iii)] if $a_i\nmid d$, there exists an $m\neq i,e$ such
that $a_i|a_m+\alpha$.
\end{itemize}
\smallskip

\noindent(2) For all distinct i,j, with $h=\text{hcf}(a_i,a_j)$,
then
\begin{itemize}
\item[(i)] $h|d$. \item[(ii)] there exists an $m\neq i,j$ such
that $h|a_m+\alpha$.
\item[(iii)] one of the following holds:\\
either there exists a monomial $x_i^mx_j^n$ of degree $d$, or there
exist monomials $x_i^{n_1}x_j^{m_1}x_{e_1}$ and
$x_i^{n_2}x_j^{m_2}x_{e_2}$ of degree $d$ such that $e_1$, $e_2$
are distinct. \item[(iv)] there exists a monomial of degree $d$
which does not involve $x_i$ or $x_j$.
\end{itemize}
\smallskip

\noindent(3) For all distinct $i,j,k$, $\text{hcf}(a_i,a_j,a_k)=1$.
\end{thm}

Fletcher has given a list of 23 families (See
\cite[p150, 15.1]{C-R}) of quasismooth 3-fold hypersurfaces with only terminal
quotient singularities with $\omega_X\cong\OO_X(1)$ and $\sum
a_i\leq 100$. We would like to prove that Fletcher's list is
complete without constraint to the $\sum a_i$.

The main idea of the proof is that, when $d$ is big, the canonical
volume of $X$ tends to be very small. But on the other hand we have
some effective lower bounds for the volume as proved in Section 3.
This helps us to exclude any possibility.

\begin{proof}[{\bf Proof of Theorem \ref{cod=1}}] Consider
the canonical hyper-surface 3-fold $X\subset
\bP(a_0,a_1,a_2,a_3,a_4)$ with $0<a_0\leq a_1\leq a_2\leq a_3\leq
a_4$. Assume $d=\deg(X)=a_0+a_1+a_2+a_3+a_4+1$. Clearly one sees
$d\geq 5a_0+1$.
We have already seen the
 equality:
 $K_X^3=\frac{d}{\pi}$ where
 $\pi:=a_0a_1a_2a_3a_4$.
Explicitly one has:
$$K^3=\frac{1}{\pi}+\frac{1}{a_1a_2a_3a_4}+\frac{1}{a_0a_2a_3a_4}+\frac{1}{a_0a_1a_3a_4}+
\frac{1}{a_0a_1a_2a_4}+\frac{1}{a_0a_1a_2a_3}.$$
Thus if we have $a_i\geq c_i>0$ for $i=0,1,2,3,4$,
then we get the inequality:
$$K^3\leq
\frac{c_0+c_1+c_2+c_3+c_4+1}{c_0c_1c_2c_3c_4}.\eqno{(9.3)}$$

We will frequently use this inequality in the following discussion.

Furthermore since $d\leq 5a_4+1$, one gets
$$a_4\geq
\frac{1}{5}(d-1)\geq 20.$$

We only have to consider the case $d\geq 101$ according to
Fletcher's work (see 15.1, p150 in \cite{C-R}). And we will show
that there is no such variety with $d \ge 101$. This verifies the
conjecture.

{\bf Claim 1.} We may assume that $a_0 \le 4$. \\
Suppose that $a_0 \ge 5$. Then by Theorem \ref{F}.3, we have either
$a_1 \ge 5, a_2 \ge 6, a_3 \ge 7$ or $(a_0,a_1,a_2,a_3)=(5,5,6,6)$.

If $(a_0,a_1,a_2,a_3)=(5,5,6,6)$, then $a_4 \ge 78$. By $(9.3)$, we
have $K^3 \le \frac{5+5+6+6+78+1}{5 \cdot 5 \cdot 6 \cdot 6 \cdot
78} < \frac{1}{420}$. By Theorem \ref{volume_chi=1}, this is a
contradiction.

If $a_1 \ge 5, a_2 \ge 6, a_3 \ge 7$, then since $a_4\geq
\frac{1}{5}(d-1)\geq 20$, it follows that  $K^3\leq
\frac{5+5+6+7+20+1}{5\cdot 5\cdot 6\cdot 7\cdot 20}< \frac{1}{477}$
by $(9.3)$, this is a contradiction as well. This proves the claim.
\qed

{\bf Claim 2.} We may assume that $a_0+a_1 \le 11$. \\
If $a_0=1$ and $a_1 \ge 11$, then $a_2 \ge 11, a_3 \ge 12$. We have
$K^3 \le \frac{56}{1 \cdot 11 \cdot 11 \cdot 12 \cdot 20} <
\frac{1}{420}$.

If $a_0=2$ and $a_1 \ge 8$, then $a_2 \ge 9, a_3 \ge 9$. We have
$K^3 \le \frac{49}{2 \cdot 8 \cdot 9 \cdot 9 \cdot 20} <
\frac{1}{420}$.

If $a_0=3$ and $a_1 \ge 7$, then $a_2 \ge 7, a_3 \ge 8$. We have
$K^3 \le \frac{46}{3 \cdot 7 \cdot 7 \cdot 8 \cdot 20} <
\frac{1}{420}$.

If $a_0=4$ and $a_1 \ge 6$, then  $a_2 \ge 7, a_3 \ge 7$. We have
$K^3 \le \frac{45}{4 \cdot 6 \cdot 7 \cdot 7 \cdot 20} <
\frac{1}{420}$. All of these can not happen by Theorem
\ref{volume_chi=1}. This proves the claim. \qed

 Because $a_4$ is the biggest one
among $a_i$, Theorem \ref{F}.1.(iii) yields that either $a_4=a_m+1$
or  $a_4|d$. In the second case, if $d=a_0+a_1+a_2+a_3+a_4+1=na_4$
for an integer $n>1$ then it's clear that  $n\leq 5$.

{\bf Case 1.} $d=5a_4$, then $a_4 \ge 21$ and  $\sum_{i=0}^3
(a_4-a_i)=1$. It follows that $a_0 \ge a_4-1  \ge 20$. Which is
absurd.

{\bf Case 2.} $d=4a_4$, then $a_4 \ge 26$ .  Since $a_0 \le 4$, we
have $\sum_{i=1}^3 (a_4-a_i) \le 5$. Hence $a_1 \ge a_4-5 \ge 21$,
which is absurd.

{\bf Case 3.} $d=3a_4$, then $a_4 \ge 34$. Since $a_0+a_1 \le 11$,
we have $(a_4-a_3)+(a_4-a_2) \le 12$. Thus $a_2 \ge a_4-12 \ge 22$.

 If $a_0\geq 2$, then $a_1\geq 2$, we get $K^3\leq
\frac{2+2+22+22+34+1}{4\cdot 22^2\cdot 34}<\frac{1}{420}$, a
contradiction.

If $a_0=1$ and $a_1\geq 3$, then similarly we have $K^3\leq
\frac{1+3+44+34+1}{1 \cdot 3\cdot 22^2\cdot 34}<\frac{1}{420}$, a
contradiction.

So now  consider the case that $(a_0,a_1)=(1,1) $ or $(1,2)$.
Notice that $P_2(X) \ge 2$ by $(9.2)$. Hence by Theorem
\ref{volume}.iii, we have $K^3 \ge \frac{5}{96}$.  On the other
hand,  by $(9.3)$, we get $K^3 \leq \frac{1+2+44+34+1}{1\cdot
2\cdot 22^2\cdot 34} < \frac{5}{96}$, or $K^3 \leq
\frac{1+1+44+34+1}{1\cdot 1\cdot 22^2\cdot 34}< \frac{5}{96}$.
Hence both cases lead to  a contradiction.

{\bf Case 4}. We consider the case $n=2$. Then we have the
relation: $a_4\geq 51$ and $a_0+a_1+a_2+a_3+1=a_4\geq 51$. Since
$a_0+a_1 \le 11$ by Claim 2, we have $2a_3 \ge a_2+a_3 \ge 39$ and
thus $a_3 \ge 20$.

If $a_0=4$, then we must have $a_1\geq 4$ and $a_2\geq 5$. Thus we
get $K^3\leq \frac{8+5+20+51+1}{4^2\cdot 5\cdot 20\cdot
51}<\frac{1}{420}$, a contradiction.

If $a_0=3$, then we must have $a_1 \ge 3, a_2 \ge 4$.  Thus we get
$K^3\leq \frac{6+4+20+51+1}{3^2\cdot 4\cdot 20\cdot 51}
<\frac{1}{420}$, a contradiction.

If $a_0=2$ and $a_1 \ge 4$, then $a_2 \ge 5$ and we have $K^3\leq
\frac{2+4+6+20+51+1}{2\cdot 4 \cdot 5 \cdot 20\cdot
51}<\frac{1}{420}$, a contradiction.  If $a_1=3$ and $a_2\geq 5$,
then we get $K^3\leq \frac{2+3+5+20+51+1}{2\cdot 3\cdot 5\cdot
20\cdot 51}< \frac{1}{320}$. On the other hand, since  $P_6(X)\geq
2$, Theorem \ref{volume}.iii gives $K^3\geq \frac{1}{311}$, a
contradiction; If $a_1=3$ and $a_2=4$, then $a_3\geq 41$. We get
$K^3\leq \frac{1}{492}$, a contradiction; If $a_1=3$ and $a_2=3$,
then $a_3\geq 42$. We get $K^3\leq \frac{1}{378}$. Since $P_3(X)\geq
2$, Theorem \ref{volume} (iii) says $K^3> \frac{1}{53}$, a
contradiction. If $a_1=2$, then  we can get $K^3 \le
\frac{2+2+2+20+51+1}{2 \cdot 2 \cdot 3 \cdot 20 \cdot 51} <
\frac{1}{120}$. Since $P_2(X)\geq 2$, Theorem \ref{volume}.(iii)
gives $K^3> \frac{1}{20}$, a contradiction.

Finally we consider the case that $a_0=1$. If $a_1\geq 6$, then
$a_2\geq 6$. Thus we can get $K^3\le \frac{1+6+6+20+51+1}{1 \cdot 6
\cdot 6 \cdot 20 \cdot 51} < \frac{1}{420}$, a contradiction. We
thus consider the case that  $a_1 \le 6$. Notice that on one hand,
 we have $P_{a_1}(X) \ge 2$, hence $K^3 \ge \frac{11}{12a_1
(a_1+1)^2}$ by Theorem \ref{volume}.(iii). On the other hand, $K^3
\le \frac{72+2a_1}{1 \cdot a_1 \cdot a_1 \cdot 20 \cdot 51}$.
However, it's easy to verify that for  $a_1 \le 5$,
$\frac{11}{12a_1 (a_1+1)^2}> \frac{72+2a_1}{1020 a_1^2 }$. This
leads to a contradiction.

It thus remain to consider the case that $a_4=a_m+1$ for some $m$.

{\bf Case 5}. If $a_4=a_3+1$, and $a_3|d$. Say $d=na_3$. Note that
$d= a_0+a_1+a_2+2a_3+2 =na_3 \ge 101$. If $n \ge 4$, then one get a
contradiction easily for $a_0+a_1 \le 11 < a_3-2$. If $n=3$, then
$a_2 \ge a_3-12 \ge 22$ since $a_3 \ge 34$. Thus one can easily
check that $K^3 < \frac{1}{420}$ unless $(a_0,a_1)=(1,1),(1,2)$. In
both case, one has $K^3 < \frac{5}{96}$. However $P_2 \ge 2$ gives
the required contradiction.

{\bf Case 6}. If $a_4=a_3+1$ and $a_3=a_m+1$, then $a_3=a_2+1$
since $a_1 \le 11$ by Claim 2. Now $a_0+a_1+3a_2+3+1 \ge 101$ gives
$ a_2 \ge 29$. If $a_1 \ge 2$, then one can easily check that $K^3
\le \frac{1+2+29+30+31+1}{1 \cdot 2 \cdot 29 \cdot 30 \cdot 31}
<\frac{1}{420}$.

Hence we only need to consider $a_1=a_0=1$. It's easy to see that
$K^3 < \frac{1}{3}$.  Now $p_g \ge 2$, thus $K^3 \ge \frac{1}{3}$
by \cite{MathAnn}, this is the required contradiction.

{\bf Case 7}. Finally, if $a_4=a_m+1$ for some $m \le 2$, then
clearly $a_4=a_2+1$ and it follows that $a_3=a_2$ or $a_3=a_2+1$.
The similar argument as in Case 6 gives a contradiction.

This completes the proof.
\end{proof}


\end{document}